\documentclass[12pt]{amsart}
\usepackage{amssymb,amsmath,amsthm,bm,mathtools}
 \usepackage[most]{tcolorbox}
\usepackage[colorinlistoftodos,prependcaption,textsize=tiny]{todonotes}
\definecolor{darkblue}{RGB}{0,0,160}
\usepackage[lining]{libertine}   
\usepackage{cabin}
\usepackage[libertine]{newtxmath}

\usepackage[colorlinks=true,allcolors=darkblue]{hyperref}
\usepackage[backrefs,msc-links,nobysame,initials,non-compressed-cites]{amsrefs}

\usepackage[T1]{fontenc}

\usepackage{color}

\def\e{{\rm e}}
\def\cic{\mathbf}
\def\eps{\varepsilon}

\def\d{{\rm d}}

\def\R {\mathbb{R}}

\def\Dil {{\mathrm{Dil}}}

\def\af{\mathfrak{a}}
\def\vaf{\vec{\mathfrak a}}

\def\Dil {{\mathsf{Dil}}}
\def\Sy {{\mathsf{Sy}}}

\def \l {\langle}
\def \r {\rangle}

\def \and{\qquad\text{and}\qquad}

\newcommand{\supp}{\mathrm{supp}\,}

\newcounter{thms}

\newcounter{other}
\numberwithin{other}{section}
\newtheorem{proposition}[other]{Proposition}
\newtheorem{theorem}[thms]{Theorem}
\newtheorem*{theorem*}{Theorem}
\newtheorem*{proposition*}{Proposition}
\newtheorem{corollary}{Corollary}
\newtheorem*{corollary*}{Corollary}
\numberwithin{corollary}{thms}
\newtheorem{lemma}[other]{Lemma}
\theoremstyle{definition}
\newtheorem{remark}[other]{Remark}
\newtheorem{definition}[other]{Definition}

\def\vint_#1{\mathchoice%
      {\mathop{\kern 0.2em\vrule width 0.6em height 0.69678ex depth -0.58065ex
              \kern -0.8em \intop}\nolimits_{\kern -0.4em#1}}%
      {\mathop{\kern 0.1em\vrule width 0.5em height 0.69678ex depth -0.60387ex
              \kern -0.6em \intop}\nolimits_{#1}}%
      {\mathop{\kern 0.1em\vrule width 0.5em height 0.69678ex depth -0.60387ex
              \kern -0.6em \intop}\nolimits_{#1}}%
      {\mathop{\kern 0.1em\vrule width 0.5em height 0.69678ex depth -0.60387ex
              \kern -0.6em \intop}\nolimits_{#1}}}
\def\vintslides_#1{\mathchoice%
      {\mathop{\kern 0.1em\vrule width 0.5em height 0.697ex depth -0.581ex
              \kern -0.6em \intop}\nolimits_{\kern -0.4em#1}}%
      {\mathop{\kern 0.1em\vrule width 0.3em height 0.697ex depth -0.604ex
              \kern -0.4em \intop}\nolimits_{#1}}%
      {\mathop{\kern 0.1em\vrule width 0.3em height 0.697ex depth -0.604ex
              \kern -0.4em \intop}\nolimits_{#1}}%
      {\mathop{\kern 0.1em\vrule width 0.3em height 0.697ex depth -0.604ex
              \kern -0.4em \intop}\nolimits_{#1}}}

\newcommand{\aveint}[2]{\mathchoice%
      {\mathop{\kern 0.2em\vrule width 0.6em height 0.69678ex depth -0.58065ex
              \kern -0.8em \intop}\nolimits_{\kern -0.45em#1}^{#2}}%
      {\mathop{\kern 0.1em\vrule width 0.5em height 0.69678ex depth -0.60387ex
              \kern -0.6em \intop}\nolimits_{#1}^{#2}}%
      {\mathop{\kern 0.1em\vrule width 0.5em height 0.69678ex depth -0.60387ex
              \kern -0.6em \intop}\nolimits_{#1}^{#2}}%
      {\mathop{\kern 0.1em\vrule width 0.5em height 0.69678ex depth -0.60387ex
              \kern -0.6em \intop}\nolimits_{#1}^{#2}}}

\renewcommand*{\cdots}{%
  \mathinner{{\cdotp}{\cdotp}{\cdotp}}%
}

\numberwithin{equation}{section}

\title[Wavelet Representation of Singular Integral Operators]{Wavelet Representation of Singular Integral Operators}

\author[F. Di Plinio]{Francesco Di Plinio}  
\thanks{F. Di Plinio was partially supported by the National Science Foundation under the grant NSF-DMS-2000510.}

\author[B. D. Wick]{Brett D. Wick}
\thanks{B. D. Wick's research partially supported in part by NSF grant NSF-DMS-1800057 as well as ARC DP190100970.}

\author[T. Williams]{Tyler Williams} 

\address{Department of Mathematics, Washington University in Saint Louis\\ \newline \indent 1 Brookings Drive, Saint Louis, Mo 63130, USA}
\email{\textnormal{francesco.diplinio@wustl.edu, bwick@wustl.edu, tgwilliams@wustl.edu}}


\begin{document}

\begin{abstract} 
This article  develops  a novel approach to the representation of singular integral operators of Calder\'on-Zygmund type in terms of continuous model operators, in both the classical and the bi-parametric setting. The representation is realized as a finite sum of  averages of wavelet projections of either cancellative or noncancellative type, which are themselves Calder\'on-Zygmund operators.   Both properties are out of reach for the established dyadic-probabilistic technique. Unlike their dyadic counterparts, our representation reflects the additional kernel smoothness of the operator being  analyzed.

Our representation formulas lead naturally to a new family of   $T(1)$ theorems on weighted Sobolev spaces whose smoothness index is naturally related to kernel smoothness. In the one parameter case, we obtain the  Sobolev space analogue of the  $A_2$ theorem; that is,  sharp dependence of the Sobolev norm of $T$  on the weight characteristic is obtained in the full range of exponents. In the bi-parametric setting, where local average sparse domination is not generally available,
 we obtain quantitative  $A_p$ estimates which are best known, and sharp  in the range $\max\{p,p'\}\geq 3$ for the fully cancellative case.
 \end{abstract}

\maketitle
\section{Introduction}

The class of  Calder\'on-Zygmund singular integrals may succinctly be described as consisting of those  linear   operators bounded on $L^{p}(\R^{d}) $ for some $1<p<\infty$, whose  Schwartz kernel off the diagonal in $\R^{d}\times \R^{d}$ satisfies the same size and smoothness properties  enjoyed by the kernels of the Hilbert $(d=1)$   or Riesz transforms. These properties, and thus the Calder\'on-Zygmund class, are quantitatively preserved by  translations and dilations   of $\R^d$.
  
Calder\'on-Zygmund singular integrals play a pivotal role in  regularity theory of elliptic PDEs, operator theory, and differentiation theory. In particular, their quantified and sharp weighted norms  have been the keystone in the solution to two important problems in the theory of distortion by quasiconformal mappings. Sharp weighted bounds for the Ahlfors-Beurling transform have been used by Petermichl and Volberg \cite{PV} to prove injectivity at the critical exponent of the solution operator to the Beltrami equation, completing the scheme devised in \cite{AIS} by Astala, Iwaniec and Saksman. A sharp estimate of the Beurling transform on suitable fractal-dimensional weights has been employed by Lacey, Sawyer and Uriarte-Tuero  \cite{LSU} to prove Astal{a}'s conjecture on distortion of Hausdorff measures under quasiconformal maps.

Beginning  with Figiel's  $T(1)$ theorem \cite{Fig}, through the seminal work of Nazarov, Treil and Volberg \cite{NTVJ} on two-weight inequalities for non-homogeneous Haar shifts, the study of singular integrals and their sharp weighted bounds  has relied on  dyadic-probabilistic techniques. An important change in perspective within this line of investigation was brought by    the work of Petermichl \cite{PetAJM} where the Hilbert transform was \emph{represented}, instead of \emph{estimated}, as an average of  dyadic shift operators.   This idea was extended to the Beurling transform, Riesz transform, and other nice convolution type Calder\'on-Zygmund operators in \cites{MR1964822, MR1992955, MR2680056}.  A forceful augmentation of this strategy is  Hyt\"onen's proof of the $A_2$ conjecture \cite{Hyt2010},  which relied on the representation of a Calder\'on-Zygmund operator as a probabilistic average of shifted dyadic operators, the simplest  of which is the dyadic martingale transform. Hyt\"onen's representation theorem has since been extended to bi-parameter Calder\'on-Zygmund operators by Martikainen in \cite{MK1} and to the multi-parameter setting by Y.\ Ou \cite{OuMp}, resulting in $T(1)$ type theorems in two and higher parameters. 

While the significance of dyadic representation theorems cannot be overstated, dyadic-probabilistic realizations of Calder\'on-Zygmund operators suffer from certain intrinsic drawbacks, originating from the  discrete nature of the Haar basis employed. A first one is that the representation formula contains dyadic shifts of arbitrarily large \emph{complexity} parameter: complexity may be described, roughly, as the width of the band of the matrix associated to the dyadic shift in the Haar basis.  The representation formula itself involves a rather delicate averaging procedure over shifted dyadic grids, and, for each shifting parameter, a countable collection of shifts of unbounded complexity. Explicitly computing the dyadic components of a generic singular integral operator is thus not feasible. 

Secondly, the dyadic representation formula, due to the roughness of the basis, fails to detect additional kernel smoothness. The latter is often relevant for the behavior of a singular integral on smoother spaces, of e.g.\ Sobolev or Besov type. We note that the recent article \cite{HytLap} partially addresses this question within the framework of dyadic representation theorems, replacing the Haar basis with a smoother wavelet but keeping the same structure involving shifted grids. The comparison with \cite{HytLap} is elaborated upon in Remark \ref{r:hytlap}.
 
Bi-parameter singular integrals on $  \R^{d_1}\times \R^{d_2}$  may   be informally defined as  elements of the closed convex hull of the set of tensor products $T_1\otimes T_2$, where each $T_j$ is a $\R^{d_j}$-singular integral operators as above.  This class arises naturally in connection with the theory of bi-harmonic functions \cite{FR1,FS} for instance in the weak factorization of functions in the Hardy space on the bi-disk \cite{LF}. The $L^p$ and mixed norm estimates for their  multilinear analogues are at the root of partial fractional Leibniz rules \cites{MPTT1,MPTT06}, namely, anisotropic variants of the bilinear estimates   popularized e.g.\ by  Kato-Ponce \cite{KP} in connection with the Navier-Stokes,    Schr\"odinger and   KdV equations.
    
One specific drawback of   dyadic representations in the  bi-parameter context, see \cite{MK1,OuMp} for instance, is that   they do not reduce $L^p$ and weighted estimates for the analyzed operator to a single model operator whose weighted theory is significantly simpler. In the one parameter case, this can be verified directly for shift operators as in \cite{MR2657437} and can also be done by means of sparse operators as in \cite{CDPOMRL}.  In higher parameters, the approach via direct verification is challenging \cites{HPW,LMV} and domination by local average sparse operators are generally not available as the counterexample of \cite{BCOR} shows. Thus, for instance, one cannot expect precise information on the dependence of $\|T\|_{L^p(w)}$ on the corresponding relevant weight characteristic. 

This paper sets forth a new technique for analyzing singular integral operators based on rank 1 wavelet projections
\[
f\mapsto s^{d}\langle f, \varphi_{(x,s)}\rangle \varphi_{(x,s)}, \qquad  (x,s)\in \R^{d}\times (0,\infty)
\]
where $\varphi_{(x,s)}$ are $L^1$-normalized wavelets living at scale $s$ near the point $x$.  The method used is to instead take a weighted average of these wavelet projections with respect to the operator wavelet coefficients $\langle T\varphi_{(x,s)}, \varphi_{(y,t)}\rangle$, with $(y,t)\in \R^{d}\times (0,\infty)$. The simplest version of this principle is the resolution of the identity operator \eqref{e:CRF} below, widely known as the Calder\'on reproducing formula.  Certainly, the Calder\'on reproducing formula \eqref{e:CRF} and wavelet coefficients  have been used countless times in the proof of $T(1)$ type estimates, beginning with the works of  David and Journ\'e \cite{DJ} and Journ\'e in the bi-parametric setting \cite{Jo}. Our approach takes these seminal ideas one step further, in that we aim for equalities, rather than inequalities, and employ the wavelet coefficients of $T$ in a wavelet averaging procedure instead of estimating them, see Lemma \ref{e:averagproc}, essentially turning the original wavelet basis into another wavelet family adapted to the operator being analyzed.  Our approach takes advantage of the fact that a Calder\'on-Zygmund operator applied to a smooth wavelet basis with compact support yields again a collection of  wavelets, though possibly rougher and with  smeared out support.

When analyzing one parameter Calder\'on-Zygmund operators, this results in a representation formula involving a single, \emph{complexity zero} cancellative operator, a single paraproduct and a single adjoint paraproduct, all of which are Calder\'on-Zygmund operators themselves,  in contrast to the dyadic expansion of e.g. \cite{Hyt2010}, involving probabilistic averaging of countably many dyadic shifts.  The following  result is a loosely stated excerpt of  Theorem \ref{t:T1}, Section \ref{s:1p}, which is referred to for a rigorous and detailed statement, in particular concerning the formal definition of $T$. 
\begin{theorem*} \label{t:T1k1}  
 Let $T$ be a linear operator on $\R^d$, satisfying the weak boundedness testing condition, the standard $\delta$-kernel estimates  
for some $\delta>0$ and with $T1, T^\star 1 \in \mathrm{BMO}(\R^d)$. Let $0<\eps<\delta$.
Then there exists a family of $L^1$-adapted, $\eps$-smooth and $(d+\eps)$-decaying cancellative wavelets $\{\upsilon_{(y,t)}: y\in \R^d, t>0  \}, $
   such that  
\begin{equation*}
\begin{split} 
 Tf(x)  &= \displaystyle\int_{\R^{d}\times (0,\infty)}  \langle f,  \varphi_{(y,t)}\rangle   \upsilon_{(y,t)} (x)    \, \frac{\d y \d t}{t}+ \Pi_{T 1}  f(x)+ \Pi^\star_{T^\star 1}  f(x), \qquad x\in \R^d
\end{split} 
\end{equation*}
where $\varphi_{(y,t)}$ is the $(y,t)$-rescaling of a smooth mother cancellative wavelet with compact support,  $\Pi_{b}$ are explicitly constructed paraproducts of the form in Definition \ref{c:parf}.  
\end{theorem*} 
If additional smoothness of the kernel is available, the cancellative portion will be made of smoother and faster decaying wavelets, provided  that additional paraproducts are subtracted. In turn, this refinement   may be  employed to obtain sharp weighted $T(1)$ bounds on Sobolev spaces. A sample is given in the theorem below, which is a simplified version of Corollary \ref{cor:t1}. The case $k=0$ has no restriction on paraproducts, thus it is yet another proof of the sharp weighted bound for Calder\'on-Zygmund operators by Hyt\"onen \cite{Hyt2010}. For the case $k>0$, the restriction  $T(x^{\gamma})=0$ for $0\leq |\gamma|<k$ is necessary, see Remark \ref{r:nec} where the relation with  \cite[Theorem 1.1]{PXT} is also expounded. 
\begin{theorem*} Let $T$ be a linear operator on $\R^d$, satisfying the weak boundedness testing condition of Definition \ref{c:weak},  whose kernel has $k$-th H\"older continuous  derivatives in the sense of  Definition \ref{c:ker}, and whose $\kappa$-th order paraproducts are in $\mathrm{BMO}(\R^d) $ for $0\leq \kappa \leq k$, according to Definition \ref{c:par}. If $k>0$, assume in addition that $T(x^{\gamma})=0$   for all multi-indices $\gamma$ on $\R^d$ with   $0\leq |\gamma|<k$. Then
$T$ has the  sharp weighted bound on the homogeneous $L^2$-Sobolev space  
\[
\| T f\|_{\dot W^{k,2}(\R^d; w)} \lesssim [w]_{A_2}\|  f\|_{\dot W^{k,2}(\R^d; w)}. 
\]
\end{theorem*}
Corollary \ref{cor:t1} may be seen as the $A_2$-theorem on Sobolev spaces, in analogy with the celebrated $A_2$ theorem of Hyt\"onen \cite{Hyt2010}. Past results in the unweighted setting, and their relation with Corollary \ref{cor:t1}, are recalled in Remark \ref{r:smotest}. Several unweighted $T(1)$ theorems of this type have been developed and used in connection with regularity of the Beltrami solution operator and smoothness of the Beurling transform on $C^{k,\alpha}$-type domains, see e.g.\ \cites{CMO,CXT,Prats2017,Prats2019}. In analogy with the application of the Petermichl-Volberg theorem \cite{PV} to the scheme of  \cite{AIS} we expect that suitable versions of Corollary \ref{cor:t1} can be employed to simplify and sharpen the  higher order Beltrami regularity results of e.g. \cite{Prats2019}. This will be carried out in future work.

The proof techniques may be naturally transported to the bi-parametric dilation setting.  We again paraphrase the main result and point the reader to Section \ref{s6}, Theorem \ref{t:T12p}, for the precise statement.

\begin{theorem*} \label{t:T12pk=0} 
  Let $T$ be a linear operator satisfying the hypotheses of a bi-parameter $\delta$-Calder\'on-Zygmund operator as in Section \ref{s6}.  Let $0 <\eps<\delta\leq 1$. Then there exists a family of $L^1$-adapted, $\eps$-smooth and $(d_j+\eps)$-decaying in the $j$-th parameter, product cancellative  wavelets \[\{\upsilon_{((y_1,t_1), (y_2,t_2))}: y_j\in \R^{d_j}, t_j>0, j=1,2  \},\] such that  for $(x_1,x_2)\in\mathbb{R}^{d_1}\times\mathbb{R}^{d_2},$
\begin{equation*}
\begin{split} Tf(x_1,x_2) = &  \int\displaylimits_{\mathbb{R}^{d_1}\times(0,\infty)}\int\displaylimits_{\mathbb{R}^{d_2}\times(0,\infty)}     \langle f,  \varphi_{(y_1,t_1)} \otimes \varphi_{(y_2,t_2)} \rangle  \upsilon_{((y_1,t_1), (y_2,t_2)) }(x_1,x_2)\, \frac{\d y_1 \d y_2 \d t_1 \d t_2}{t_1 t_2}\\ 
& + \textnormal{four paraproduct terms}
 +  \textnormal{four partial paraproduct terms.}\quad 
\end{split}
\end{equation*}
\end{theorem*}
Notice that unlike the one parameter results recalled in Remark \ref{r:smotest}, no smooth $T(1)$ type theorems have appeared in the literature before, even in the unweighted case.
This result may be contrasted with the bi-parameter dyadic representation theorem of Martikainen \cite{MK1}. The assumptions on $T$ are of the same nature as the ones appearing in \cite{MK1}, namely weak boundedness, full and partial kernel estimates, paraproducts in  product BMO. However, we drop the diagonal BMO conditions appearing in \cite{MK1} which are subsumed by a combination of the other assumptions. In addition to the simpler, and more computationally feasible nature of the continuous formula, the    model operators we obtain have a much simpler weighted theory, which allows for quantitative, and sharp in certain cases, weighted norm inequalities for $T$.
\begin{theorem*}   \label{cor:t12Intro}  Let $T$ be a  $(k_1,k_2)$-smooth bi-parameter Calder\'on-Zygmund operator, see Section \ref{s6}. If  $(k_1,k_2)\neq (0,0)$ assume in addition the paraproduct condition \eqref{e:cort1b}. For all $1<p<\infty$  product $ A_p$-weights  $w$ on $\R^{\mathbf d}\coloneqq \mathbb R^{d_1}\times\mathbb{R}^{d_2}$ there holds
\begin{equation*}
\|\nabla_{x_1}^{k_1}\nabla_{x_2}^{k_2} T f\|_{ L^p(\R^{\mathbf d}; w)} \lesssim_{k,\delta}  [w]_{A_p}^{\max\left\{3,\frac{2p}{p-1}\right\}}\|\nabla_{x_1}^{k_1}\nabla_{x_2}^{k_2} f\|_{ L^p(\R^{\mathbf d}; w)}
\end{equation*}
  If $T$ is  fully cancellative, the improved estimate
\[
\|\nabla_{x_1}^{k_1}\nabla_{x_2}^{k_2} T f\|_{ L^p(\R^{\mathbf d}; w)} \lesssim_{k,\delta}  [w]_{A_p}^{\theta(p)}\|\nabla_{x_1}^{k_1}\nabla_{x_2}^{k_2}  f\|_{ L^p(\R^{\mathbf d}; w)} , \qquad \theta(p)=\begin{cases} \frac{2}{p-1} & 1<p\leq 
\frac32 \\   \textrm{\emph{see} \eqref{e:sharp} 
} & \frac32<p < 3 \\ 2 & p \geq 3 \\ \end{cases}
\]
is available. The above estimate is sharp when $\max\{p,p'\} \geq 3 $. 
\end{theorem*}
The above result,   precisely stated   in Corollary  \ref{cor:t12}, generalizes and quantifies R.\ Fefferman's qualitative weighted bounds for bi-parameter Journ\'e-type operators \cite{RF1}. While Martikainen's work \cite{MK1}  did not contain weighted $T(1)$-type implications,  a simplified proof of Fefferman's result was recently obtained in \cite{HPW}  relying on the representation from \cite{MK1}.  Some quantitative estimates, weaker\footnote{The exponent  of the $A_2$ constant obtained in \cite[Corollary 3.2]{BPS} is 10, in contrast with power 3 obtained in \eqref{e:sharp}. Therein, it is claimed that tracking the constants in the argument of \cite{HPW} yields power 8.} than the ones of Corollary \ref{cor:t12}, have been obtained in \cite{BPS} by a shifted square function form-type domination for cancellative Journ\'e operators, also relying on \cite{MK1} within the proof.  At  present, it does not seem possible to match the quantification obtained in Corollary \ref{cor:t12}  using dyadic representation theorems in the vein of \cite{MK1,OuMp}.  Part of the challenge with this is that one parameter proofs of the quantitative results typically rely upon some variant of stopping time (sparse operators, weak-type (1,1), or Bellman functions) as a key ingredient and not easily adaptable to the bi-parameter setting.  Our analysis based on square function methods is able to circumvent this issue at least for $\max\{p,p'\}\geq3$.
 
Qualitative estimates for bi-linear, bi-parameter paraproduct free singular integrals have been obtained very recently in \cite{LMV}. We expect that multilinear variants of our result can   quantify and extend the scope of the weighted inequalities of \cite{LMV}.  The method used in this paper has an additional benefit over dyadic representation.   The dyadic representation theorems rely heavily upon the Haar basis,  making  it difficult to work in the settings of other Calder\'on-Zygmund operators that respect different dilation structures,  such as  Zygmund-type dilations \cite{FP,MRSI,MRSII}.  Our method is easily adaptable towards representation and weighted $T(1)$ theorems for  more general dilation structures; this will be pursued in future work. 

\subsection*{Structure}  Section \ref{s:2} contains the definition of the one parameter  wavelet classes  and introduces  the related intrinsic square function.  Section \ref{s3} provides technical lemmas in the one parameter setting, describing the  averaging procedure of the wavelet basis and setting up an auxiliary Alpert basis to handle higher degree paraproducts.  Section \ref{s:1p} gives the statement and   proof of the one parameter smooth representation theorem and weighted Sobolev $T(1)$ corollary, see Theorem \ref{t:T1} and Corollary \ref{cor:t1}.  Sections \ref{s:4} through  \ref{s7} set forth the statement and proof of the bi-parameter versions,  Theorem \ref{t:T12p} and Corollary \ref{cor:t12}.  The concluding Section \ref{s8} is devoted to establishing weighted norm inequalities for the model operators appearing in the bi-parameter representation, via new weighted bounds for intrinsic operators such as the bi-parameter intrinsic square function.
  \subsection*{Notation} The symbol $C=C(a_1,\ldots, a_n)$ and the constant implied by almost inequality and comparability symbols  $\lesssim_{a_1,\ldots,a_n}, \sim_{a_1,\ldots,a_n}$ are meant to depend on the parameters $a_1,\ldots, a_n$ only  and may vary at each occurrence without explicit mention.   It is convenient to employ the Japanese bracket symbol 
\[
\l x \r = \max\{1,|x|\}
, \qquad x\in \R^d.\]  
 The Fourier transform of $f\in \mathcal S(\R^d) $ is normalized as
\[ 
\widehat f(\xi) = \frac{1}{(\sqrt{2\pi})^d} \int_{\R^d} f(x)  \e^{-ix\cdot\xi } \, \d x, \qquad \xi \in \R^d.
 \] 
 
\subsection*{Acknowledgments} The authors are deeply thankful to Alexander Barron, Henri Marti\-kai\-nen and Yumeng Ou for illuminating conversations on bi-parameter $T(1)$ theorems and weigh\-ted norm inequalities. The authors gratefully acknowledge Walton Green for his insightful reading and suggestions which led to significant improvements to the clarity of the statements and exposition.
   
\section{Wavelet Classes and the Intrinsic Square Function} \label{s:2} In this section  we introduce  the normalized classes of wavelets with limited decay that appear in our wavelet representation theorem for   Calder\'on-Zygmund forms on $\R^d$.  
 
\subsection{Analysis in the symmetry parameter space}Our analysis of these forms is based on a symmetry parameter space description. In the classical, single-parameter setting of Section \ref{s:1p}, the parameter space and its associated  natural measure $\mu$  are
\[
z=(x,s)\in Z^{d}\coloneqq
\R^{d} \times (0,\infty), \qquad \int_{Z^d} f(z)\, \d \mu(z)\coloneqq \int\displaylimits_{\R^d \times (0,\infty)} f(x,s)\, \frac{\d x\d s}{s}. 
\]
Points of $Z^d$ conveniently parametrize open balls in $\R^d$ by
\[
\mathsf{B}_z=\{y\in \R^d: |y-x|<s\},\qquad z=(x,s)\in Z^d. 
\]
When two points, or families of points, of $Z^d$ appear in the same statement and the context allows for it, the notation $z=(x,s)$ and the corresponding Greek version $\zeta=(\xi,\sigma)$ are used; for instance, see \eqref{e:Delta1} below. 
For each $\zeta=(\xi,\sigma)\in Z^d$ it is convenient to refer to the following partition of $Z^d$:
\begin{equation}
\label{e:partz}
\begin{split}
&Z^d_+(\zeta)\coloneqq \{z=(x,s)\in Z^d:s\geq \sigma\}=F_+(\zeta) \sqcup S(\zeta)    \sqcup A(\zeta), \\
& 
  F_+(\zeta)\coloneqq \big\{z=(x,s)\in Z^d_+(\zeta) :    |x-\xi|>3s \big\},
\\ 
 &S(\zeta)\coloneqq \big\{z=(x,s)\in Z^d_+(\zeta)  :   s\in[\sigma, 3\sigma], |x-\xi|<3 s \big\},
 \\  & A(\zeta)\coloneqq \big\{z=(x,s)\in Z^d_+(\zeta):   s> 3\sigma, |x-\xi|<3s \big\}.
\end{split}
\end{equation}
\subsubsection{Symmetries parametrized by $z\in Z^d$}Throughout, the one parameter family of symmetries on $\phi\in \mathcal S(\R^d)$ is defined as
\[\begin{split}
&\mathsf{Tr}_{x} \phi(\cdot) = \phi(\cdot- x),  \qquad
 \mathsf{Dil}^p_{s}\phi(\cdot) = \frac{1}{s^{d/p}} \phi \left( \frac{\cdot}{s}\right), \qquad  x\in \R^d, \, s>0,\,  p\in (0,\infty]; \\
 &\mathsf{Sy}_{z} \phi = \mathsf{Dil}^1_{s} \circ \mathsf{Tr}_{x}, \qquad  z=(x,s) \in Z^d.
\end{split}
\]
\subsubsection{Cutoffs}Choose $\alpha\in \mathcal C^\infty( \R^d) $, radial and with
$\alpha=1$ on $\mathsf{B}_{(0,2)}$, $\supp \, \alpha \subset \mathsf{B}_{(0,4)}$, and accordingly define the cutoffs
\begin{equation}
\label{e:cutoffs}
\alpha_{z}\coloneqq  \mathsf{Tr}_{x} \mathsf{Dil}^\infty_{s}\alpha ,\qquad \beta_z =1-\alpha_z, \qquad z=(x,s) \in Z^d.
\end{equation}
Note that $ \supp\, \alpha_z \subset 4\mathsf{B}_z=\mathsf{B}_{(x,4s)}$ $ \supp \,\beta_z\subset \R^{d} \setminus 2\mathsf{B}_z$, and 
unlike most other functions parametrized by $z\in Z^d$ the cutoffs $\alpha_z,\beta_z$ will always be $\infty$-normalized. 
\subsubsection{Measuring decay in $Z^d$}For a  decay parameter $\nu>0$, define the function
\begin{equation}
\label{e:brack}
[\cdot]_\nu: Z^d \to (0,1], \qquad [(x,s)]_\nu\coloneqq \frac{(\min\{1,s\})^\nu}{(\max\{1,s,|x|\})^{d+\nu}}.
\end{equation}
Throughout the article, the fact that
\begin{equation}
\label{e:integbrack} \int_{Z^d} [z]_{\nu} \,\d \mu(z) \lesssim_\nu 1, \qquad \nu >0
\end{equation}
will be heavily exploited.  
The  geometric separation in the parameter space $Z^{d}$ is then described by the function
\begin{equation}
\label{e:Delta1}
[z,\zeta]_\nu = \frac{1}{s^d} \left[\left(\frac{\xi-x}{s}, \frac{\sigma}{s} \right)\right]_\nu = \frac{
\left(\min\{s,\sigma\}\right)^{\nu}}{\left(\max\{s,\sigma,|x-\xi|\}\right)^{d+\nu}},\qquad z=(x,s), \, \zeta=(\xi,\sigma) \in Z^d.
\end{equation}
\subsubsection{Integration by parts}
Throughout the paper, for $k\in \mathbb N$, a multi-index $\gamma\in \R^d$ with $|\gamma|=k$ and a Schwartz function $f$ with $\widehat f(\xi)=O(|\xi|^{k})$ as $\xi \to 0$, we denote
\begin{equation}
\label{e:pgamma}
\partial^{-\gamma}  f (x) = \frac{1}{(\sqrt{2\pi})^d} \int\displaylimits_{\R^d} \widehat{f}(\xi) \frac{(i\xi)^\gamma}{|\xi|^{2k}} \e^{ix\cdot \xi}\, \d \xi.
\end{equation} Then $\partial^{-\gamma}  f\in \mathcal S(\R^d)$ and
Plancherel's theorem implies the equality
\[
\langle f, g\rangle = \sum_{|\gamma|=k} \left\langle \partial^{-\gamma} f, \partial^ \gamma g \right\rangle \eqqcolon \langle \nabla^{-k} f,\nabla^k g\rangle, \qquad  g\in \mathcal S(\R^d).
\]
\begin{remark} 
\label{r:riesz}  Let $\kappa \in \R$, $u \in \{1,\ldots,d\}$ and $|\nabla|^\kappa$, $R_u$ be respectively the $\kappa$-th order Riesz potential and $u$-th Riesz transform on $\R^d$,
\[
|\nabla|^{-\kappa} f(x)  = \frac{1}{(\sqrt{2\pi})^d} \int\displaylimits_{\R^d} |\xi|^{-\kappa}\widehat{f}(\xi)   \e^{ix\cdot \xi}\, \d \xi, \qquad 
R_{u} f(x)  = \frac{1}{(\sqrt{2\pi})^d} \int\displaylimits_{\R^d} \widehat{f}(\xi) \frac{i\xi_u}{|\xi|} \e^{ix\cdot \xi}\, \d \xi.
\]
 For a multi-index $\gamma=(\gamma_1,\ldots, \gamma_d)$, let $R^\gamma=R_{1}^{\gamma_1}\circ \cdots \circ R_{d}^{\gamma_d}$.  With this notation
$
\partial^{-\gamma} =  |\nabla|^{-|\gamma|} R^\gamma
$ up to a multiplicative constant depending on $d,|\gamma|$ only. This multiplicative constant will be ignored in the subsequent uses of this remark.
\end{remark}
\subsubsection{Mother wavelet}
Let $\Phi\in \mathcal C^\infty (\R^d)$ be  radial and supported on $\mathsf{B}_{(0,1)}$,   $D\in \mathbb N $ a fixed large parameter,\footnote{For instance, when proving Theorems \ref{t:T1}, \ref{t:T12p} below, any $D\geq 8(\max\{k_1,k_2\}+ d_1 + d_2)$ will suffice.} and 
 $a= a(d,D)>0$ chosen so that \eqref{e:CRF} below holds.  Define the mother wavelet $\varphi$ by
\begin{align}
\label{e:mw1} 
&\varphi\coloneqq a\Delta^{4D} \Phi \in \mathcal C^\infty (\R^d),  \qquad\mathrm{supp} \,\varphi \subset \mathsf{B}_{(0,1)}, \qquad  \varphi \; \textrm{ radial}, \qquad 
\int_{\R^{d}} |\varphi| \, \d x =C(d,D).
\end{align}
This definition implies 
\begin{align}
\label{e:ck1paral} &
\partial^{-\alpha} \varphi = \partial^{\alpha} \Delta^{4D-|\alpha|} \Phi\in \mathcal S(\R^d), \qquad \mathrm{supp}\, \partial^{-\alpha} \varphi \subset \mathsf{B}_{(0,1)}, \qquad \forall 0\leq |\alpha| \leq D,
\end{align}
and in particular
\begin{align}
\label{e:ck1par}
&\int_{\R^{d}} x^{\gamma} \psi (x) \, \d x  =0   
\end{align}
holds  for all  $\psi\in\{\partial^{-\alpha}\varphi: 0\leq|\alpha|\leq  D\}$  and all $0\leq |\gamma|\leq D$.  
 The  translated, rescaled functions
\begin{equation}
\label{e:mw5}
\varphi_z = \Sy_z \varphi\qquad z\in Z^d \
\end{equation}
yield the Schwartz version of the Calder\'on reproducing formula \cite{BT,FJWb,WilsonB}
\begin{equation}
\label{e:CRF}
h= \int _{Z^d} \langle h,\varphi_\zeta  \rangle   \varphi_\zeta   \,\d \mu(\zeta) , \qquad h\in \mathcal S(\R^d).
\end{equation}
\subsection{Wavelet classes} \label{ss:21}
For $\nu>0$, $0<\delta\leq 1$ define 
the norm on $ \mathcal S(\R^d)$
\[
\|\phi\|_{\star,\nu,\delta}= \sup_{x\in \R^d}\langle x \rangle^{d+\nu}\left|\phi(x)\right| +  \sup_{x\in \R^d}\sup_{\substack{ h\in \mathbb \R^d \\ 0<|h|\leq 1}} \langle x \rangle^{d+\nu}\frac{\left|\phi(x+h) - \phi(x)\right|}{|h|^{\delta}}.
\] 
Using this norm and $\mathsf{Sy}_z$, adapted classes are defined. For    $k\in \mathbb N, $   $0<\delta\leq 1$ set
\[
\Psi_z^{k,\delta;1}=\left\{\phi\in \mathcal S(\R^d) : s^{|\gamma|}\left\| (\mathsf{Sy}_{z})^{-1} \partial^{\gamma} \phi \right\|_{\star,k+\delta,\delta} \leq 1  : 0\leq |\gamma|\leq k\right\}, \qquad z=(x,s)\in Z^d.
\]
The membership $\phi\in \Psi^{k,\delta;1}_{z}$, for a fixed $z = (x,s)\in Z^d$, yields the following quantitative decay and smoothness conditions: for each multi-index $\gamma$ on $\R^d$ with $0\leq |\gamma|\leq k$, there holds 
\begin{align}
&
|\partial^\gamma\phi(y)| \leq \frac{1}{s^{d+|\gamma|}} \left \langle  \frac{y-x}{s} \right\rangle^{-(d+k+\delta)} , \qquad y \in \R^d;  \label{e:decayconcrete1}
\\ &
|\partial^\gamma\phi(y+h)-\partial^\gamma\phi(y)| \leq \frac{|h|^\delta}{s^{d+|\gamma|+\delta}} \left \langle  \frac{y-x}{s} \right  \rangle^{-(d+k+\delta)} , \qquad y \in \R^d, h \neq 0.  \label{e:decaysmoothconcrete1}
\end{align}
Then set
\begin{align} \label{e:smoothclass1p}
&  \Psi_z^{k,\delta;0}\coloneqq\left\{\psi \in \Psi_z^{k,\delta;1}: \;\textrm{\eqref{e:ck1par} holds} \; \forall\, 0 \leq |\gamma| \leq k \right\}.
\end{align}
When $k=0$, the notation is simplified by writing $\Psi_z^{\delta;1},\Psi_z^{\delta;0}$ in place of $\Psi_z^{0,\delta;1}, \Psi_z^{0,\delta;0}$.
\begin{remark} \label{r:ibp}
Note that $\varphi_z$  defined in \eqref{e:mw5} belongs to $C\Psi_{z}^{D,1;0} $. More generally if $0\leq |\gamma| \leq D$  
\begin{equation}
\label{e:phigammaz}
\varphi_{\gamma,z} \coloneqq  \mathsf{Sy}_z[ \partial^{-\gamma} \varphi] = s^{-|\gamma|} \partial^{-\gamma}   \mathsf{Sy}_z \varphi \in C\Psi_{z}^{D,1;0}, \qquad \mathrm{supp}\,  \varphi_{\gamma,z}\subset \mathsf{B}_z.\end{equation}
\end{remark}
Limited decay wavelets enjoy the following almost-orthogonality estimates. 
\begin{lemma} \label{l:ttstar1}Let  $0<\eta<\delta\leq 1$, $0\leq k\leq D$, $z=(x,s),\zeta=(\xi,\sigma) \in Z^d$ with $s\leq \sigma$. Then 
\[
\sup_{\psi \in \Psi^{k,\delta;0}_{z} } \sup_{\phi \in \Psi^{k,\delta;1}_{\zeta} }
\left|\langle \psi, \phi  \rangle \right| \lesssim_{\eta} [z,\zeta]_{k+\eta},
\qquad
\sup_{\phi \in \Psi^{k,\delta;1}_{\zeta} }
\left|\langle \varphi_z, \phi  \rangle \right| \lesssim_{k,\delta} [z,\zeta]_{k+\delta}.
\]
\end{lemma}
 \begin{proof}  Consider the first estimate. By scale invariance and symmetry one may reduce to the case $\zeta=(0,1)$, $z=(x,s)$ with $s\leq 1$. Further, assume $|x|\geq 1$ as the case $|x|\leq 1$ is strictly easier.
In this case
$[z,\zeta]_{k+\eta }= s^{k+\eta}{|x|^{-(d+\eta)}}.$
Thanks to the vanishing moment properties of $\psi$, one can subtract $T_x(y)$, the Taylor polynomial of $\phi$ of order $k$ centered at $x$. Then one has 
\[ \begin{split} &
|\l \psi,\phi \r| \leq \int\displaylimits_{|y-x|< 1} \left|\phi(y) - T_x(y)  \right|  |\psi(y)|\,  \, \d y +  \int_{|y-x|\geq 1}   |T_x(y)||\psi(y)|\,  \, \d y + \int\displaylimits_{|y-x|\geq 1}  |\phi(y)| |\psi(y)|\,  \, \d y. 
\end{split}
\]
Using  \eqref{e:decaysmoothconcrete1} for $\nabla^k\phi$  and \eqref{e:decayconcrete1} for $\psi$,
\begin{equation}
\label{e:lemdec0}\nonumber\begin{split} &\quad 
\int\displaylimits_{|y-x|< 1}  \left|\phi(y) - T_x(y)  \right| |\psi(y)| \, \d y \lesssim \frac{1}{|x|^{(d+\delta+k)}} \int\displaylimits_{|y-x|< 1} \frac{|y-x|^{\delta+k}}{\left(1+\frac{|y-x|}{s}\right)^{d+k+\delta}} \frac{\d y}{s^d}\\ & \lesssim \frac{s^{k+\delta} |\log s| }{|x|^{(d+k+\delta)}}\lesssim_\eta [z,\zeta]_{k+\eta }.\end{split}
\end{equation}
Using \eqref{e:decayconcrete1} for $\nabla^k\phi$ instead gives,  \[ \begin{split}
\label{e:lemdec1}
&  \int_{|y-x|\geq 1} |T_x(y)|  |\psi(y)|\,  \, \d y \lesssim \frac{s^{k+\delta}}{|x|^{d+k+\delta}} \int\displaylimits_{|y-x|>1} |y-x|^{-(d+\delta)} {\d y} \lesssim    [z,\zeta]_{\delta },\\
& \int\displaylimits_{\substack{|y-x|\geq 1\\ 2|y|< |x|}}  |\phi(y)| |\psi(y)|\,  \, \d y \leq \int_{2|y-x|>|x|}  |\psi(y)| \, \d y \lesssim \int_{\frac{|x|}{2s}}^\infty  t^{-(k+\delta+1)} \d t \lesssim  [z,\zeta]_{\delta },
\\ &
\int\displaylimits_{\substack{|y-x|\geq 1\\ 2|y|> |x|}}  |\phi(y)| |\psi(y)|\,  \, \d y \leq  \frac{1}{|x|^{(d+\delta+k)}} \int_{|y-x|>1}  |\psi(y)| \, \d y \lesssim   [z,\zeta]_{\delta }.
\end{split}
\]
Assembling the last two displays yields the claimed first estimate.

The second estimate is proved similarly. Again, renormalize to have $\zeta=(0,1)$, $z=(x,s)$ with $s\leq 1$.   Taking advantage of Remark \ref{r:ibp} to rely on the vanishing mean of $\varphi_{\gamma,z} $, and using the decay  \eqref{e:decaysmoothconcrete1} for  $  \phi \in \Psi^{k,\delta;1}_\zeta$ gives
\[
|\langle \phi, \varphi_z \rangle| \leq  s^k \sum_{|\gamma|=k} \int\displaylimits_{\mathsf{B}_{z}}|  \partial^{\gamma} \phi(y) - \partial^{\gamma} \phi(x) | |  \varphi_{\gamma,z}(y)|\, \d y
\lesssim  \frac{s^{k+\delta}}{\max\{1, |x|\}^{d+k+\delta} } = [z,\zeta]_{k+\delta},\] 
and the proof is complete.
\end{proof}

\subsection{Intrinsic Forms and Sparse Estimates}
Lemma \ref{l:ttstar1} leads to the $L^2$-boundedness of an intrinsic square function associated to the classes $\Psi^{\delta;0}_z$.  This square function will now be defined.
For $f\in L^p(\R^d)$, $1\leq p \leq \infty$, $\iota\in \{0,1\}$  and $z\in Z^d$ define the intrinsic wavelet coefficients
\begin{align} \label{e:Phidelta1}&
\Psi^{\delta;\iota}_z f  \coloneqq  \sup_{\psi \in \Psi^{\delta;\iota}_z} |\langle f,\psi\rangle |. \end{align}
The $\iota=0$ coefficients enter the  intrinsic square function
\begin{equation}
\label{e:Sdelta1}
{\mathrm{S}_\delta} f(x) = \left(\int\displaylimits_{0 }^\infty\left(\Psi^{\delta;0}_{(x,s)} f\right) ^2 \, \frac{\d s }{s}\right)^{\frac12}.
\end{equation}
The value $\delta>0$ is fixed but arbitrary and, whenever possible, it will be omitted from the notation in \eqref{e:Phidelta1} and \eqref{e:Sdelta1}, writing for instance   ${\mathrm{S}}f$ instead of ${\mathrm{S}}_{\delta}f$. Lemma \ref{l:ttstar1} implies easily the $L^2$ estimate of the following proposition.
\begin{proposition} \label{p:intrinsicsf1} Let $f\in L^2(\R^d)$. Then $\|{\mathrm{S}_\delta} f\|_2\lesssim_\delta \|f\|_2$.
\end{proposition}
\begin{proof} It suffices to work with $f\in L^2(\R^d)$ of unit norm.  
  Standard considerations  and a change of variable reduce the claimed bound to  the estimate
\begin{equation} \label{e:ttstar1}
\begin{split} &
\int\displaylimits_{ {(x,s) \in\R^d \times (0,\infty)}} \int\displaylimits_{ {(\alpha, \beta)\in  \R^d \times (0,1)}}|\langle f, \psi_{(x,s)} \rangle| |\langle s^d \psi_{(x,s)}, \psi_{(x+\alpha s,\beta s)}\rangle| |\langle f, \psi_{(x+\alpha s,\beta s)}\rangle| \,\frac{\d x\d s \d\alpha \d\beta}{s\beta }\\ &  \lesssim    \int\displaylimits_{z\in Z^d} |\langle f, \psi_z \rangle|^2  \,{\d\mu(z)} \end{split}\end{equation}
with implied constant uniform over the choice of $\psi_z \in   \Psi^{\delta;0}_z$ and $ z\in Z^d $.  
 Lemma \ref{l:ttstar1} then yields 
\[
|\langle s^d \psi_{(x,s)}, \psi_{(x+\alpha s,\beta s)}\rangle| \lesssim [(\alpha,\beta)]_{\frac \delta 2}
\]
so that \eqref{e:ttstar1} follows  by an application of Cauchy-Schwarz and \eqref{e:integbrack}.
\end{proof}
Next, consider the  intrinsic bisublinear form
\begin{equation}
\label{e:modelform1}
\Psi^\delta(f, g) = \int\displaylimits_{Z^d}  \Psi_z^\delta f \cdot \Psi_z^\delta g \, \d \mu (z)
\end{equation}
acting  on pairs $f,g\in \bigcup_{1\leq p\leq \infty }L^{p}(\R^d)$. Note that the sum in \eqref{e:modelform1} is of nonnegative terms, therefore issues of convergence in the definition may be disregarded. 
 \begin{proposition} \label{t:sparsemodel1}
  For each pair $f,g\in L^{1}(\R^d)$ there exists a sparse   collection $\mathcal S$ of cubes of $\R^d$   with the property that
\[ \Psi^\delta(f,g) \lesssim_\delta \sum_{Q\in  \mathcal S} |Q|  \langle f \rangle_Q \langle g \rangle_Q, \qquad \langle f\rangle_Q \coloneqq \frac{1}{|Q|} \int |f|\cic{1}_Q \,\mathrm{d} x.
\]
\end{proposition}
The interested reader can consult \cite{CUDPOULMS,CDPOBP} for the definitions relative to sparse operators and $(p_1,p_2)$ sparse bounds.  The $(1,1)$ sparse bound for the form $\Psi^\delta$ is exactly the conclusion of Proposition \ref{t:sparsemodel1}.

A proof of this proposition is not provided. In fact, standard calculations show that $\Psi^\delta(f, g)$ may be linearized to  a $L^2$-bounded    $\eta$-H\"older continuous Calder\'on-Zygmund kernel form for all $0<\eta<\delta$. Therefore,  sparse bounds may be obtained following one of the standard approaches in the literature, see e.g. Lerner \cite{Ler2013,Ler2015}, Lacey \cite{Lac2015} (see also \cite{HRT}), or \cite{CoCuDPOu,CDPOMRL}. A direct proof along the lines of  \cite{CoCuDPOu,CDPOMRL} may also be employed to obtain stronger domination results where local oscillations replace  averages in the sparse form. This will appear in forthcoming work.

The intrinsic form of  \eqref{e:modelform1} models  cancellative operators of Calder\'on-Zygmund type. Analogous intrinsic forms modeling paraproducts are also needed in the analysis. Referring to the wavelet coefficients \eqref{e:Phidelta1}, define on triples $f_j\in L^1_{\mathrm{loc}}(\R^d)$ the intrinsic forms
\begin{equation}
\label{e:modelpp1}
\Pi^{j,\delta}(f_1,f_2,f_3) = \int\displaylimits_{Z^d} \Psi_z^{\delta;1} f_j \left(\prod_{\iota\in \{1,2,3\}\setminus j}  \Psi_z^{\delta;0} f_\iota\right) \, \d \mu(z).
\end{equation}
The index $j$ in the notation \eqref{e:modelpp1} identifies the noncancellative index of the paraproduct form. As these are modeling bilinear, one parameter Calder\'on-Zygmund forms, the case where $j\in \{1,2\}$ and $f_3\in \mathrm{BMO}(\R^d)$ is of particular interest. In this case, the simplified notation
\begin{equation}
\label{e:modelpp2}
\pi^\delta_{f_3}(f_1,f_2) = \Pi^{1,\delta}(f_1,f_2,f_3)
\end{equation} is adopted.
The analogous result to  Proposition \ref{t:sparsemodel1} for paraproducts follows. As for \eqref{e:modelform1}, when $\delta$ is fixed and not important in that context, write $\pi_{f_3}$ in place of $\pi^\delta_{f_3}$.
\begin{proposition} \label{t:parasparsemodel1} Let $f_3\in \mathrm{BMO}(\R^d)$. 
  For each pair $f_j\in L^{1}(\R^d)$ there exists a sparse   collection $\mathcal S$ of cubes of $\,\R^d$ with the property that
\[ \pi_{f_3}(f_1,f_2) \lesssim_\delta \|f_3\|_{\mathrm{BMO}(\R^d)} \sum_{Q\in  \mathcal S} |Q| \prod_{j=1}^2 \langle f_j \rangle_Q.
\]
\end{proposition}
The proof of Proposition \ref{t:parasparsemodel1} is also omitted: similarly to the remarks following Proposition \ref{t:sparsemodel1}, it may be obtained, to name a few, from any of the references \cite{CoCuDPOu,CDPOMRL,Lac2015,Ler2013,Ler2015}. 
\section{Some Technical Preliminaries}
\label{s3}
This section contains a few technical tools that will be used in the proofs of the representation theorems.
\subsection{Alpert basis} Choose a collection  of functions $\phi_{\gamma}\in \mathcal S(\R^d)$, indexed by multi-indices $0\leq |\gamma|\leq 2D$, supported in the unit cube of $\R^d$ with the properties
\begin{equation}
\label{e:int1}
\int_{\R^d} x^{\alpha}\phi_\gamma(x) \, \d x=\delta_{\gamma \alpha} \qquad \forall 0\leq |\alpha|\leq 2D.
\end{equation}
The collection $\phi_\gamma$ has been explicitly constructed by Alpert \cite{Alp}, see also the extension to general measures in \cite[Theorem 1.1]{RSW}.
A first technical lemma involves the Alpert basis.
\begin{lemma} \label{l:alpert}Let $ z=(x,s), \zeta=(\xi,\sigma)\in Z^d$, and with reference to \eqref{e:partz}, $ z\in A(\zeta)$. Define
\[
P_{z,\zeta} (v) \coloneqq \sum_{0\leq|\gamma|\leq k}  \langle \varphi_z, \mathsf{Sy}_\zeta \phi_\gamma \rangle \left(\frac{ v-\xi}{\sigma}\right)^\gamma,\qquad \chi_{z,\zeta}(v)\coloneqq  \varphi_z(v) -P_{z,\zeta} (v),  \qquad v\in \R^d.
\] 
Then
\begin{equation}
\label{e:Lpoly}
|\chi_{z,\zeta}(v) | \lesssim \frac{1}{s^d}\left(\frac{|v-\xi|}{s}\right)^{k}\min\left\{1,\frac{\max\{|v-\xi|,\sigma\}}{s}\right\}, \qquad v\in \R^d.
\end{equation}
\end{lemma}
\begin{proof} Let \[T_\xi \varphi_z  (v)=\sum_{0\leq|\gamma|\leq k} q_\gamma \left(\frac{ v-\xi}{\sigma}\right)^\gamma, \qquad q_\gamma\coloneqq \frac{\sigma^{|\gamma|}
\partial^\gamma\varphi_z(\xi)}{\gamma!}= \langle   T_\xi \varphi_z ,\mathsf{Sy}_\zeta \phi_\gamma\rangle  \]be the degree $k$ Taylor polynomial of $\varphi_z$ centered at $\xi$; the equality involving $q_\gamma$ is due to \eqref{e:int1}. By Taylor's theorem, as $\supp \mathsf{Sy}_\zeta \phi_\gamma \subset \mathsf{B}_\zeta$,
\[
\left| \langle \varphi_z, \mathsf{Sy}_\zeta \phi_\gamma \rangle -q_\gamma \right|=
\left|\langle \varphi_z- T_\xi \varphi_z,\mathsf{Sy}_\zeta \phi_\gamma\rangle  \right| \lesssim \| \varphi_z - T_\xi \varphi_z\|_{L^\infty(\mathsf{B}_\zeta)}\lesssim \frac{\sigma^{k+1}}{s^{d+k+1}}.\]
It follows that
\[
|\chi_{z,\zeta}(v) |
\lesssim |\varphi_z(v) - T_\xi \varphi_z (v) | + \frac{1}{s^{d}}\sum_{0\leq |\gamma|\leq k} \frac{\sigma^{k+1-|\gamma|}}{s^{k+1-|\gamma|}} \left(\frac{ |v-\xi|}{s}\right)^{|\gamma|}. 
\]
The first summand of the last display complies with the estimate in the right hand side of \eqref{e:Lpoly}, by Taylor's theorem and the fact that $\varphi_z \in \Psi^{k,1;0}_z$. The second summand is also bounded by the right hand side of \eqref{e:Lpoly}: this is easily seen by checking the cases $|v-\xi|\leq \sigma, \sigma<|v-\xi|\leq s, |v-\xi|>s$ separately. The latter remark completes the proof of the Lemma.
\end{proof}

\subsection{Averaging Yields Rough Wavelets}
In the representation theorems, the key steps involve a certain averaging of the wavelet $\varphi$ of \eqref{e:mw1}.
\begin{lemma} \label{l:average1} Let    $\{\varphi_z :z\in Z^d\}$ be as in \eqref{e:mw5}. Let   $0<\eta<\delta\leq 1$ and $0\leq k\leq D$.  Let $u:Z^d \to \mathbb C$ be a Borel measurable function with $|u(z)|\leq 1$. Then, there exists $C\lesssim_{k,\delta,\eta} 1$ such that for all $z=(x,s)\in Z^d$
\begin{align}
\label{e:averagproc} &
\psi_{z} (\cdot) \coloneqq \int\displaylimits_{\substack{\alpha \in \R^d }} \int\displaylimits_{0<\beta\leq  1}\frac{\beta^{k+ \delta}u((\alpha, \beta))}{\langle\alpha\rangle^{d+ k+\delta}}  \varphi_{ (x+\alpha s,\beta s) }(\cdot) \frac{ \d \beta \d \alpha}{\beta} \in C\Psi^{k,\delta;0}_z, \\ &\label{e:averagprocl}
\nu_{z} (\cdot) \coloneqq \int\displaylimits_{\substack{\alpha \in \R^d }} \int\displaylimits_{\beta>1}\frac{u((\alpha, \beta))}{(\max\{| \alpha|,\beta\})^{d+ k+\delta}  }  \varphi_{ (x+\alpha  s,\beta s) }(\cdot) \frac{ \d \beta \d \alpha}{\beta} \in C\Psi^{k,\eta;0}_z.
\end{align}
In particular, with reference to \eqref{e:brack}, 
\begin{equation}
\label{e:averagproc2} 
\upsilon_{z} \coloneqq  \psi_{z} + \nu_{z}= \int \displaylimits_{(\alpha,\beta) \in Z^d }  [(\alpha,\beta)]_{k+\delta} u((\alpha, \beta)) \varphi_{ (x+\alpha  s,\beta s) } \frac{ \d \beta \d \alpha}{\beta} \in C\Psi^{k,\eta;0}_z.
  \end{equation}
\end{lemma}
The proof of Lemma \ref{l:average1} is postponed till after the following useful application. 
\begin{lemma} \label{l:averages}   Let $\{\varphi_z :z\in Z^d\}$ be as in \eqref{e:mw5}.  Let    $\zeta\in Z^{d}$ be fixed   and $q_\zeta\in \Psi_{\zeta}^{k,1;1}$ with $\supp\, q_\zeta\subset \mathsf{B}_\zeta$. Then there exists an absolute constant $C=C(d,k)$ and  $\vartheta_\zeta\in C \Psi_{\zeta}^{k,1 ;1}$ such that
\begin{equation}
\label{e:averages}
\int\displaylimits_{z\in A(\zeta)} \langle h,\varphi_z \rangle \langle \varphi_z ,q_\zeta\rangle\, \d \mu(\zeta) = \langle h ,\vartheta_\zeta\rangle \quad \forall h \in \mathcal S(\R^d).
\end{equation}
Furthermore,
\begin{equation}
\label{e:alpert2} \int_{\R^d} x^\gamma \vartheta_\zeta(x)\, \d x =
\int_{\R^d} x^\gamma q_\zeta(x) \,\d x, \qquad \forall 0\leq |\gamma|\leq k.
\end{equation}
\end{lemma}
\begin{proof} Write $\zeta=(\xi,\sigma)$ throughout.   Formula \eqref{e:CRF} yields that
\[
\begin{split}
 &\quad \int\displaylimits_{A(\zeta)} \langle h,\varphi_z  \rangle \langle \varphi_z, q_\zeta\rangle\, \d \mu(z) = \langle h ,q_\zeta\rangle - \int\displaylimits_ {  Z^{d}\setminus A(\zeta)} \langle h,\varphi_z\rangle \langle \varphi_z, q_\zeta\rangle\, \d \mu(z) .
\end{split}
\]
Support considerations show that $\langle \varphi_z, q_\zeta\rangle=0 $ for $z=(x,s)$ with $|x-\xi|>3\max\{\sigma, s\}$. An application of Fubini's theorem leads to the equality
\[
 \int\displaylimits_ {  Z^{d}\setminus A(\zeta)} \langle h,\varphi_z\rangle \langle \varphi_z, q_\zeta\rangle\, \d \mu(z) =   \langle h, \psi_\zeta\rangle, \qquad   \psi_\zeta\coloneqq  \int\displaylimits_ { I(\zeta) }  \langle \varphi_z, q_\zeta\rangle  \varphi_z\,  \d \mu(z), 
\] where  $I(\zeta)\coloneqq\{(x,s):s\leq 3\sigma, |x-\xi|\leq 3\max\{s,\sigma\}\}$.
If $(x,s) \in I(\zeta)$, Lemma \ref{l:ttstar1} implies that $ |\langle \varphi_{(x,s)}, q_{\zeta}\rangle|\lesssim \sigma^{-d}(s/\sigma)^{k+1} $. A change of variable and an application of Lemma \ref{l:average1}, \eqref{e:averagproc} in particular,   shows $\psi_\zeta\in C  \Psi_{\zeta}^{k,1 ;0}$. The proof is completed by setting  $\vartheta_z=q_z-\psi_z$ and deducing \eqref{e:alpert2} from Fubini's theorem.
\end{proof}
\subsection{Proof of Lemma \ref{l:average1}}
 First of all,  Fubini's theorem immediately implies that  $\psi_{z}$ and $\nu_z$  inherit the moment properties \eqref{e:ck1par}.   The memberships $c\psi_{z}\in \Psi^{k,\delta;1}_z,$ $ c\nu_{z}\in \Psi^{k,\eta;1}_z $ are needed and proved now.
\begin{proof}[Proof of \eqref{e:averagproc}] 
By invariance of assumptions and conclusions under the family $\Sy_z$, it suffices to  work in the case $z=(0,1)$. As $z$ is  thus fixed below, it is omitted from the subscript notation. We turn to showing that    $\|\partial^\gamma \psi\|_{\star,k+\delta,\delta}\lesssim 1$ for each $\gamma$ with $0\leq |\gamma|=\kappa\leq k$.
  Fix $w\in \R^d$, and let $\phi=\partial^\gamma \varphi$ locally. 
 Then,
 \[
 \partial^\gamma \psi(w) = \int\displaylimits_{\substack{\alpha \in \R^d }} \int\displaylimits_{0<\beta\leq  1}\frac{\beta^{k-\kappa+\delta}u((\alpha, \beta))}{\langle\alpha\rangle^{d+ k+\delta}}  \phi\left( \frac{w-\alpha}{\beta}\right) \frac{ \d \beta \d \alpha}{\beta^{d+1}}.
 \] 
    Due to the support properties of $\varphi$, one observes that the functions
\[\alpha\mapsto \phi\left( \frac{w-\alpha}{\beta}\right), \quad \alpha\mapsto \phi\left( \frac{w+h-\alpha}{\beta}\right),\] 
are supported in the cube $Q_w=w+[-3,3]^d$, and  $\langle \alpha\rangle \sim  \langle w\rangle$ for $\alpha\in Q_w$. Hence,
\begin{equation}
\label{e:avepf01,1}
\begin{split}
 |\partial^{\gamma}\psi(w)| &\lesssim  \langle w\rangle^{-(d+k+\delta)} \int\displaylimits_{\substack{\alpha \in Q_w \\ 0<\beta\leq  1}} \beta^{\delta} \left|\phi\left({\textstyle \frac{w - \alpha}{\beta}}
\right)\right| \frac{ \d \beta \d \alpha}{\beta^{d+1}}
=  \int\displaylimits_{\substack{ v\in \R^d\\ 0<\beta\leq  1}} \beta^{\delta-1} \left|\phi\left(v\right)\right| \d v\d \beta \, \lesssim_{\delta} 1
\end{split}
\end{equation}
by Fubini's theorem and the change of variable $v=\frac{w - \alpha}{\beta}$. Hence 
$
\sup_{x\in \R^d}\langle x \rangle^{k+\delta} |\partial^\gamma\psi(x)|\lesssim 1.$
\noindent We turn to the H\"older continuity estimate
\begin{equation}
\label{e:hce} \left| \partial^{\gamma}\psi(w+h)-\partial^{\gamma}\psi(w)\right|\lesssim |h|^{\delta}  \langle w\rangle^{-(d+k+\delta)}, \qquad h\in \R^d.
\end{equation}
This is stronger than $\|\partial^\gamma\psi\|_{\star,k+\delta,\delta}\lesssim 1$ only in the range $|h|\leq \frac12$, which  will now be assumed.
 Proceeding as before, two integrals must be controlled 
\begin{equation}
\label{e:avepf1,1}
\begin{split}
&
\int\displaylimits_{\substack{\alpha\in Q_w \\  0<\beta\leq  \frac{|h|}{2}}} \beta^{\delta} \left[\left|\phi\left(\textstyle\frac{w -\alpha}{\beta}\right)\right|  + \left|\phi\left(\textstyle\frac{w +h - \alpha}{\beta}\right)\right|\right] \frac{  \d \alpha\d \beta}{\beta^{d+1}} 
 +
 \int\displaylimits_{\substack{\alpha\in Q_w \\ \frac{|h|}{2}<\beta\leq 1}}\beta^{\delta}\left|\phi\left(\textstyle\frac{w +h - \alpha}{\beta}\right)-\phi\left(\textstyle\frac{w - \alpha}{\beta}\right)\right|\frac{  \d \alpha\d \beta}{\beta^{d+1}}.
 \end{split}
\end{equation}
A change of variable shows that both   summands in the first integral of \eqref{e:avepf1,1} are    \[
\lesssim  
\int_{0}^{\frac{|h|}{2}}\beta^{\delta-1} \,\d\beta \lesssim |h|^{\delta}.
\]
Notice that in the $\alpha$-support of the second integral in \eqref{e:avepf1,1}, that $ |h|\leq 2\beta$ and \[\min\{|w  - \alpha|,|w +h - \alpha|\}\leq \beta\] because of the support property of $\phi$. Therefore    such support has diameter  $\lesssim \beta$.
Using this fact and the mean value theorem,   the second integral in \eqref{e:avepf1,1} is
\[\lesssim 
 |h|
\int_{\frac{|h|}{2}}^1 \beta^{\delta-2} \,\d\beta \lesssim |h|^{\delta}.
\]
This completes the proof that $\psi\in C\Psi^{\delta;0}_{(0,1)}$ as desired.
\end{proof}
\begin{proof}[Proof of \eqref{e:averagprocl}] Again normalize $z=(0,1)$. Fixing   $0\leq \kappa \leq k$, and using the local notation $f\coloneqq\nabla^\kappa \nu_z$, it must be shown that $\|f\|_{\star,k+\eta,\eta}\lesssim 1$. Note that
\[
  f(\cdot) =
\int\displaylimits_{\substack{\alpha \in \R^d}} \int\displaylimits_{\beta>1} \frac{u((\alpha, \beta))}{(\max\{| \alpha|,\beta\})^{d+ k+\delta} }  \phi\left( \frac{\cdot-\alpha}{\beta}\right) \frac{ \d \beta \d \alpha}{\beta^{d+\kappa +1}},
\]
where $\phi=\nabla^\kappa \varphi$ locally.  Bound the factor ${\beta^{-\kappa}}$ below by 1, even if it may improve certain estimates slightly.
 Fix $w,h\in \R^d$ with $|h|\leq \frac12$. First, observe that for each $\beta>1$, the set \[
Q_{\beta}=\left\{{\alpha}\in \R^d: \phi\left( \frac{w-\alpha}{\beta}\right) \neq 0  \right\} \cup \left\{{\alpha}\in \R^d: \phi\left( \frac{w+h-\alpha}{\beta}\right) \neq 0  \right\}
\] 
has diameter $\lesssim \beta$ due to the support condition on $\varphi$, whence $|Q_\beta|\lesssim \beta^d$. Furthermore, if $|w|\geq 4\beta$ and ${\alpha}\in Q_\beta$ then $|\alpha| \geq\frac{| w|}{2}\geq 2\beta$.  This provides 
\begin{equation}
\label{e:conch1} \nonumber
\begin{split}
\left| f(w) \right|& \lesssim \langle w\rangle^{-(d+k+\delta)} \int\displaylimits_{1}^{\max\{\frac{|w|}{4},1\}}  \frac{\d \beta}{\beta} +  \int\displaylimits_{\max\{\frac{|w|}{4},1\}}^{\infty}  \frac{\d \beta}{\beta^{d+k+\delta+1}} \lesssim \langle w\rangle^{-(d+k+\delta)} \log \langle w\rangle  \lesssim_\eta \langle w\rangle^{-(d+k+\eta)}.
\end{split}
\end{equation}
Using the mean value theorem for $\varphi$ and the previous observations 
\begin{equation}
\label{e:conch2} \nonumber
\begin{split}
\left| f(w+h) - f(w)\right|& \leq \int\displaylimits_{\beta>1}  \int\displaylimits_{\substack{{\alpha} \in Q_\beta}} \left|   {\phi}\left( \frac{w+h-\alpha}{\beta} \right)  - \phi\left( \frac{w-\alpha}{\beta} \right) \right| \frac{ \d \beta \d {\alpha}}{{(\max\{| \alpha|,\beta\})^{d+ k+\delta}} \beta^{d+1}}\\  & \lesssim |h| \int_{1}^\infty \frac{\d \beta}{\max\{|w|, 4\beta\}^{d+k+\delta}\beta^{2}} \lesssim |h| \langle w\rangle^{-(d+k+\delta)},
\end{split}
\end{equation}
and collecting the last two estimates is more than enough to show that $\|f\|_{\star,k+\eta,\eta}\lesssim 1$. This also completes the proof of Lemma \ref{l:average1}. 
\end{proof}

\section{Wavelet Representation of One Parameter Calder\'on-Zygmund Operators} \label{s:1p}

This section provides a representation theorem for one parameter Calder\'on-Zygmund forms $\Lambda$ involving wavelet coefficients with vanishing moments. 
Throughout the section,  $\Lambda$ stands for a  continuous bilinear  form on $\mathcal S(\R^d)$ with  adjoint form 
\[\Lambda^\star: 
\mathcal S(\R^d) \times \mathcal S(\R^d)\to \mathbb C,\qquad 
\Lambda^\star (f,g) \coloneqq \overline{\Lambda(g,f)}.
\] and two adjoint linear continuous operators  
 \[
 T,T^\star: \mathcal S(\R^d)\to \mathcal S'(\R^d), \qquad 
 \langle T f, g \rangle =  \Lambda( f, g), \qquad  \langle T^\star f, g\rangle =  \overline{\Lambda( g, f)}. 
 \]
  
 \subsection{Weak boundedness, kernel estimates and paraproducts}   
Below $k\in \mathbb N$ and $\delta>0$ are two parameters quantifying the weak boundedness and  off-diagonal kernel smoothness of the form $\Lambda$. This quantification is summarized by the norm
\begin{equation}
\label{e:deltaSI}
\|\Lambda\|_{\mathrm{SI}(\R^d,k,\delta)}\coloneqq \|\Lambda\|_{\mathrm{WB},\delta}+\|\Lambda\|_{\mathrm{K},k,\delta} 
\end{equation} with the quantities on the right hand side defined below. 
\begin{definition}[Weak boundedness]  \label{c:weak} 
 The form $\Lambda$ has the \textit{$\delta$-weak boundedness property} if there exists $C>0$ such that 
\[
s^{d}|\Lambda(\varphi_{z},\upsilon_{z}) | \leq C 
\]
uniformly over all $z=(x,s)\in Z^d$ , $\varphi_{z}, \upsilon_{z}\in \Psi_{z}^{\delta;1}$ with $\supp\,\varphi_{z}, \supp\,\upsilon_{z} \subset \mathsf{B}_z$. In this case, call $\|\Lambda\|_{\mathrm{WB},\delta}$ the least such constant $C $.
\end{definition}
\begin{definition}[Kernel estimates]\label{c:ker}  
For a function $K=K(u,v):\R^d\times \R^d\to \mathbb C$, recall the finite difference notation
\[
\Delta_{h|\cdot} K(u,v) = K(u+h,v) -   K(u,v),\qquad
\Delta_{\cdot| h} K(u,v) = K(u,v+h) -   K(u,v), \qquad u,v,h\in\R^d.
\]
The continuous bilinear form $\Lambda$  on $\mathcal S(\R^d)$ has the \textit{standard $(k,\delta)$-kernel estimates} if the following holds. There exists a function $K:\R^d\times \R^d\to \mathbb C$, $k$-times continuously differentiable away from the diagonal in  $ \R^d\times \R^d$ such that
\[
\Lambda(f,g)= \int\displaylimits _{\R^d\times \R^d} K(u,v) f(v)g(u) \, {\d v\d u}
\]
whenever $f,g \in \mathcal S(\R^d)$ are disjointly supported,  and satisfying the size and smoothness estimates for all
$u\neq v \in \R^{d}$, $h\in \R^{d}$ with $0<|h|\leq \frac12|u-v|$:
\begin{align}
\label{e:kernel1}
 & |u-v|^{d} \left[| \nabla_u^\kappa K(u,v)| + |\nabla_v^\kappa K(u,v)| \right] \leq C,\qquad 0\leq \kappa\leq k;
 \\\label{e:kernel2}  & |u-v|^{d+k} \left[|\Delta_{h|\cdot} \nabla^{k}_uK(u,v)| + |\Delta_{\cdot|h} \nabla^{k}_vK(u,v)| \right]
  \leq C   \left(\frac{|h|}{|u-v|}\right)^{\delta}.
\end{align}
Call $\|\Lambda\|_{\mathrm{K},k,\delta}$ the least  constant $C$ such that \eqref{e:kernel1} and \eqref{e:kernel2} hold.
\end{definition}{
Say that $ \Lambda\in \mathrm{SI}(\R^d,k,\delta)$ if the constant \eqref{e:deltaSI} is finite. The next two examples of forms in $\mathrm{SI}(\R^d,k,\delta)$ are the fundamental building blocks of Calder\'on-Zygmund forms with higher degree smoothness.
 \begin{definition}[Wavelet form] \label{c:cancf} Let $ \{ \beta_z,\upsilon_{z}\in \Psi_{z}^{k,\delta;0}:z\in Z^d \}$ be two families of cancellative wavelets. The form
\begin{equation}
\label{e:cancform}
\Lambda(f,g) = \int _{Z^d}  \langle f,  \beta_{z}\rangle  \langle \upsilon_{z},g \rangle  \, \d \mu(z)
\end{equation}
belongs to $ \mathrm{SI}(\R^d,k,\delta)$ and $\|\Lambda\|_{ \mathrm{SI}(\R^d,k,\delta)}\lesssim 1$. The weak boundedness property is contained in Proposition \ref{t:sparsemodel1} while the $(k,\delta)$ kernel estimate is obtained via a standard computation reliant on \eqref{e:decayconcrete1}-\eqref{e:decaysmoothconcrete1}.
 \end{definition}

 \begin{definition}[Paraproduct forms] \label{c:parf} Let $0\leq |\gamma|\leq D$ be a multi-index. Call the family
  $\{\vartheta_{\gamma,z}\in C \Psi^{D,1;1}_z:z\in Z^d\}$ a $\gamma$-family if 
 \begin{equation}
\label{e:alpert3} \int_{\R^d} x^\alpha\vartheta_{\gamma,z}(x)\, \d x =
t^{|\alpha|}\delta_{\gamma\alpha}, \qquad \forall 0\leq |\alpha|\leq |\gamma|.
\end{equation}
For instance, if $\phi_\gamma$ satisfies \eqref{e:int1}, then $\{\vartheta_{\gamma,z}\coloneqq \mathsf{Sy}_z\phi_\gamma:z\in Z^d\}$ is a $\gamma$-family.
  For a function $b\in \mathrm{BMO}(\R^d)$, and multi-indices $  \gamma,\alpha $, referring to \eqref{e:phigammaz} for  $\varphi_{\alpha,z}$ define
\begin{equation}
\label{e:para1p}
\Pi_{b,\gamma,\alpha} (f, g) = \int_{Z^d} \langle b, \varphi_{\alpha,z}\rangle \langle f, \vartheta_{\gamma,z}\rangle   \langle   \varphi_z,g\rangle   \,\d \mu(z).
\end{equation}
If $\gamma=\alpha$, simply write  $\Pi_{b,\gamma}$. 
It is important to stress, see Remark \ref{r:ibp}, that $\varphi_{\gamma,z}\in C\Psi^{D,1;0}_z$ for all $z\in Z^d$.
Absolute convergence of the above integral for $f,g\in L^1 (\R^d)$ is granted by the easily verified intrinsic estimate
\[
\left|\Pi_{b,\gamma,\alpha} (f, g)\right| \lesssim \pi_b(f,g)
\]
referring to \eqref{e:modelpp2}. In particular $\Pi_{b,\gamma,\alpha}$ has the $(1,1)$-sparse bound, which implies $L^2(\R^d)$ estimates and \emph{a fortiori} weak boundedness property of $\Pi_{b,\gamma,\alpha}$, with $\|\Pi_{b,\gamma,\alpha}\|_{\mathrm{WB},\delta}\lesssim \|b\|_{\mathrm{BMO}(\R^d)}.$ Standard calculations show that $ \|\Pi_{b,\gamma,\alpha}\|_{\mathrm{K},k,1}\lesssim_k \|b\|_{\mathrm{BMO}(\R^d)}$ for all $0\leq k \leq D$, so that
\[
\|\Pi_{b,\gamma,\alpha}\|_{ \mathrm{SI}(\R^d,k,1)} \lesssim \|b\|_{\mathrm{BMO}(\R^d)}, \qquad 0\leq k \leq D.
\]
\end{definition}
\begin{remark} The weak boundedness property of Definition \ref{c:weak} tests $\Lambda$ on smooth functions. The recent literature related to $T(1)$ and representation theorems, see for instance \cite{Hyt2010,MK1,LMOV} and references therein, favors testing conditions on indicator functions. When the form $\Lambda$ also satisfies  kernel estimates, the weak boundedness condition employed in this paper actually follows from indicator-type conditions and is therefore less restrictive. More precisely, suppose that the bilinear form $\Lambda$ is well defined on $L^\infty_0(\R^d) \times L^\infty_0(\R^d) $ and satisfies 
\[
s^{-d}\left|\Lambda\left(\cic{1}_{\mathsf{B}_z},\cic{1}_{\mathsf{B}_z}\right)\right|\leq 1\qquad  \forall z=(x,s) \in Z^d
\]
in addition to the $\delta$-kernel estimates \eqref{e:kernel1} and \eqref{e:kernel2}. Then $\|\Lambda\|_{\textrm{WB},\delta }\lesssim 1$, namely $\Lambda$ has the weak boundedness property   of Definition \ref{c:weak}. A  proof of this implication is  found in \cite{TLN}. 
\end{remark}
\subsection{Calder\'on-Zygmund forms of class $(k,\delta)$} \label{ss:czkdelta}Let $\phi\in \mathcal S(\R^d)$ be an auxiliary function with 
\begin{equation}
\label{e:phiR}\cic{1}_{ \mathsf{B}_{\left(0,1\right)}} \leq \phi \leq  \cic{1}_{\mathsf{B}_{(0,2)}},\end{equation}
and introduce the notation, for each multi-index $0\leq |\alpha|\leq k$ and $R>0$
\begin{equation}
\label{e:pR}
p_{R}^\alpha\in \mathcal S(\R^d), \qquad p_R^\alpha (x) = x^\alpha \Dil_{R}^\infty\phi(x), \quad x\in \R^d.
\end{equation}
Let  $\mathcal S_{D}(\R^d) $ be the subspace of functions   $\psi \in \mathcal S(\R^d)$ with the vanishing moment property \eqref{e:ck1par} for all multi-indices $0\leq |\alpha|\leq D$.  
If $0\leq |\alpha|\leq k< D$ and $ \Lambda \in \mathrm{SI}(\R^d,k,\delta)$, the limits 
\begin{equation}
\label{e:Tlambda}
\Lambda(  x^\alpha, \psi)  =  \lim_{R\to \infty}\Lambda\left(p_{ R}^\alpha,\psi\right), \qquad  \psi \in \mathcal S_{D}(\R^d)
\end{equation}
exist, do not depend on the particular choice of $\phi$, and define  linear continuous functionals on $\mathcal S_{D}(\R^d)$ see \cite[Lemma 1.91]{FTW} for a proof.
\begin{remark} \label{r:temam} If $ \Lambda \in \mathrm{SI}(\R^d,k,\delta)$ is a wavelet form of the type \eqref{e:cancform}, then the functionals $\Lambda(  x^\alpha, \cdot)$ vanish for all $0\leq |\alpha|\leq k$. This is easily verified by appealing to the cancellation properties of the families $ \{ \beta_z,\upsilon_{z}\in \Psi_{z}^{k,\delta;0}:z\in Z^d \}$.
\end{remark}
With \eqref{e:Tlambda} in hand, it is possible to ask whether $ \Lambda \in \mathrm{SI}(\R^d,k,\delta)$ admits $\kappa$-th order paraproducts for $0\leq \kappa \leq k$.
 \begin{definition}[$\Lambda$ has paraproducts of $\kappa$-th order] \label{c:par} 
Say that $ \Lambda \in \mathrm{SI}(\R^d,k,\delta)$ has \textit{paraproducts of $0$-th order}  
  if there exists $b_{0},b_0^\star\in \mathrm{BMO}(\R^d)$ with the property that \begin{equation}
\label{e:limitR0}
\Lambda(  \cic{1}, \psi) =  \l b_{0},\psi \r,\qquad   \Lambda^\star(  \cic{1}, \psi) =  \l \psi, b_{0}^\star \r \qquad \forall \psi \in \mathcal S_D(\R^d)\end{equation}
 If this is the case, referring to \eqref{e:para1p},  define the \emph{$0$-th order cancellative part} of $\Lambda$ as
\[
\Lambda_{0}(f,g) = \Lambda(f,g) -\left[\Pi_{b_0,0}(f,g) +\Pi_{b_0^\star,0}(g,f)\right]
\]
We now define inductively the property of having paraproducts of order $\kappa$  for $1\leq \kappa\leq k$. Suppose $\Lambda$ has paraproducts of order $0\leq \kappa<k$ and 
the  \emph{$\kappa$-th order cancellative part}  of $\Lambda$ has been defined.
Then $\Lambda$ has \textit{paraproducts of $(\kappa+1)$-th order}  
  if for each multi-index $\alpha$ with $  |\alpha|= \kappa+1$ there exists $b_{\alpha},b_\alpha^\star\in \mathrm{BMO}(\R^d)$ with the property that \begin{equation}
\label{e:limitR}
\begin{split}
\Lambda_{\kappa}(  x^\alpha, \psi) =  (-1)^{\kappa+1}\l b_{\alpha},\partial^{-\alpha}\psi \r, 
\qquad 
\Lambda^\star_{\kappa}(  x^\alpha, \psi) =  (-1)^{\kappa+1}\l \partial^{-\alpha}\psi, b_{\alpha}^\star \r 
\end{split}
\end{equation}
for all $\psi\in \mathcal S_D(\R^d)$.
Notice that the pairings on the right hand sides are well defined, as $\partial^{-\alpha} \psi \in H^1(\R^d) $ whenever $\psi\in \mathcal S_D(\R^d)$ and $|\alpha|<D$.  If this is the case, we define the \textit{$k$-th order cancellative part} of $\Lambda$ by
\begin{equation}
\label{e:cancpart}
\Lambda_{\kappa+1}(f,g) = \Lambda_{\kappa}(f,g) - 
\sum_{ |\alpha|=\kappa+1} \left[\Pi_{b_\alpha,\alpha} (f, g)+ \Pi_{b_\alpha^\star,\alpha} (g, f)\right].
\end{equation}
Here we set $\Lambda_{-1}(f,g)=\Lambda(f,g)$ to be consistent with the definition of $\Lambda_0(f,g)$ given above.
\end{definition}

\begin{remark}

Observe that \eqref{e:cancpart} is equivalent to
\begin{equation}
\label{e:cancpart2}
\Lambda_{\kappa}(f,g) = \Lambda(f,g) - 
\sum_{0\leq |\alpha|\leq \kappa} \Pi_{b_\alpha,\alpha} (f, g)+ \Pi_{b_\alpha^\star,\alpha} (g, f).
\end{equation}
The inductive procedure of the proof of Theorem \ref{t:T1} reduces to the case $\Lambda(f,g)=\Lambda_{\kappa}(f,g)$.
\end{remark}
\begin{remark} \label{r:BMO}
The $0$-th order condition 
\eqref{e:limitR0} is equivalent to the familiar assumption
\[
T(\cic{1})=b\in  \mathrm{BMO}(\R^d), \qquad T^\star(\cic{1})=b_\star \in \mathrm{BMO}(\R^d).
\] For $0\leq \kappa \leq k-1$, 
let $T^{\,}_{\kappa} , T^\star_{\kappa}$ be the adjoint operators to $\Lambda_{\kappa}$. In view of Remark \ref{r:riesz}, as $R^\gamma $ preserves $\mathrm{BMO}(\R^d)$, the condition may be reformulated as
\begin{equation}
\label{e:bmoa}
 |\nabla|^{{\kappa}} T^{\,}_{\kappa-1}(x^\alpha) =  a_\alpha\in \mathrm{BMO}(\R^d), \qquad |\nabla|^{\kappa} T^\star_{0,\kappa-1}(x\mapsto x^\alpha) =  a _{\alpha}^\star\in \mathrm{BMO}(\R^d), \end{equation}
in the sense of $\mathcal S_D'(\R^d)$,   where $a_\alpha\coloneqq R^{\alpha} b_\alpha$ and similarly for $a_{\alpha}^\star$. 
 \end{remark} 
\begin{remark} 
\label{r4:gp}
Using \eqref{e:alpert3} and Remark \ref{r:ibp}, one directly computes
\begin{equation}
\label{e:equalitiespar}
\Pi_{b,\gamma,\alpha} (x^\beta, g) = (-1)^{|\alpha|}\delta_{\gamma\beta} \langle b, \partial^{-\alpha} g \rangle, \qquad \Pi^\star_{b,\gamma,\alpha} (x^\beta,f)=0, \qquad 0\leq |\beta|\leq |\gamma|.
\end{equation}
Thus  $\Pi_{b,\gamma}$ has paraproducts of order $|\gamma|$ according to Definition \ref{c:par}, with $b_\beta=\delta_{\gamma\beta}b $  and $b_\beta^\star=0$ for all $0\leq |\beta|\leq |\gamma|.$
\end{remark}
\begin{definition}[] The continuous bilinear form $\Lambda$   belongs to the class $\mathrm{CZ}(\R^d,k,\delta)$ of  \textit{$(k,\delta)$-Calder\'on-Zygmund (CZ) forms}  if $\Lambda \in \mathrm{SI}(\R^d,k,\delta)$ and $\Lambda$ has paraproducts of order $k$. For further use, define the norm
\begin{equation}
\label{e:deltaCZ}
\|\Lambda\|_{\mathrm{CZ}(\R^d,k,\delta)}\coloneqq \|\Lambda\|_{\mathrm{SI}(\R^d,k,\delta)}+ \sum_{0\leq |\alpha|\leq k}\left(\|b_\alpha\|_{\mathrm{BMO}(\R^d)} +\|b_\alpha^\star\|_{\mathrm{BMO}(\R^d)}\right).
\end{equation}
\end{definition}
The statement of Theorem \ref{t:T1} below is the representation and (sparse) $T(1)$-theorem  for $(k,\delta)$-CZ forms. Its proof is postponed to Subsection \ref{ss:pfA}. The weighted $T(1)$ result is stated separately  in Corollary \ref{cor:t1} with the deduction of the corollary  given at the end of this subsection.
\begin{theorem} \label{t:T1} Let $k\in \mathbb N$, $0 <\eps<\delta\leq 1$. There exists an absolute constant $C=C_{k,\delta,\eps,d}$ such that the following holds.
 Let $\Lambda$ be a standard $(k,\delta)$-CZ form, satisfying the weak boundedness condition, the kernel estimates 
and having paraproducts  
with normalization $\|\Lambda\|_{\mathrm{CZ}(\R^d,k,\delta)} \leq 1. $ 
Then, there exists a family
$ \{ \upsilon_{z}\in C\Psi_{z}^{k,\eps;0}:z\in Z^d \}, 
 $   such that for all $f,g\in \mathcal S(\R^d)$
\begin{equation}
\label{e:representation1}
\begin{split} 
 \Lambda(f,g) &= \int_{Z^d}  \langle f,  \varphi_{z}\rangle  \langle \upsilon_{z},g \rangle  \, \d \mu(z)+
\sum_{0\leq |\gamma|\leq k} \Pi_{b_\gamma,\gamma} (f, g)+ \Pi_{b_\gamma^\star,\gamma} (g, f),
\end{split} \end{equation}
where, for all $0\leq |\gamma|\leq k$, $\Pi_{b_\gamma, \gamma}$ and  $\Pi_{b_\gamma^\star, \gamma}$ are explicitly constructed paraproducts of the form in Definition \ref{c:parf}, with $\|b_\gamma\|_{\mathrm{BMO}(\R^d)}, \|b_\gamma^\star\|_{\mathrm{BMO}(\R^d)} \leq \|\Lambda\|_{\mathrm{CZ}(\R^d,k,\delta)} $ as in Definition \ref{c:par}. 
\end{theorem} 
\begin{remark} \label{r:symmetrybroken} The first term on the right hand side of \eqref{e:representation1} is not symmetric with respect to taking adjoints: there is one point in the proof, see \eqref{e:symmbroken} where the symmetry between $f$ and $g$ is broken by choosing on which side the averaging Lemma \ref{l:average1} is applied. A representation where the $\varphi_z$-family is kept on the $g$-side may be obtained by reversing this choice.
\end{remark}
\begin{remark}
\label{r:hytlap}The recent article \cite{HytLap} devises a dyadic representation theorem for   $L^2$-bounded Calder\'on-Zygmund operators $T$ whose kernel obey estimates \eqref{e:kernel1} for some $k\geq 1$ and satisfying special cancellation assumptions, namely $T(x^\gamma), T^\star(x^\gamma)$ is a polynomial of degree $\leq k$ for all multi-indices $0\leq |\gamma|\leq k$, which is equivalent to assuming \eqref{e:bmoa} holds with $a_\gamma=a_{\gamma}^{\star}= 0$ for all $\gamma$. Their representation is of dyadic nature in the sense that it involves shifted dyadic grids and dyadic shift operators of arbitrary complexity, just like the classical one of  \cite{Hyt2010}, however a wavelet orthonormal basis is employed as a building block of the dyadic shifts instead of the Haar basis. The useful gain in comparison with \cite{Hyt2010} is a better decay rate of the dyadic shift coefficients in the representation.

 Theorem \ref{t:T1} differs from \cite[Theorem 1.1]{HytLap}  in several aspects. First, it is of $T(1)$ type and the vanishing   assumptions on the higher order paraproducts is replaced by the more general BMO assumption \eqref{e:bmoa}.   Furthermore,    our model operators have complexity zero and no averaging procedure over shifted grids is needed, resulting in a  compact and  computationally feasible expansion. This feature allows for a seamless and powerful extension to the bi-parameter case, which we perform in Theorem \ref{t:T12p} below.
\end{remark}
 We come to the $T(1)$ theorem. Notice that \eqref{e:cort1} below is a vacuous assumption when $k=0$, whence Theorem \ref{t:T1} has  a sparse, sharp weighted version of the classical $T(1)$ theorem  as a corollary. Also notice that no assumption is being made on the adjoint paraproducts $b_\gamma^\star$.
\begin{corollary}   \label{cor:t1}
Suppose that  $\Lambda$ is a standard $(k,\delta)$-CZ form with
\begin{equation}
\label{e:cort1}
b_\gamma=0 \qquad \forall 0\leq |\gamma |< k.
\end{equation}
Then   referring to \eqref{e:modelform1} and \eqref{e:modelpp2}, for each $|\alpha|=k$ the following estimate is true
\begin{equation}
\label{e:cort2}
\left|\Lambda(f,\partial^\alpha g)\right| \lesssim_\eta \sum_{|\beta|=k}\left[\Psi^{\eta}(\partial^\beta f,  g) + \sum_{  |\gamma|= k} \pi_{b_\gamma} (\partial^\beta f, g)+  \sum_{  0\leq |\gamma|\leq k}  \pi_{b_\gamma^\star} (g, \partial^\beta f)\right].
\end{equation}
Furthermore, the sharp weighted bound on the weighted Sobolev space $\dot W^{k,p}(\R^d; w)$ holds
\begin{equation}
\label{e:cort3}
\|T f\|_{\dot W^{k,p}(\R^d; w)} \lesssim [w]_{A_p}^{\max\left\{1,\frac{1}{p-1}\right\}}\| f\|_{\dot W^{k,p}(\R^d; w)}, \qquad   p\in(1,\infty). 
\end{equation}
\end{corollary}
\begin{remark}\label{r:nec} Condition \eqref{e:cort1} is also necessary for \eqref{e:cort3} to hold, i.e. Corollary \ref{cor:t1} is a characterization of \eqref{e:cort3}. This generalizes the case $\Omega=\R^d$  of \cite[Theorem 1.1]{PXT} to the non-convolution case; in fact, a scaling argument shows that when $\Omega=\R^d$, condition b.\ in \cite[Theorem 1.1]{PXT} is equivalent to \eqref{e:cort1}. 
To see the necessity, suppose that  \eqref{e:cort3} holds for some exponent $p_0$ and all weights $w\in A_{p_0}$. Extrapolation of weighted norm inequalities \cites{CMPb,DuoJFA} then implies that \eqref{e:cort3} holds for $p=2d$ and $w$ equals Lebesgue measure. The content of Corollary \ref{cor:t1} also allows to assume that the adjoint $T$ to $\Lambda$ equals
\[
T f = \sum_{0\leq |\gamma|< k} \int_{Z^d} \langle b_{\gamma}, \varphi_{\gamma,z}\rangle \langle f, \vartheta_{\gamma,z}\rangle     \varphi_z   \,\d \mu(z).
\]
Fix $0\leq |\gamma|=\kappa< k$ and let $\eps\coloneqq k-(\kappa+\frac 1 2)>0$. Then define  $f_R(x) \coloneqq R^{\eps}{x}^\gamma \alpha_{(0,R)}(x)$ where $\alpha_z$ is the cutoff from \eqref{e:cutoffs} and $R>1$ is arbitrary. It is immediate to show that $\|f_R\|_{\dot W^{k,2d}(\R^d)}\sim 1 $, so  $Tf_R$ is a bounded sequence in $\dot W^{k,2d}(\R^d)$. Also, using the properties \eqref{e:alpert3} followed by \eqref{e:CRF},  $R^{-\eps}Tf_R \to \partial^{-\gamma}b_\gamma =T(x^\gamma) $ in the sense of Definition \ref{c:par}. These two properties entail $T(x^\gamma)=0$ in $\dot W^{k,2d}(\R^d)$. Thus $T(x^\gamma)$ is a polynomial of degree $\leq k$. Appealing to Definition \ref{c:par} again reveals that $b_\gamma=0$ as claimed.
\end{remark}
\begin{remark} \label{r:smotest}
Testing type theorems for smooth singular integral operators have previously appeared in several works: a non-exhaustive list includes \cites{TorMem,MeyEl,KunWang,FTW,HO} as well as the already mentioned \cites{HytLap,PXT} and references therein.
In particular, \cite[Theorem 1, cases (6,7)]{KunWang} is essentially equivalent to the unweighted version of Corollary \ref{cor:t1}. 
 Corollary \ref{cor:t1} appears to be the first  weighted $T(1)$ theorem of this type. A sparse bound in the vein of Corollary \ref{cor:t1} was proved in \cite{BB} for the case $k=1$ using techniques from \cite{Ler2013}. However,  the result of \cite{BB} is not of testing type and was obtained under the stronger assumption that $T$ is \emph{a priori} bounded on the Sobolev space $\dot W^{1,2}(\R^d)$.
\end{remark}

\begin{proof}[Proof of Corollary \ref{cor:t1}] Before the actual proof, make the following observations referring to the wavelets $\varphi_z, \upsilon_z$ in the representation \eqref{e:representation1}: for $z=(x,s)\in Z$, $ |\alpha|= |\beta| =|\nu| = k$, $0\leq |\gamma| \leq k$  
\begin{equation}
\label{e:corpf1}
\begin{split}
& s^{k} \partial^{\alpha} \varphi_z,\;   s^{k} \partial^{\alpha} \upsilon_z
 \in C\Psi_{z}^{\eps;0},  
\qquad  s^{-k} \partial^{-\beta} \vartheta_{\nu,z},\; s^{k} \partial^{\alpha} \vartheta_{\gamma,z}
   \in  C\Psi_{z}^{1;1}.
\end{split}
\end{equation}
Applying the representation theorem to $\Lambda$,  and using the assumptions on $b_\gamma$
\[
\begin{split}
&\Lambda(f,\partial^{\alpha}g)=  \int_{Z^d}  \langle f,  \varphi_{z}\rangle  \langle \upsilon_{z},\partial^{\alpha} g \rangle  \, \d \mu(z)+
\sum_{ |\gamma|= k} \Pi_{b_\gamma,\gamma} (f, \partial^{\alpha} g)+ \sum_{  |\gamma|\leq k}  \Pi_{b_\gamma^\star,\gamma} (\partial^{\alpha} g, f).
\end{split}
\]
 Integrating by parts and using Remark \ref{r:ibp} gives
\[
\begin{split}
 \left|\int_{Z^d}  \langle f,  \varphi_{z}\rangle  \langle \upsilon_{z},\partial^{\alpha} g \rangle  \, \d \mu(z)\right|= \left|\sum_{|\beta|=k}
  \int_{Z^d}  \langle \partial^\beta f,  \varphi_{\beta, z}\rangle  \left\langle  s^{k} \partial^{\alpha} \upsilon_z, g \right\rangle \d \mu(z) \right| \lesssim_\eta \sum_{ |\beta|=k}\Psi^{\eta}(\partial^\beta f,g).
\end{split}
\]
Fixing $|\nu|=k$ in the $b_\nu$-type paraproduct, and integrating by parts
\[
 \left|\Pi_{b_\nu,\nu} (f, \partial^\alpha g)   \right|=  \left|\sum_{ |\beta|=k}\int_{Z^d} \langle b_{\nu}, \varphi_{{\nu},z}\rangle \left\langle \partial^\beta f,  s^{-k} \partial^{-\beta} \vartheta_{{\nu},z} \right\rangle   \langle  s^{k} \partial^{\alpha} \varphi_z , g\rangle   \,\d \mu(z)\right|\lesssim \sum_{|\beta|=k} \pi_{b_{\nu}}(\partial^\beta f, g) .
\]
The $b_{\gamma}^\star$, $|0|\leq k \leq \gamma$ type paraproduct is controlled similarly: with reference to Remark \ref{r:ibp} for $\varphi_{\beta, z}$,
\[
\left|\Pi_{b_\gamma^\star,\gamma} (\partial^\alpha g, f) \right| = \left|\sum_{ |\beta|=k}\int_{Z^d} \langle b_\gamma^\star, \varphi_{\gamma,z}\rangle \langle g,  s^{k} \partial^{\alpha} \vartheta_{\gamma,z} \rangle   \langle  \varphi_{\beta, z}, \partial^\beta  f\rangle   \,\d \mu(z)\right|\lesssim \sum_{|\beta|=k} \pi_{b_\gamma}(\partial^\beta f, g) .
\]
This completes the proof of \eqref{e:cort2}. The weighted norm inequality then follows as a consequence of the sparse estimates
\[
\left|\langle \partial^\alpha T f, g\rangle \right| = \left| \Lambda(f,\partial^{\alpha}g)\right| \lesssim   \sum_{Q\in  \mathcal S} |Q|  \langle |\nabla^k f| \rangle_Q \langle g \rangle_Q, \qquad |\alpha|=k
\]
obtained by combining \eqref{e:cort2} with the Propositions \ref{t:sparsemodel1} and \ref{t:parasparsemodel1}.
\end{proof}
\subsection{Proof of Theorem \ref{t:T1}} \label{ss:pfA}Start by normalizing $\|\Lambda\|_{\mathrm{CZ}(\R^d,k,\delta)}=1$. Throughout the proof, the properties \eqref{e:mw1}-\eqref{e:mw5} and \eqref{e:int1} will be referred to frequently.  Recall that $\eps\in (0,\delta)$ is   fixed but arbitrary, and let $\eta=\frac{\eps+\delta}{2}$. 
Throughout the proof,  for $z,\zeta \in Z^d$ we write
\begin{equation}
\label{e:U}
\chi_{z,\zeta} \coloneqq   \varphi_z -P_{z,\zeta}\cic{1}_{A(\zeta)}(z) 
\end{equation}
 referring to  \eqref{e:partz} and Lemma \ref{l:alpert};
 note that this does not override the definition of  Lemma  \ref{l:alpert}.
\begin{lemma}\label{l:U1} $ \left|\Lambda(\chi_{z,\zeta}, \chi_{\zeta,z}) \right|\lesssim [(z,\zeta)]_{k+\eta}  \|\Lambda\|_{\mathrm{CZ}(\R^d,k,\delta)}$.
\end{lemma} 
\begin{proof} It suffices by symmetry to work in the region  $z\in Z^d_+(\zeta)$, see \eqref{e:partz}.    The estimates are then verified by case analysis.
\vskip1mm \noindent 
\emph{Case  $  z\in S(\zeta)$.} Estimate $\Lambda(\chi_{z,\zeta}, \chi_{\zeta,z})=\Lambda(\varphi_{z},  \varphi_{\zeta})$ appealing to the weak boundedness property.   The details are standard and omitted.
\vskip1mm \noindent \emph{Case $z\in A(\zeta)$.}
Let $\alpha_\zeta,\beta_\zeta$ as in \eqref{e:cutoffs}.  Then
\begin{equation}\label{e:taylor3}
\Lambda(\chi_{z,\zeta}, \chi_{\zeta,z})=\Lambda(\chi_{z,\zeta},\varphi_\zeta)  = \Lambda(\Theta, \varphi_{\zeta}) + \Lambda(\Xi, \varphi_{\zeta}), \qquad \Theta\coloneqq\chi_{z,\zeta}\alpha_\zeta, \qquad \Xi\coloneqq\chi_{z,\zeta}\beta_\zeta
\end{equation}
and one seeks estimates each of the summands in the last right hand side. For the first, apply the weak boundedness property for the point $\tilde \zeta=(\xi,4\sigma)$, so that  $\mathsf{B}_{\tilde{\zeta}}=4\mathsf{B}_\zeta$, and use \eqref{e:Lpoly} to estimate $\|\Theta \|_{\infty}$, obtaining
\begin{equation}
\label{e:taylor4}
|\Lambda(\Theta, \varphi_{\zeta})| \leq  \|\Lambda\|_{\mathrm{WB},\delta} \|\Theta \|_{\infty} \lesssim  \frac{1}{s^d} \left(\frac{\sigma}{s} \right)^{k+1}\lesssim [(z,\zeta)]_{k+1}.  \end{equation}
Continue now to estimate the second summand. The functions $\Xi$ and $ \varphi_{\zeta}$ have disjoint support, and thus the kernel representation of the form $\Lambda$ can be used. For each fixed ${v} \in \R^{d}\setminus 2\mathsf{B}_\zeta$, consider the function $F_{v}\in \mathcal C^k(\overline {\mathsf{B}_\zeta}), $ $F_{v}(u) \coloneqq K(u,v)$ for $u\in \overline{ \mathsf{B}_\zeta}$. Then
\begin{equation}
\label{e:taylor5a}
\begin{split} & \quad
|\Lambda(\Xi, \varphi_{\zeta})| = \left| \int_{ \R^{d}\setminus 2\mathsf{B}_\zeta}   \Xi({v})  \int_{  \mathsf{B}_\zeta} K(u,v) \varphi_{\zeta}(u) \, \d u \d {v} \right| = \left| \int_{ \R^{d}\setminus 2\mathsf{B}_\zeta}   \Xi({v})  \langle F_{v}, \varphi_{\zeta}  \rangle\,  \d v\right| 
\\  &\leq \sigma^k\sum_{|\gamma|=k}\left| \int_{ \R^{d}\setminus 2\mathsf{B}_\zeta}   \Xi({v})  \langle \partial_u^\gamma F_{v}, \sigma^{-k}\partial^{-\gamma} \varphi_{\zeta}  \rangle  \right| \\ & \leq \sigma^{k}  \sum_{|\gamma|=k}\int_{ \R^{d}\setminus 2\mathsf{B}_\zeta}  | \Xi({v})|  \sup_{u\in \mathsf{B}_\zeta } \left|\Delta_{u-\xi|\cdot}\partial^\gamma_u K(\xi,v) \right|  \|\varphi_{\gamma, \zeta}\|_1 \, \d {v}
\\ &
\lesssim   \sigma^{k+\delta}\|\Lambda\|_{\mathrm{K},k,\delta} \int_{ \R^{d}\setminus 2\mathsf{B}_\zeta} \frac{ | \Xi({v})|}{|{v}-\xi|^{d+k+\delta}}   \, \d {v}.
\end{split}\end{equation}
Here, the passage to the second line  is obtained by using $\supp\,\varphi_{\zeta}, \supp \, \varphi_{\gamma,\zeta}\subset \mathsf{B}_\zeta$, consult \eqref{e:ck1paral}, and integrating by parts.  The subsequent (in)equality follows from the mean zero property of  $\varphi_{\gamma,\zeta} $, see Remark \ref{r:ibp}, and the next step is obtained via the kernel estimates \eqref{e:kernel2}. Bound the last right hand side by splitting the integral on $\R^{d}\setminus 2\mathsf{B}_\zeta$ into the pieces
\begin{equation}
\label{e:taylor5}
  \sigma^{k+\delta}  \int\displaylimits_{ \sigma< |v-\xi|\leq s} \frac{ | \Xi(v)|}{|v-\xi|^{d+k+\delta}}   \, \d v \lesssim \frac{1}{s^d}   \left(\frac{\sigma}{s}\right)^{k+\delta}  \int_{\sigma}^s \frac{1}{t}   \, \d t \lesssim  \frac{1}{s^d} 
\left(\frac{\sigma}{s}\right)^{k+\delta} \log \left(\frac{s}{\sigma}\right)   
\end{equation}
where the $\delta$-geometric mean of the estimates in  \eqref{e:Lpoly} is used, and
\begin{equation}
\label{e:taylor6}
  \sigma^{k+\delta}  \int\displaylimits_{  |y-u|> s} \frac{ | \Xi(y)|}{|y-\xi|^{d+k+\delta}}   \, \d y \lesssim  \frac{1}{s^d} \frac{\sigma^{k+\delta}}{s^k}  \int_{s}^\infty \frac{1}{t^{1+\delta}}   \, \d t \lesssim    
\frac{1}{s^d}\left(\frac{\sigma}{s}\right)^{k+\delta}.
\end{equation} 
Putting together   \eqref{e:taylor3}, \eqref{e:taylor5a}, \eqref{e:taylor5} and \eqref{e:taylor6} provides
\begin{equation}
\label{e:repet}
 \left|   \Lambda(\chi_{z,\zeta},\varphi_\zeta) \right|  \lesssim\frac{\|\Lambda\|_{\mathrm{CZ}(\R^d,k,\delta)}}{s^d}\left(\frac{\sigma}{s}\right)^{k+\delta} \log \left(\frac{\sigma}{s}\right)   \lesssim_\eta   \frac{1}{s^d}\left(\frac{\sigma}{s}\right)^{k+\eta} =  [(z,\zeta)]_{k+\eta}
\end{equation}
as claimed.
\vskip1mm \noindent \emph{Case $z\in F_+(\zeta)$.} In this case $\Lambda(\chi_{z,\zeta}, \chi_{\zeta,z})=\Lambda(\varphi_{z}, \varphi_{\zeta})$ and the   supports of $\varphi_{z}$ and $\varphi_{\zeta} $ are separated. Thus, the kernel representation of $\Lambda$, the cancellation of $\varphi_{\gamma,\zeta}$ and the kernel estimates can be used. Proceeding exactly like in \eqref{e:taylor5a} with $\varphi_z$ in place of $\Xi$,
\begin{equation}
\label{e:taylor5afar}
\begin{split} 
|\Lambda(\varphi_z, \varphi_{\zeta})|  
\lesssim   \sigma^{k+\delta}\|\Lambda\|_{\mathrm{K},k,\delta} \int_{ \mathsf{B}_{z} }\frac{ | \varphi_z(v)|}{|v-\xi|^{d+k+\delta}}   \, \d v \lesssim \frac{\sigma^{k+\delta}}{s^d |x-\xi|^{k+\delta}} = [z,\zeta]_{k+\delta}\leq [z,\zeta]_{k+\eta}.
\end{split}\end{equation}
This completes the proof of the lemma.
\end{proof}
The main line of proof of Theorem \ref{t:T1} now begins. First of all, notice that \[
\|\Lambda\|_{\mathrm{CZ}(\R^d),\kappa,\delta} \leq \|\Lambda\|_{\mathrm{CZ}(\R^d,k,\delta)}=1, \qquad 0\leq \kappa \leq k.
\]
 The proof uses induction on $0\leq \kappa \leq k$ in a subtle way. 
 \subsubsection{Base case and main part of inductive step} The proof of the representation Theorem \ref{t:T1} first begins under the additional assumption 
 \begin{center}
 $a(k)$:  $b_\gamma,b_\gamma^\star\neq 0\implies |\gamma|=k$
 \end{center}
 namely, all paraproducts vanish except those of highest order. Notice that $a(0)$ is not an extra assumption.
 Let now $f,g \in \mathcal S(\R^d)$.  Use \eqref{e:CRF},  bilinearity, $\mathcal S(\R^d)$-continuity of $\Lambda$ and definition \eqref{e:U}  to expand $\Lambda(f,g)$ as
\begin{equation} 
\label{e:rep1,1} \nonumber
\begin{split}
 & \quad \int\displaylimits_{Z^d\times Z^d }    \langle f,\varphi_{z} \rangle \langle \varphi_{\zeta},g \rangle \Lambda(\varphi_{z},  \varphi_{\zeta}) \, \d \mu(z) \d \mu(\zeta) 
=   \int\displaylimits_{Z^d\times Z^d }   \langle f,\varphi_{z} \rangle \langle \varphi_{\zeta},g \rangle \Lambda(\chi_{z,\zeta},\chi_{\zeta,z}) \, {\d \mu(z) \d \mu(\zeta)} 
\\ &\quad + 
\int\displaylimits_{Z^d_\zeta} \int\displaylimits_{A(\zeta)}  \langle f,\varphi_{z} \rangle \langle \varphi_{\zeta},g \rangle \Lambda(P_{z,\zeta},\varphi_\zeta) \,  {\d \mu(z) \d \mu(\zeta) }  +  
\int\displaylimits_{Z^d_z} \int\displaylimits_{A(z)}  \langle f,\varphi_{z} \rangle \langle \varphi_{\zeta},g \rangle \Lambda(\varphi_z,P_{\zeta,z}) \, {\d \mu(\zeta) \d \mu(z) }.
\end{split}
\end{equation}
Making the change of variable $\xi=x+\alpha s,\sigma=\beta s$ and using Fubini's theorem in the inner variable of $\langle \varphi_{(x+\alpha s,\beta s)},g\rangle $, the first summand in the last right hand side equals
\begin{equation}
\label{e:symmbroken}
\int\displaylimits_{Z^{d}} \langle f,\varphi_{z} \rangle \langle \upsilon_{z},g \rangle\,  \d
\mu(z), \quad \upsilon_{(x,s)} \coloneqq  \int\displaylimits_{ (\alpha,\beta)\in Z^d   } [(\alpha,\beta)]_{k+\eta}  u_{(x,s)} ( \alpha,\beta)  \varphi_{(x+\alpha s,\beta s)} \frac{ \d \beta \d \alpha}{\beta}
\end{equation}
where \[u_{(x,s)} ( \alpha,\beta)\coloneqq\frac{  \Lambda\left(\chi_{(x,s),(x+\alpha s,\beta s)},\chi_{(x+\alpha s,\beta s),(x,s)}\right) }{ [(\alpha,\beta)]_{k+\eta}}\] is uniformly bounded by Lemma \ref{l:U1}. Thus   $\upsilon_{z}\in C\Psi_{z}^{k,\eps;0}$ by Lemma \ref{l:average1} applied with $u=u_{(x,s)}$.   This constructs the first expression in the right hand side of \eqref{e:representation1}.  With reference to Remark \ref{r:symmetrybroken},  an alternative form of the term in \eqref{e:symmbroken} with roles of $f$ and $g$ exchanged up to conjugation, may be obtained  by making instead the change of variable $x=\xi+\alpha\sigma, s=\beta\sigma$ and applying Lemma  \ref{l:average1} accordingly.

It remains to identify the second and third summand of the main decomposition as a sum of paraproduct terms.  Turning to this task for the first term, begin by noticing that due to assumption $a(k)$  used twice, and Remark \ref{r:ibp}
 \[
 \Lambda\left(y\mapsto \left( \frac{y-\xi }{\sigma}\right)^{\gamma},\varphi_\zeta\right) =  \Lambda\left( y\mapsto y^\gamma, \sigma^{-|\gamma|}\varphi_\zeta \right) = \begin{cases} \langle b_\gamma,   \varphi_{\gamma,\zeta} \rangle&|\gamma|=k \\ 0 & |\gamma|\neq k.
  \end{cases}
 \]
Therefore, applying Lemma \ref{l:averages} to $h=f$, $q_\zeta= \Sy_\zeta \phi_\gamma$ to obtain the last equality
\[\begin{split}
&\quad \int\displaylimits_{Z^d_\zeta} \int\displaylimits_{z\in A(\zeta)}  \langle f,\varphi_{z} \rangle \langle \varphi_{\zeta},g \rangle \Lambda(P_{z,\zeta},\varphi_\zeta) \,  {\d \mu(z) \d \mu(\zeta) }  \\&=\sum_{|\gamma|=k}
\int\displaylimits_{Z^d_\zeta} \int\displaylimits_{z\in A(\zeta)}  \langle f,\varphi_{z} \rangle \langle \varphi_z, \Sy_\zeta \phi_\gamma \rangle  \langle \varphi_{\zeta},g \rangle  \langle b_\gamma,  \varphi_{\gamma,\zeta}\rangle  \,  {\d \mu(z) \d \mu(\zeta) }  
\\ & =  \sum_{|\gamma|=k}
\int _{Z^d}    \langle f,\vartheta_{\gamma,\zeta} \rangle    \langle \varphi_{\zeta},g \rangle  \langle b_\gamma,  \varphi \rangle  \,  {  \d \mu(\zeta) } \coloneqq\Pi_{b_\gamma,\gamma}(f,g).
\end{split}
\]
Notice that Lemma \ref{l:averages} together with \eqref{e:mw1}-\eqref{e:mw5} and \eqref{e:int1} ensure that $\vartheta_{\gamma,\zeta}$, the output of Lemma  \ref{l:averages} corresponding to  $q_\zeta= \Sy_\zeta \phi_\gamma$, belongs to $ \Psi_\zeta^{D,1; 1}$  and is a $\gamma$-family. A totally symmetric argument deals with the third summand in the main decomposition, and completes the proof of \eqref{e:representation1} under the assumption $a(k)$.

\subsubsection{Induction} 
\label{ss:ind} 
It remains to explain how to obtain \eqref{e:representation1} without the assumption $a(k)$. In fact, it will be shown that $\Lambda$ satisfies an instance of representation \eqref{e:representation1} for all $0\leq \kappa \leq k$. This is done by induction on $\kappa$. 
 Before starting the induction, observe that $\|\Lambda\|_{\mathrm{CZ}(\R^d),\kappa,\delta}\leq \|\Lambda\|_{\mathrm{CZ}(\R^d,k,\delta)}\leq 1$. For $\kappa=0$, \eqref{e:representation1} is achieved in the previous step. 

Assume that $\Lambda$ has been represented in the form \eqref{e:representation1} for an integer $0<\kappa<k$. Taking advantage of  Definition \ref{c:parf}, and in particular of Remark \ref{r4:gp}, the  $\kappa$-cancellative part of $\Lambda$ defined in  \eqref{e:cancpart2} satisfies  
 $\|\Lambda\|_{\mathrm{CZ}(\R^d,k,\delta)}\lesssim  1$ and the equality 
\[
 {\Lambda}_{\kappa}(f,g)= \Lambda(f,g)- \sum_{0\leq |\gamma|\leq \kappa} \Pi_{b_\gamma,\gamma} (f, g)+ \Pi_{b_\gamma^\star,\gamma} (g, f) = \int_{Z^d}  \langle f,  \varphi_{z}\rangle  \langle \upsilon_{z},g \rangle  \, \d \mu(z)\]
for some family
$\{ \upsilon_{z}\in C\Psi_{z}^{\kappa,\eps;0}:z\in Z^d \}$. The last equality of the above display tells us that  all paraproducts of ${\Lambda}_{\kappa}$ having  order less than or equal to $\kappa$ equal zero, cf.\ Remark \ref{r:temam}.
 Therefore,
$ {\Lambda}_{\kappa}(f,g)  $ satisfies the assumptions of the theorem, and in addition $a(\kappa+1)$. Apply the previous part of the proof to $ {\Lambda}_{\kappa}(f,g) $, and obtain
\[\begin{split}
 {\Lambda}_{\kappa}(f,g) = \int_{Z^d}  \langle f,  \varphi_{z}\rangle  \langle \upsilon_{z},g \rangle  \, \d \mu(z)+ 
\sum_{ |\gamma|=\kappa+1} \Pi_{b_\gamma,\gamma} (f, g)+ \Pi_{b_\gamma^\star,\gamma} (g, f) 
\end{split}
\] with $\{\upsilon_{z}\in C\Psi_{z}^{\kappa +1,\eta;0}:z\in Z^d\}$.
This equality, rearranged, yields a representation of $\Lambda$  in the form \eqref{e:representation1} for the value $\kappa+1$, completing the inductive step and  the proof of Theorem \ref{t:T1}.

\section{Wavelets in the Bi-Parameter Setting} 
\label{s:4}

This section lays out the bi-parameter analogue of the  framework we described in Section \ref{s:2}, in preparation to a bi-parameter version of the representation Theorem \ref{t:T1}. 
  Throughout,  $\mathbf{d}=(d_1,d_2)$ is used to keep track of   dimension in each parameter. The base space is the product Euclidean space 
\[
x=(x_1,x_2)\in\R^{\mathbf{d}} \coloneqq \R^{d_1}\times \R^{d_2}.
\]  
If $\phi\in \mathcal S(\R^{\mathbf{d}})$ and $F\in \mathcal S(\R^{{d_1}})$, denote 
\[\langle \phi,F\rangle_1 = \int_{\R^{d_1}} \phi(x_1,\cdot) F(x_1)\, \d x_1\in \mathcal S(\R^{{d_2}})
\]
and similarly with roles of $1,2$ reversed.
 If  $\phi: \R^{\mathbf{d}}\to X$, $x_1\in \R^{d_1}, x_2\in \R^{d_2}$, the corresponding slices will be denoted by
\begin{equation}
\label{e:slices}
\begin{split}
&\phi^{[1,x_1]}: \R^{d_2} \to X, \qquad \phi^{[1,x_1]} \coloneqq \phi(x_1,\cdot), 
\\ 
&\phi^{[2,x_2]}: \R^{d_1} \to X, \qquad \phi^{[2,x_2]}\coloneqq \phi(\cdot,x_2) .
\end{split}
\end{equation}
 Our parameter space is thus the  product space
$
Z^{\mathbf{d}}= Z^{ {d}_1}  \times Z^{ {d}_2}  
$
with product measure $\d\mu(z) = \d\mu(z_1) \d\mu(z_2)$.
Vector notation for points of $Z^{\mathbf{d}}$ is not used and instead, we write $z=(z_1,z_2)\in Z^{\mathbf{d}} $. One embeds $Z^{d_j},$ $j=1,2$ into  $Z^{\mathbf{d}} $, regarded as a space of symmetries on $\phi \in \mathcal S(\R^{\mathbf{d}}) $, by taking tensor product with the identity transformation in the complementary parameter. Set
\[
\begin{split}
&
\left(\mathsf{Sy}^{1}_{z_1} \phi \right) (y_1,y_2) \coloneqq\left( \mathsf{Sy}_{z_1}\phi^{[2,y_2]}\right)(y_1) =\frac{1}{s_1^{d_1}}   \phi\left({ \frac{ y_1-x_1}{s_1}}, y_2\right), 
\\
&
\left(\mathsf{Sy}^{2}_{z_2} \phi \right) (y_1,y_2) \coloneqq \left(\mathsf{Sy}_{z_2}\phi^{[1,y_1]}\right)(y_2) =\frac{1}{s_2^{d_2}}   \phi\left(y_1, \frac{ y_2-x_2}{s_2}\right), 
\end{split}
\]
for $z_j=(x_j,s_j) \in Z^{d_j}$. Note that $\mathsf{Sy}^{1}_{z_1}, \mathsf{Sy}^{2}_{z_2} $ commute since they act on separate variables.  The bi-parameter family of symmetries indexed by $z\in Z^{\mathbf{d}}$ are obtained by composition,
\[\begin{split}
 & \mathsf{Sy}_{z} \phi = \mathsf{Sy}^{1}_{z_1} \circ \mathsf{Sy}^{2}_{z_2}, \qquad z=(z_1,z_2) \in  Z^{\mathbf{d}}.
\end{split}
\] 
\subsection{Wavelet Classes and the Intrinsic Square Function} \label{ss:wcisf2p}This is the bi-parameter analogue of Section \ref{ss:21}. The notation now overrides what was previously laid out in the one parameter case; unfortunately, due to the different dilation structure and working with borderline sufficient decay in each of the parameters, definitions may not be re-used.

For $\nu=(\nu_1,\nu_2)\in (0,\infty)^2$,
and $\delta>0$, define $C_{\star,\nu,\delta}$ to be the subspace of the $\delta$-H\"older continuous functions on $\R^{\mathbf{d}}$ whose norm 
\[
\|\phi\|_{\star,\nu,\delta}= \sup_{x\in \R^{\mathbf{d}}}  \left( \prod_{j=1,2}\langle x_j \rangle^{d_j+\nu_j} \right) |\phi(x)|  + \sup_{x \in \R^{\mathbf{d}}}   \sup_{\substack{ h\in \R^{\mathbf{d}} \\ 0<|h|\leq 1}} \left( \prod_{j=1,2}\langle x_j \rangle^{d_j+\nu_j} \right)  \frac{\left|\phi(x+h) - \phi(x)\right|}{|h|^{\delta}}
\] is finite.
In the bi-parameter case, the relevant cancellation properties of $\psi$ are encoded by requiring \eqref{e:ck1par} to hold in the variable $x_\iota$ for each slice $\psi^{[{\widehat\iota},x_{{\widehat\iota}}]}$ and each $\iota=1,2$.  This necessitates the introduction of the bi-parameter analogue of the classes $\Psi^{k,\delta;\iota}_z$. Hereafter, $\gamma_j\in \mathbb N^{d_j}$ for either $j=1,2$  is a multi-index on $\R^{d_j}$. For  $k=(k_1,k_2)\in \mathbb N^{2}, $   $0<\delta\leq 1$ define
\[\begin{split}
&
\Psi^{k,\delta;1,1}_z\coloneqq\left\{\phi \in \mathcal {S}(\R^{\mathbf{d}}):  s_1^{|\gamma_1|}s_2^{|\gamma_2| }\left\| (\mathsf{Sy}_{z})^{-1} \partial^{\gamma_1} \partial^{\gamma_2}  \phi \right\|_{\star,k+\delta,\delta} \leq 1 \right\},
\\
&  \Psi^{k,\delta;0,1}_z\coloneqq\left\{\phi \in \Psi^{k,\delta;1,1}_z: \textrm{\eqref{e:ck1par} holds for } \psi=\psi^{[2,x_2]  },  d=d_1,\,\gamma=\gamma_1,\, \forall x_2 \in \R^{d_2}, \forall 0 \leq |\gamma_1| \leq k_1 \right\},
\\
&  \Psi^{k,\delta;1,0}_z\coloneqq\left\{\phi \in\Psi^{k,\delta;1,1}_z: \textrm{\eqref{e:ck1par} holds for } \psi=\psi^{[1,x_1]  },  d=d_2,\,\gamma=\gamma_2,\, \forall x_1 \in \R^{d_1}, \forall 0 \leq |\gamma_2| \leq k_2 \right\},
\\
&  \Psi^{k,\delta;1,1}_z\coloneqq \Psi^{k,\delta;0,1}_z\cap  \Psi^{k,\delta;1,0}_z,
\end{split}\]
where $z=(z_1,z_2)=((x_1,s_1), (x_2,s_2)) \in Z^{\mathbf{d}}$. The  resulting  decay, smoothness and cancellation properties satisfied by  $\phi\in \Psi^{k,\delta;\theta_1,\theta_2}_{z}$ are efficiently described by
\begin{align}
\label{e:bip1} &
s_\iota^{d_{\iota}+|\gamma_\iota|} \left\langle \frac{y_\iota - x_\iota }{s_\iota} \right\rangle^{d_\iota+k_\iota+\delta}[\partial^{\gamma_\iota} \phi]^{[\iota,y_\iota]} \in \Psi^{k_{{\widehat\iota}},\delta;\theta_{{\widehat\iota}}}_{z_{{\widehat\iota}}} \qquad \forall y_\iota\in \R^{d_\iota}, \,0\leq |\gamma_\iota|\leq k_\iota, \, \iota \in \{1,2\}.
\end{align}

The analogue of the almost orthogonality  Lemma \ref{e:Delta1} in the bi-parameter setting is the following lemma.
\begin{lemma} 
\label{l:ttstar2} Let $m\in\mathbb N$, $0<2\eta<\delta\leq 1$,  $z,\zeta\in Z^{\mathbf{d}}$,  $\psi_z \in \Psi^{(2m,2m),\delta; 0,0}_{z}$, $\psi_{\zeta} \in \Psi^{(2m,2m),\delta; 0,0}_{\zeta}$ . Then
\[
\left|\langle \psi_z, \psi_\zeta  \rangle \right| \lesssim_{m,\eta} [z_1,\zeta_1]_{m+\eta} [z_2,\zeta_2]_{m+\eta} .
\]
\end{lemma}
\begin{proof}
Let $\iota$ be either $1$ or $2$. Applying \eqref{e:bip1} together with the first estimate of Lemma \ref{l:ttstar1} and integrating,
\[ 
\begin{split} \left|
\langle \psi_z, \psi_\zeta  \rangle\right|   =   \int_{\R^{d_\iota}}  \left|\langle \phi_z^{[\iota,y_\iota]} ,\psi_\zeta^{[\iota, y_\iota]}  \rangle\right| \, \d y_{\iota} \lesssim  \left[\max\left\{s_\iota,\sigma_\iota,|x_\iota-\xi_\iota|\right\}\right]^{-d_\iota} [z_{\hat\iota},\zeta_{\hat\iota}]_{2m+ 2\eta} 
\end{split}
\]
 and the lemma follows by taking the $1/2$-geometric average of the two inequalities.
\end{proof}
\subsection{Intrinsic Bi-Parameter Wavelet Coefficients}
The definition of the intrinsic bi-parame\-ter wavelet coefficients is next given. These may be defined in the generality of $f\in \mathcal S'(\R^{\mathbf{d}})$. For such $f$, and $z=(z_1,z_2) \in Z^{\mathbf{d}}$, set \begin{equation}
\label{e:Psidelta2}
\Psi^{\delta;(\iota_1,\iota_2)}_z f  =  \sup_{\psi \in \Psi^{\delta;\iota_1,\iota_2}_z} |\langle f,\psi\rangle |.  \end{equation}
A standard argument based on \eqref{e:bip1} shows that if $f\in L^1_{\mathrm{loc}}(\R^{\mathbf{d}})$, 
\[
\Psi^{\delta;(1,1)}_{(x_1,s_1),(x_2,s_2)} f \lesssim_\delta   \mathrm{M}_{d_1,d_2} f (x), \qquad x=(x_1,x_2)\in \R^{\mathbf{d}}, \quad s_1,\, s_2>0
\]
where $ \mathrm{M}_{d_1,d_2}  $ is the bi-parameter maximal function on $\R^{\mathbf{d}}$. In particular the wavelet coefficients of $f\in L^p(\R^{\mathbf{d}}) $ are finite for $f\in L^p(\R^{\mathbf d})$, $p>1$, as $ \mathrm{M}_{d_1,d_2}  f$ is finite a.\ e.\ in that case.
The remainder of this section contains a basic $L^2$ estimate for the intrinsic square function
\begin{equation}
\label{e:Sdelta2}
\mathrm{SS}_{\delta} f(x_1,x_2) = \left(\, \int\displaylimits_{(0,\infty)^2} 
\left[\Psi^{\delta;(0,0)}_{(x_1,t_1), (x_2,t_2)} f \right]^2 
  \frac{\d t_1 \d t_2}{t_1 t_2}\right)^{\frac12}.
\end{equation}
Again, the parameter $\delta$ will be fixed and play no role and the operator will be represented as $\mathrm{SS}$ later in the paper.
\begin{proposition} 
\label{p:intrinsicsf2}  $\|{\mathrm{SS_{\delta}}} f\|_2\lesssim_\delta \|f\|_2$. As a consequence,
\begin{equation}
\label{e:intrinsicsf2}
\int_{Z^{\mathbf{d}}}\left[ \Psi^{\delta;(0,0)}_{z} f\right] \left[\Psi^{\delta;(0,0)}_{z} g\right]\,  \d \mu (z) \lesssim \|f\|_2\|g\|_2.
\end{equation}
\end{proposition}

\begin{proof} Notice that \eqref{e:intrinsicsf2} follows from the square function estimate via two application of Cauchy-Schwarz inequality.
The argument for the square function estimate   is analogous to the one employed for \eqref{e:ttstar1}. Working with $f\in L^2(\R^d)$ of unit norm, and fixing  $\psi_z \in   \Psi^{\delta;0,0}_z, z\in Z^{\mathbf d} $, it suffices to estimate
\[ \label{e:ttstar1bip}
\begin{split} &\quad 
\int\displaylimits_{ z\in Z^{\mathbf{d}}}  \int\displaylimits_{ \substack{(\alpha_1, \beta_1)\in  \R^{d_1} \times (0,1) \\ (\alpha_2, \beta_2)\in  \R^{d_2} \times (0,1)} }|\langle f, \psi_{z} \rangle| |\langle s_1^{d_1}s_2^{d_2} \psi_{z}, \psi_{\zeta(\alpha_1,\beta_1, \alpha_2,\beta_2)}\rangle| |\langle f, \psi_{\zeta(\alpha_1,\beta_1, \alpha_2,\beta_2)}\rangle| \,\d \mu(z)\frac{ \d\alpha_1 \d\beta_1\d\alpha_2 \d\beta_2}{\beta_1  \beta_2 }\\ &  \lesssim    \int\displaylimits_{z\in Z^{\mathbf d}} |\langle f, \psi_z \rangle|^2  \,{\d\mu(z)} \end{split}\]
as well as three more integrals covering all possible  relationships between the scales of $z=((x_1,s_1),(x_2,s_2))$
and $\zeta(\alpha_1,\beta_1, \alpha_2,\beta_2)=((x_1+\alpha_1 s_1,\beta_1 s_1), (x_2+\alpha_2 s_2,\beta_2 s_2))$, which are estimated in an analogous fashion. In this specific case, Lemma \ref{l:ttstar2} entails the bound
\[
|\langle s_1^{d_1}s_2^{d_2} \psi_{z}, \psi_{\zeta(\alpha_1,\beta_1, \alpha_2,\beta_2)}\rangle| \lesssim  [(\alpha_1,\beta_1)]_{\frac \delta 4}[(\alpha_2,\beta_2)]_{\frac \delta 4}
\]
and the required control is again obtained  via a combination of two instances of  \eqref{e:integbrack} and Cauchy-Schwarz.
\end{proof}
The $L^p$-theory of double square functions is well studied. On the other hand, working with non-compactly supported, non-tensor product wavelets, is non-standard. In this generality, $L^p$-estimates may be obtained  by direct product John-Nirenberg type arguments  involving Journ\'e's lemma. In this article, for reasons of space, $L^p$-bounds are obtained as a particular case of the sharp quantitative bound of Theorem \ref{t:w2} below.

 \section{Wavelet Representation of Bi-Parameter Calder\'on-Zygmund Operators}
\label{s6}
 Throughout this section  $\Lambda$ indicates a generic bilinear continuous form on $\mathcal S( \R^{\mathbf{d}})\times \mathcal S( \R^{\mathbf{d}})$. If $f_j\in \mathcal S(\R^{d_j})$ for $j=1,2$, then $f_1\otimes f_2\in \mathcal S( \R^{\mathbf{d}})$ stands for $(x_1,x_2)\mapsto f_1(x_1) f_2(x_2)$. Define the full adjoint of $\Lambda$ by
\[
\Lambda^\star: 
\mathcal S(\R^{\mathbf{d}}) \times \mathcal S(\R^{\mathbf{d}})\to \mathbb C,\qquad 
\Lambda^\star (f,g) \coloneqq \overline{\Lambda(g,f)},
\] 
and, when $\Lambda$ acts on tensor products, the partial adjoints are given by
\[
\begin{split}
&\Lambda^{\star_1},\Lambda^{\star_2}: 
\mathcal S(\R^{d_1}) \otimes  \mathcal S(\R^{d_2}) \times \mathcal S(\R^{d_1}) \otimes  \mathcal S(\R^{d_2})\to \mathbb C,\qquad 
\\
&\Lambda^{\star_1}(f_1\otimes f_2, g_1\otimes g_2) \coloneqq  \overline{\Lambda(g_1\otimes f_2, f_1\otimes g_2) }, \qquad 
\Lambda^{\star_2}(f_1\otimes f_2, g_1\otimes g_2) \coloneqq  \overline{\Lambda(f_1\otimes g_2, g_1\otimes f_2) }.
\end{split}\]
To unify, we write $\Lambda^\circ =\Lambda$ and $\Lambda^{\mathfrak{a}}$, with $\af$ varying in the set $\vaf=\{\circ,\star,\star_1,\star_2\}$  
At times, the adjoint linear operators $T^\af : \mathcal S( \R^{\mathbf{d}})\to \mathcal S'( \R^{\mathbf{d}})$,
\[
\langle T^\af (f_1\otimes f_2), g_1\otimes g_2\rangle \coloneqq \Lambda^\af (f_1\otimes f_2, g_1\otimes g_2), \qquad \af \in \vaf
\]
will be considered. 
A bi-parameter wavelet basis is needed. For $j=1,2$ let $\varphi_j\in \mathcal C^\infty_0(\R^{d_j} ) $ be such that $\varphi=\varphi_j$, $d=d_j$ in \eqref{e:mw1}, and  $D\geq 8(d_1+d_2)$ sufficiently large. 
Set $\varphi\coloneqq \varphi_1 \otimes \varphi_2\in \mathcal C^\infty_0(\R^{\mathbf{d}} )$ as our mother wavelet on $\R^{\mathbf{d}}$, and  rescale it by
\begin{equation}
\label{e:varphi2p}
\varphi_z=\varphi_{z_1} \otimes  \varphi_{z_2}\coloneqq  \Sy_z \varphi = \Sy_{z_1} \varphi_1 \otimes \Sy_{z_2} \varphi_2, \qquad z =(z_1,z_2)\in Z^{\mathbf{d}}.
\end{equation}
With this position,  $\varphi_{z}\in c\Psi^{(D,D),1; 0}_z$ for all $z\in Z^{\mathbf{d}}$.
\subsection{Bi-parameter singular integrals} The boundedness   and kernel smoothness properties of bi-parameter singulars are quantified by the parameters  $k=(k_1,k_2)\in \mathbb N^2$ and $\delta>0$, and summarized by the norm
  \begin{equation}
\label{e:SInorm2p0}
\|\Lambda\|_{\mathrm{SI}(\R^{\mathbf{d}},k,\delta)}\coloneqq \|\Lambda\|_{\mathrm{PWB},k,\delta
} + \|\Lambda\|_{\mathrm{K,f},k,\delta}
\end{equation}
whose summands are described below. Notice that the norm \eqref{e:SInorm2p0} is stable under full and partial adjoints: this fact will be used without explicit mention from  now on.
\begin{definition}[Partial kernel and weak boundedness] \label{d:part1} For $j=1,2$,   $z_j\in Z^{d_j}$, and $u_{z_j}, v_{z_j} \in \Psi^{k_j,\delta;1}_{z_j}$ with $\supp\, u_{z_j}, \supp\, v_{z_j} \subset \mathsf{B}_{z_j}$, define the forms
\[
\begin{split} &
\Lambda_{1,u_{z_1}, v_{z_1}}: \mathcal{S}(\R^{d_2}) \times\mathcal{S}(\R^{d_2}) \to \mathbb C, \qquad \Lambda_{1,u_{z_1}, v_{z_1}}(f_2,g_2) = s_1^{d_1}\Lambda(u_{z_1} \otimes f_{2},v_{z_1} \otimes g_{2}),
\\
&
\Lambda_{2,u_{z_2}, v_{z_2}}: \mathcal{S}(\R^{d_1}) \times\mathcal{S}(\R^{d_1}) \to \mathbb C, \qquad \Lambda_{2,u_{z_2}, v_{z_2}}(f_1,g_1) = s_2^{d_2}\Lambda( f_{1} \otimes u_{z_2},g_{1} \otimes v_{z_2}).
 \end{split}
\]
The form $\Lambda$ has the \textit{$(k,\delta)$-partial kernel and weak boundedness} properties if there exists $C>0$ such that 
\[
\left\|\Lambda_{1,u_{z_1}, v_{z_1}} \right\|_{\mathrm{WB},\delta} + \left\|\Lambda_{1,u_{z_2}, v_{z_2}} \right\|_{\mathrm{K},k_2,\delta} + \left\|\Lambda_{2,u_{z_2}, v_{z_2}} \right\|_{\mathrm{WB},\delta} + \left\|\Lambda_{2,u_{z_2}, v_{z_2}} \right\|_{\mathrm{K},k_1,\delta} \leq C 
\]
uniformly over $z_j\in Z^{d_j}$ and $u_{z_j}, v_{z_j} \in \Psi^{k_j,\delta;1}_{z_j}$ with $\supp\, u_{z_j}, \supp\, v_{z_j} \subset \mathsf{B}_{z_j}$, $j=1,2$. In this case, $\|\Lambda\|_{\mathrm{PWB},k,\delta}$ is the least such constant $C $. The   $\|\Lambda\|_{\mathrm{PWB},\delta}$  norm is stable under full and partial adjoints, and subsumes all of weak boundedness and partial kernel assumptions of \cite{MK1}, see also \cite{OuMp}.
\end{definition}
\begin{definition}[Full kernel] \label{d:full} 
   For a function $K=K(u,v)$ on  $\R^{\mathbf{d}}\times \R^{\mathbf{d}}$, with $u=(u_1,u_2), v=(v_1,v_2)$ again using the finite difference notation
\[
\begin{split} &
\Delta^{1}_{ h_1|\cdot} K(u,v) \coloneqq K\left((u_1+h_1, u_2),v \right)-  K(u,v), \qquad \Delta^{2}_{ h_2|\cdot} K(u,v) \coloneqq K\left((u_1 , u_2+h_2),v \right)-  K(u,v),\\
&\Delta^{1}_{\cdot| h_1} K(u,v) \coloneqq K\left(u,(v_1+h_1,v_2) \right)-  K(u,v), \qquad \Delta^{2}_{\cdot| h_2} K(u,v) \coloneqq K\left(u,(v_1,v_2+h_2) \right)-  K(u,v),
\end{split}\]
for  $u=(u_1,u_2),\, v=(v_1,v_2) \in \R^{\mathbf{d}}$, $ h_j\in \R^{d_j}, j=1,2$. In preparation for \eqref{e:full}, introduce   the norms 
\[
\begin{split}
& 
\left\| K \right\|_{\kappa }\coloneqq \sup_{(u,v)\in\R^{\mathbf{d}}\times \R^{\mathbf{d}}}  \left(\prod_{j=1,2}  |u_j-v_j|^{d_j+\kappa_j } \right)  \left|   K(u,v)\right|, \\
& 
\left\| K \right\|_{\kappa,\delta, \Delta^{1}_{ \square|\cdot} }\coloneqq \sup_{(u,v)\in\R^{\mathbf{d}}\times \R^{\mathbf{d}}}  \sup_{0<2|h_1|< |u_1-v_1| } \frac{|u_1-v_1|^{d_1+\kappa_1 +\delta}}{|h_1|^{\delta}} |u_2-v_2|^{d_2+\kappa_2} \left| (\Delta^{1}_{ h_1|\cdot} K)(u,v)\right|, 
\\  
& 
\left\| K \right\|_{\kappa,\delta, \Delta^{1}_{ \square|\cdot} \Delta^{2}_{ \square|\cdot} }\coloneqq \sup_{(u,v)\in\R^{\mathbf{d}}\times \R^{\mathbf{d}}}  \sup_{\substack{0<2|h_j|< |u_j-v_j|  \\ j=1,2 }}\left(\prod_{j=1,2} \frac{|u_j-v_j|^{d_j+\kappa_j +\delta}}{|h_j|^{\delta}} \right)  \left| (\Delta^{1}_{ h_1|\cdot} \Delta^{2}_{ h_2|\cdot} K)(u,v)\right|, 
\end{split}
\]
with similar definitions for the other finite difference operators: here $\kappa=(\kappa_1,\kappa_2)\in [0,\infty)^{2}$ and $\delta>0$.
 Then the form 
$\Lambda$ satisfies the \textit{full kernel estimates} if the following holds. There exists a   $k=(k_1,k_2)$-times continuously differentiable  $K(u,v)$ on $\R^{\mathbf{d}} \times \R^{\mathbf{d}}$ such that
\begin{equation}
\label{e:full1}
\Lambda(f_1 \otimes f_2,g_1 \otimes g_2) = \int_{\R^{\mathbf{d}} \times \R^{\mathbf{d}}} K(u,v)  f_1(u_1)f_2(u_2) g_1(v_1) g_2(v_2) \,\d u \d v\end{equation}
for all tuples $f_j,g_j\in \mathcal S(\R^{d_j})$ such that $\mathrm{supp}\, f_j \cap \mathrm{supp}\, g_j = \varnothing$, $j=1,2$, with the property that  for all   $0\leq \kappa_1 \leq k_1, \; 0\leq \kappa_2 \leq k_2$,
\begin{equation}
\label{e:full}
\begin{split}
   \left\|\nabla_{u_1}^{\kappa_1 }\nabla_{u_2}^{\kappa_2} K\right\|_{(\kappa_1,\kappa_2)}+ \left\|\nabla_{v_1}^{\kappa_1 }\nabla_{v_2}^{\kappa_2} K\right\|_{(\kappa_1,\kappa_2)} &\leq {C}, \\
  \left\| \nabla_{u_1}^{k_1}\nabla_{u_2}^{\kappa_2}K  \right\|_{(k_1,\kappa_2),\delta, \Delta^{1}_{\square |\cdot} } 
+ \left\| \nabla_{v_1}^{k_1}\nabla_{v_2}^{\kappa_2}K  \right\|_{(k_1,\kappa_2), \delta, \Delta^{1}_{\cdot|\square} }& \leq C,
\\  \left\| \nabla_{u_1}^{k_1}\nabla_{u_2}^{\kappa_2}K  \right\|_{(\kappa_1,k_2),\delta, \Delta^{2}_{\square |\cdot} } 
+ \left\| \nabla_{v_1}^{k_1}\nabla_{v_2}^{\kappa_2}K  \right\|_{(\kappa_1,k_2),\delta, \Delta^{2}_{\cdot|\square} } & \leq C,
   \\
    \left\| \nabla_{u_1}^{k_1}\nabla_{u_2}^{k_2} K  \right\|_{(k_1,k_2), \delta, \Delta^{1}_{\square|\cdot}\Delta^{2}_{\square|\cdot}} +
\left\|\nabla_{v_1}^{k_1}\nabla_{v_2}^{k_2} K  \right\|_{(k_1,k_2), \delta, \Delta^{1}_{\cdot|\square}\Delta^{2}_{\cdot|\square} }  &  \leq C   ,
 \\  \left\| \nabla_{u_1}^{k_1}\nabla_{v_2}^{k_2} K  \right\|_{(k_1,k_2),\delta,\Delta^{1}_{\square|\cdot}\Delta^{2}_{\cdot|\square} } 
+\left\|\nabla_{v_1}^{k_1}\nabla_{u_2}^{k_2} K  \right\|_{(k_1,k_2),\delta, \Delta^{1}_{\cdot|\square}\Delta^{2}_{ \square|\cdot} }   & \leq C   .
\end{split}
\end{equation}
 The least $C$ such that \eqref{e:full1} and \eqref{e:full} hold for each $\Upsilon \in  {\bf\Lambda}$  will be denoted by $\|\Lambda\|_{\mathrm{K,f},k,\delta}$. Notice that the latter constant is preserved under full and partial adjoints as well. If $k_1=k_2=0$, these are the usual full kernel estimates of a bi-parameter Calder\'on-Zygmund operator, see for example \cite{MK1}. 
\end{definition}
Say that $ \Lambda\in \mathrm{SI}(\R^{\mathbf d},k,\delta)$ if the constant \eqref{e:SInorm2p0} is finite.  The next definitions establish the main examples of forms in  $\mathrm{SI}(\R^{\mathbf d},k,\delta)$.
 \begin{definition}[Bi-parameter wavelet form] \label{c:cancfbip} Associate to the two families of bi-parameter cancellative wavelets $\left\{\beta_z,\upsilon_{z}\in C\Psi_{z}^{k,\delta;0,0}:z\in Z^{\mathbf{d}}\right\} $ the bi-parameter wavelet form
\begin{equation}
\label{e:cancformbip}
\Lambda(f,g) =  \int_{Z^{\mathbf{d}}}     \langle f,  \beta_{z}\rangle  \langle \upsilon_{z},g \rangle  \, \d \mu(z).
\end{equation}
This form belongs to $ \mathrm{SI}(\R^{\mathbf{d}},k,\delta)$ and $\|\Lambda\|_{ \mathrm{SI}(\R^{\mathbf{d}},k,\delta)}\lesssim 1$. The weak boundedness property is a particular case of estimate \eqref{e:intrinsicsf2} from Proposition \ref{p:intrinsicsf2}. The partial kernel and full kernel estimates may be obtained  by repeatedly employing \eqref{e:bip1} and standard computations.  
 \end{definition}
{
\begin{definition}[Bi-parameter paraproduct forms]  \label{d:6p} Paralleling our treatment of the one parameter case, three types of bi-parameter paraproducts will be defined.   
For a pair of multi-indices $(\gamma_1,\gamma_2)$ on $\R^{d_1} \times \R^{d_2}$, a function $b\in \mathrm{BMO}(\R^{\mathbf{d}})$, $ (\iota_1,\iota_2)\in \{0,1\}^2$  and $f,g\in \mathcal S(\R^{\mathbf{d}})$, define
d\begin{equation}
\label{e:paragammas}
\begin{split}
&
\Pi_{(0,0),b,(\gamma_1,\gamma_2)} (f, g) \coloneqq \int_{Z^{\mathbf{d} }} \langle b,\varphi_{\gamma_1,z_1}\otimes \varphi_{\gamma_2,z_2}\rangle  \langle f,\vartheta_{\gamma_1,z_1} \otimes \vartheta_{\gamma_2,z_2}\rangle  \langle \varphi_z, g \rangle \, \d \mu(z),
\\& \Pi_{(0,1),b, (\gamma_1,\gamma_2)} (f, g) \coloneqq \int_{Z^{\mathbf{d} }} \langle b,\varphi_{\gamma_1,z_1}\otimes \varphi_{\gamma_2,z_2}\rangle  \langle f,\vartheta_{\gamma_1,z_1} \otimes \varphi_{z_2}  \rangle  \langle \varphi_{z_1}\otimes \vartheta_{\gamma_2,z_2}, g \rangle \, \d \mu(z),
\end{split}\end{equation}
where $\varphi_{\gamma_j,z_j} =\Sy_{z_j} \partial^{-\gamma_j}\varphi_{j}$, see \eqref{e:ck1paral} and \eqref{e:phigammaz}, and  the family
  $\{\vartheta_{\gamma_j,z_j}\in C \Psi^{D,1;0}_z:z\in Z^{d_j}\}$ is a $\gamma_j$-family as in \eqref{e:alpert3}.
 The first form is usually termed \emph{full paraproduct}, while the second is usually referred to as a \emph{partial  paraproduct}. 
The paraproducts defined below are related  as partial adjoints, namely
\[
\Pi_{(0,1),b,(\gamma_1,\gamma_2)} = \left(\Pi_{(0,1),b,(\gamma_1,\gamma_2)}\right)^{\star_1}.
\] For this reason,  the notation $\Pi_{ b,\gamma} $  will  be used in place  of  $\Pi_{(0,0),b,\gamma} $ throughout this section.
  Standard computations relying on the smoothness and compact support of the wavelets appearing in \eqref{e:paragammas}  lead to the following controls on the partial kernel and weak boundedness, and full kernel constants of the paraproduct forms: for   any multi-index $\gamma=(\gamma_1,\gamma_2)$, there holds 
 \begin{equation}
\label{e:paragammas2}
\left\|\Pi_{ b, \gamma}  \right\|_{ \mathrm{SI}(\R^{\mathbf{d}},k,\delta)} \lesssim_{k} \|b\|_{\mathrm{BMO}(\R^{\mathbf{d}})}.
\end{equation}

A third family of paraproducts, which are termed \emph{half-paraproducts}, are constructed using the definitions \eqref{e:para1p} in each parameter $\iota\in \{1,2\}$.
 Let $\kappa_{\hat \iota}\in \mathbb N,\eta>0$, which are kept implicit in the notation, $ \gamma_{\iota}$ be a multi-index on $\R^{d_{\iota}}$ and $\mathbf{a}$ be a continuous  map on $ Z^{d_{\hat\iota}}\times Z^{d_{\hat \iota}}$ taking values in  $\mathrm{BMO} (\R^{d_{\iota}})$. Define, a priori on $\mathcal S(\R^{\mathbf{d}}) \times  \mathcal S(\R^{\mathbf{d}})$, the form
\begin{equation}
\label{e:hpp}
 \Pi^\iota_{\mathbf{a},\gamma_{\iota},\alpha_\iota} (f,g) =\int\displaylimits_{Z^{d_{\hat \iota}}\times Z^{d_{\hat \iota}}}  \Pi_{\mathbf{a}(z_{\hat \iota},\zeta_{\hat \iota}),\gamma_{\iota}, \alpha_\iota}\left(\langle f,\varphi_{z_{\hat \iota}}\rangle_{\hat \iota},\langle g,\varphi_{\zeta_{\hat \iota}}\rangle_{\hat \iota}\right) [z_{\hat \iota},\zeta_{\hat \iota}]_{\kappa_{\hat \iota}+\eta} \d\mu(z_{\hat \iota}) \d\mu(\zeta_{\hat \iota}) 
\end{equation} where $\Pi_{\mathbf{a}(z_{\hat \iota},\zeta_{\hat \iota}),\gamma_{\iota},\alpha_\iota}$ refers to \eqref{e:para1p} for $b=\mathbf{a}(z_{\hat \iota},\zeta_{\hat \iota}),\gamma=\gamma_\iota,\alpha=\alpha_\iota, d=d_\iota$.  
Arguing in a similar fashion to \eqref{e:paragammas2}, we record the estimate
 \begin{equation}
\label{e:paragammas3}
\left\| \Pi^\iota_{\mathbf{a},\gamma_{\iota},\alpha_\iota} \right\|_{ \mathrm{SI}(\R^{\mathbf{d}},k,\delta)} \lesssim \sup_{(z_{\hat \iota}, \zeta_{\hat \iota}) \in Z^{d_{\hat\iota}}\times Z^{d_{\hat \iota}}}\|\mathbf{a}(z_{\hat\iota}, \zeta_{\hat\iota})\|_{\mathrm{BMO}(\R^{d_{\iota}})}, \qquad k=(k_1,k_2), \, k_{\iota}= |\gamma_\iota|, \,k_{\hat\iota}=\kappa_{\hat\iota}. 
\end{equation}
Furthermore, it is a particularly useful observation that, in view of \eqref{e:equalitiespar} and referring to the notation introduced in \eqref{e:Tlambda}, which is legitimately used whenever $f_2,g_2\in\mathcal S(\R^{d_2})$ are fixed
\begin{equation}
\label{e:eequal2par}
\Pi^1_{\mathbf{a},\gamma_{1},\alpha_1} (x_1^{\beta_1} \otimes f_2, g_1\otimes g_2 ) =0 \qquad  \forall 0\leq |\beta_1|\leq |\gamma_1|, \; \beta_1 \neq \gamma_1 
\end{equation}
in the sense of linear functionals acting on  $g_1\in\mathcal S_{D}(\R^{d_1})$, 
and similarly for adjoints and half-paraproducts in the second parameter.
\end{definition}
The next technical lemma shows that the norm \eqref{e:SInorm2p0} controls certain symbols  obtained by partial testing of $\Lambda$ against monomials. For this, the modified wavelets
\[ \chi_{z,\zeta} = \varphi_z -P_{z,\zeta}\cic{1}_{ A(\zeta)}(z), \qquad z,\zeta \in Z^{d}
\]
introduced in \eqref{e:U} will be needed. The proof of the  next Lemma is postponed to the end of this Section, see Subsection \ref{ss:pflqs}.
 \begin{lemma} \label{l:qs}
Let $\iota \in \{1,2\}$, $\gamma_\iota$ be a multi-index in $\R^{d_\iota}$ and $k=(k_1,k_2)$ be such that $|\gamma_\iota|\leq k_\iota$.  For $\iota \in \{1,2\}$, $z_{\hat \iota},\zeta_{\hat \iota} \in Z^{d_{\hat \iota}}$, multi-indices $\gamma_\iota$ in $\R^{d_\iota}$, and $\af\in \{\circ, \star\}$ define the functionals $  \mathbf{q}^{\iota,\af}_{\gamma_\iota} [\Lambda]$ by
\begin{equation}
\label{e:qs}
\begin{split}
&     \big\langle \mathbf{q}^{1,\af}_{ \gamma_1}[\Lambda](z_2,\zeta_2), f_1 \big\rangle \coloneqq  
 \Lambda^{\af}(x_1^{\gamma_1} \otimes \chi_{z_2, \zeta_2},|\nabla|^{|\gamma_1|}  f_1 \otimes \chi_{  \zeta_2,z_2} ),\\[1mm]
&   \big\langle \mathbf{q}^{2,\af}_{ \gamma_2} [\Lambda](z_1,\zeta_1), f_2\big\rangle \coloneqq   \Lambda^{\af}(\chi_{z_1,  \zeta_1}\otimes x_2^{\gamma_2}  ,\chi_{\zeta_1,  z_1}\otimes  |\nabla|^{|\gamma_2|} f_2 )
\end{split}
\end{equation}
initially acting  on the subspace of $\mathcal S(\R^{d_\iota})$ of \ functions $f_\iota$ whose frequency support does not contain the origin. For    a multi-index $\alpha_\iota$ in $\R^{d_\iota}$, and $0<\eta<\delta$, also define
\begin{equation}
\label{e:as}
  \mathbf{a}^{\iota,\af}_{ \gamma_\iota,\alpha_\iota}[\Lambda](z_{\hat \iota},\zeta_{\hat \iota})\coloneqq [z_{\hat \iota},\zeta_{\hat \iota}]_{k_{\hat \iota}+\eta}^{-1}R^{\alpha_\iota} \mathbf{q}^{\iota,\af}_{ \gamma_\iota}(z_{\hat \iota},\zeta_{\hat \iota}), \qquad 
\end{equation}
where $R^{\alpha_1}$ is the Riesz transform associated to $\alpha_1$.
  Assume that  $\mathbf{q}^{\iota,\af}_{\gamma'_\iota} \equiv 0$  for all multi-indices on $\R^{d_\iota}$ with $|\gamma'_\iota|<|\gamma_\iota|$.  Then  
 \[\displaystyle \sup_
 {  Z^{d_{\hat \iota}}\times Z^{d_{\hat \iota}} }  \|\mathbf{a}^{\iota,\af}_{\gamma_\iota,\alpha_\iota}  [\Lambda]\|_{\mathrm{BMO}(\R^{d_{\hat \iota}})} \lesssim \left\|\Lambda  \right\|_{ \mathrm{SI}(\R^{\mathbf{d}},k,\delta) } . \]
 \end{lemma}

\subsection{Bi-Parameter Calder\'on-Zygmund forms of class $(k,\delta)$}     If $\alpha_\iota$ is a multi-index on $\R^{d_\iota}$, the notation  $p_R^{\alpha_\iota}$ refers to $p_R^{\alpha}$ from \eqref{e:pR},  with $d=d_\iota, \alpha=\alpha_\iota $. Below, $\mathcal S_{D}(\R^{\mathbf{d}}) $ be the subspace of functions   $\psi \in \mathcal S(\R^d)$ with the same  bi-parameter vanishing moment property of the functions of $\Psi^{D,\delta; 0,0 }_z$, for some $z\in Z^{\mathbf d} $.
Then, if $0\leq |\alpha_\iota|\leq k_\iota$ for $\iota\in \{1,2\}$,  and $ \Lambda \in \mathrm{SI}(\R^{\mathbf{d}},k,\delta)$, and $\af \in \vaf$,  the limits 
\begin{equation}
\label{e:Tlambdabip}
\Lambda^{\af}(  x^{\alpha_1}\otimes x^{\alpha_2}, \psi)  =  \lim_{R\to \infty}\Lambda\left(p_{ R}^{\alpha_1}\otimes p_{ R}^{\alpha_2},\psi\right), \qquad  \psi \in \mathcal S_{D}(\R^{\mathbf{d}})
\end{equation}
exist  and define  linear continuous functionals on $\mathcal S_{D}(\R^d)$: with the full kernel estimates at one's disposal,  the proof presented in \cite[Lemma 1.91]{FTW} extends to the bi-parameter case without essential changes.
\begin{remark} \label{r:indbp} If $ \Lambda \in \mathrm{SI}(\R^{\mathbf d},k,\delta)$ is a wavelet form of the type \eqref{e:cancformbip}, an immediate  byproduct of  the cancellation properties of the families $ \beta_z,\upsilon_{z}\in C\Psi_{z}^{k,\delta;0,0}$ is that the  functionals \[\Lambda^\af(x^{\alpha_1}\otimes x^{\alpha_2}, \cdot), \qquad \mathbf{a}^{\iota,\af}_{ \gamma_\iota,\alpha_\iota}[\Lambda]\]  vanish for all $0\leq |\alpha_\iota|,|\gamma_\iota|\leq k_\iota$, $\iota =1,2$.
\end{remark}
The next and final definition details   our assumptions on  the functionals \eqref{e:Tlambdabip} associated to  $ \Lambda \in \mathrm{SI}(\R^{\mathbf d},k,\delta)$. We ask whether  $ \Lambda \in \mathrm{SI}(\R^{\mathbf d},k,\delta)$ has paraproducts of order $\kappa$ for $ 0\leq \kappa \leq \min\{k_1,k_2\}$ and, if that is the case, at the same time define the $\kappa$-th order cancellative part of $\Lambda$ for all $0\leq \kappa \leq \max\{k_1,k_2\}$. 
\begin{definition}[Paraproducts of order $0\leq \kappa \leq \min\{k_1,k_2\}$] 
\label{d:para2} Say that $ \Lambda \in \mathrm{SI}(\R^{\mathbf d},k,\delta)$  has  \textit{paraproducts of order $0$} if for each $\af\in \vaf$ there exists $b_{0}^{\af}\in \mathrm{BMO}(\R^{\mathbf{d}})$, the BMO product space, such that \[
\Lambda^\af (\cic{1} \otimes \cic{1}, \psi) = \langle b_{0}^{\af}, \psi \rangle \qquad
\forall \psi \in \mathcal S_D(\R^{\mathbf d}).\]
As customary, we  use the $T(1)$ notation and write
$
b_{0}^{\af} = T^\af (\cic{1} \otimes \cic{1})
$.
If $\Lambda$ has paraproducts of order $0$ 
the \text{$0$-th order cancellative part of} $\Lambda$ is given by
\[
\Lambda_{0}(f,g) \coloneqq \Lambda(f,g)  -\sum_{\af \in \vaf } \left[\Pi_{ b_{0}^{\af}}\right]^{\af} (f,g) -\sum_{\substack{\af \in \{\circ, \star\}\\ \iota \in \{1,2\}}}
 \left[ \Pi^\iota_{\mathbf{a}^{\iota,\af}_{0,0}[\Lambda]}\right]^{\af}(f,g).  
\]
Suppose $\Lambda$ has paraproducts of order $0\leq \kappa<\min\{k_1,k_2\}$ and 
the  $\kappa$-th order cancellative part  of $\Lambda$ has been defined.
Then $\Lambda$ has \textit{paraproducts of $(\kappa+1)$-th order} if  for each $\gamma=(\gamma_1,\gamma_2)$ with $|\gamma_1|= \kappa+1= |\gamma_2| $  and 
$\af \in   \vaf$ there exists $b_{\gamma}^{\af}\in \mathrm{BMO}(\R^{\mathbf{d}})$  such that\[
\left[\Lambda_{\kappa}\right]^{\af}(  x_1^{\gamma_1}\otimes x_2^{\gamma_2},\psi )=\left\langle b_{\gamma}^{\af}, \partial^{-\gamma_1}\partial^{-\gamma_2}\psi \right \rangle \qquad \forall \psi \in \mathcal S_D(\R^{\mathbf d}).
\]
If $T_\kappa^\af$ stand for the adjoints to $\Lambda_\kappa$,  the corresponding $T(1)$ notation is then
\[
b_{\gamma}^{\af} = R_{\R^{d_1}}^{\gamma_1} R_{\R^{d_2}}^{\gamma_2} |\nabla_{\R^{d_1}}|^{|\gamma_1|}|\nabla_{\R^{d_2}}|^{|\gamma_2|} T_\kappa^\af(x_1^{\gamma_1} \otimes x_2^{\gamma_2}), \qquad \af \in \vaf.
\]
The inductive definition is closed by defining the \textit{$\kappa+1$-th order cancellative part of $\Lambda$} as 
\begin{equation}
\label{e:kordpartbip}
\Lambda_{\kappa+1}(f,g) = \Lambda_{\kappa} (f,g) - \sum_{\substack {\gamma=(\gamma_1,\gamma_2) \\ |\gamma_1|=|\gamma_2|=\kappa+1 } }
\sum_{\af \in \vaf } \left[\Pi_{ b_{\gamma}^{\af},\gamma}\right]^{\af} (f,g) -\sum_{\substack{\af \in \{\circ, \star\}\\ \iota \in \{1,2\}\\ |\gamma_\iota|=|\alpha_\iota|=\kappa+1}}
 \left[ \Pi^\iota_{\mathbf{a}^{\iota,\af}_{\gamma_\iota,\alpha_\iota}[\Lambda_\kappa]}\right]^{\af}(f,g).  
\end{equation}
We do not define paraproducts of order $\kappa+1$ for $\min\{k_1,k_2\}\leq \kappa\leq \max\{k_1,k_2\}-1$. However, we define the $(\kappa+1)$-th order  cancellative  part of $\Lambda$, also inductively, by 
\begin{equation}
\label{e:kordpartbip2}
\Lambda_{\kappa+1}(f,g) = \Lambda_{\kappa} (f,g)   -\sum_{\substack{\af \in \{\circ, \star\}  \\| \gamma_{\iota^*}|= |\alpha_{\iota^*}|=\kappa+1}}
 \left[ \Pi^{\iota^*}_{\mathbf{a}^{{\iota^*},\af}_{\gamma_{\iota^*},\alpha_{\iota^*}}[\Lambda_\kappa]}\right]^{\af}(f,g), \qquad \iota^*=\arg\max\{k_\iota\}.
\end{equation}
\end{definition}
}
 \begin{definition}[] The  bi-parameter bilinear form $\Lambda$   belongs to the class $\mathrm{CZ}(\R^{\mathbf{d}},k,\delta)$ of  \textit{$(k,\delta)$-Calder\'on-Zygmund forms}  if $\Lambda \in \mathrm{SI}(\R^{\mathbf{d}},k,\delta)$ and $\Lambda$ has paraproducts of order $\min\{k_1,k_2\}$, with norm
\begin{equation}
\label{e:CZnorm2p}
\|\Lambda\|_{\mathrm{CZ}(\R^{\mathbf{d}},k,\delta)}\coloneqq \|\Lambda\|_{\mathrm{SI}(\R^{\mathbf{d}},k,\delta)}+ 
\sum_{\substack{0\leq \kappa \leq \min\{k_1,k_2\} \\  |\gamma_1|=|\gamma_2|=\kappa \\\af \in \vaf}}  \left\| |\nabla_{\R^{d_1}}|^{|\gamma_1|}|\nabla_{\R^{d_2}}|^{|\gamma_2|} T_\kappa^\af(x_1^{\gamma_1} \otimes x_2^{\gamma_2})\right\|_{\mathrm{BMO}(\R^{\mathbf d})}.
\end{equation}
\end{definition}
Forms in the $\mathrm{CZ}(\R^{\mathbf{d}},k,\delta)$ may be represented as a linear combination of  a wavelet form of type \eqref{e:cancformbip} and order $k$ plus a finite linear combination of paraproducts and half-paraproducts.
  \begin{theorem} \label{t:T12p} Let $k=(k_1,k_2)\in \mathbb N^2$ with $\max\{k_1,k_2\}+1 \leq D$, $0 <\eps<\delta\leq 1$. There exists an absolute constant $C=C_{k,\delta,\epsilon,d}$ such that the following holds.
 Let $\Lambda$ be a  form of class $ \mathrm{CZ}(\R^{\mathbf{d}},k,\delta) $  with normalization $\|\Lambda\|_{\mathrm{CZ}(\R^{\mathbf{d}}),k,\delta} \leq 1$.
Then, there exists  a family \[\left\{\upsilon_{z}\in C\Psi_{z}^{k,\eps;0,0}:z\in Z^{\mathbf{d}}\right\},\] and  functions   $\mathbf{a}^{\iota,\af}_{\gamma_\iota,\alpha_\iota}$ on $ Z^{d_{\hat \iota }}\times Z^{d_{\hat \iota }}$ taking values in a   bounded subset of  $\mathrm{BMO} (\R^{d_\iota})$,  $\iota=1,2$ such that for all $f,g\in \mathcal S(\R^{\mathbf{d}})$
\begin{equation}
\label{e:representation2}
\begin{split}
  \Lambda(f,g) & =   \int_{Z^{\mathbf{d}}}     \langle f,  \varphi_{z}\rangle  \langle \upsilon_{z},g \rangle  \, \d \mu(z)
 \\  &\quad +  \sum_{\substack {\gamma=(\gamma_1,\gamma_2) \\ 0 \leq |\gamma_1|=|\gamma_2|\leq  \min\{k_1,k_2\}  \\ \af \in \vaf }} \left[\Pi_{ b_{\gamma}^{\af},\gamma}\right]^{\af} (f,g) 
+  
\sum_{\substack{\iota \in \{1,2\} \\0\leq |\gamma_\iota| \leq k_\iota \\ |\alpha_\iota|=|\gamma_\iota| \\ \af \in \{\circ, \star\}}}
 \left[ \Pi^\iota_{\mathbf{a}^{\iota,\af}_{\gamma_\iota,\alpha_\iota}}\right]^{\af}(f,g).
\end{split}
\end{equation}
\end{theorem}
 Theorem \ref{t:T12p}    will be proved in Section \ref{s7}.  
   The next corollary to Theorem \ref{t:T12p} is easily proved by combining with \eqref{e:representation2} the estimates of Propositions \ref{p:wbTcanc}, \ref{p:full} and \ref{p:half}, and, for the cases $k\neq(0,0)$, an integration by parts argument akin to the one used in the proof of Corollary \ref{cor:t1}.
\begin{corollary}   \label{cor:t12} Let   $  k\in \mathbb N^2,\delta>0$ and   $\Lambda \in \mathrm{CZ}(\R^{\mathbf{d}},k,\delta)$ be  as in Theorem \ref{t:T12p}. Assume in addition   the bi-parameter analogue of \eqref{e:cort1}
\begin{equation}
\label{e:cort1b}
\begin{split} &
(|\gamma_1|, |\gamma_2|) \neq k \implies 
b^\circ_{\gamma}=0;
\\ &
j\in\{1,2\},\, |\gamma_j|<k_j\implies b^{\star_j}_{\gamma}=0,  \;\mathbf{a}^{j,\circ }_{\gamma_j,\alpha_j} = 0 \quad \forall |\alpha_j|=|\gamma_j|.
\end{split}
\end{equation}
 Let $1<p<\infty$ and $w$ be a product $ A_p$-weight in $\mathbb R^{\mathbf{d}}$.
Then, if $T$ stands for the adjoint   to $\Lambda$, 
\begin{equation}
\label{e:rate}
\| \nabla_{x_1}^{k_1} \nabla_{x_2}^{k_2} Tf\|_{ L^p(\R^d; w)} \lesssim_{k,\delta}  [w]_{A_p}^{\max\left\{3,\frac{2p}{p-1}\right\}}\| \nabla_{x_1}^{k_1} \nabla_{x_2}^{k_2}f\|_{ L^p(\R^d; w)}.
\end{equation}
\end{corollary}
Notice that \eqref{e:cort1b} is not an additional assumption in the case $k=(0,0)$ and, proceeding as in Remark \ref{r:nec}, necessary for \eqref{e:rate} to hold otherwise. 
 \begin{remark}  For the reader's convenience, we point out that among the bounds provided in these propositions, the exponent in \eqref{e:rate} is achieved by the paraproduct estimate of Proposition \ref{p:full} and their adjoints. 
In fact,     if $\Lambda$ is a $((0,0),\delta)$  Calder\'on-Zygmund form whose paraproduct terms appearing in \eqref{e:representation2} 
all vanish, then the weighted norm bound for its adjoint
\begin{equation}
\label{e:sharp}
\|T \|_{ L^p(\R^d; w)} \lesssim  [w]_{A_p}^{\theta(p)}, \qquad \theta(p)=\begin{cases} \frac{2}{p-1} & 1<p\leq\frac 32
\\[4pt] \frac{2p-1}{p-1} & \frac 32<p\leq  2 \\[4pt] 
\frac{p+1}{p-1} & 2< p\leq 3
\\[4pt]  2 & p>3
 \end{cases}
\end{equation}
may be read by applying Proposition \ref{p:wbTcanc} to the cancellative terms in \eqref{e:representation2} for $\Lambda$, if $p\geq 2$, or for its full adjoint $\Lambda^\star$ if $1<p<2$. Comparing with the one parameter case, see \cite[Theorem 2]{LLP}, estimate \eqref{e:sharp} is sharp for $\max\{p,p'\}\geq 3$: there seem to be no instances of sharp weighted norm inequalities for bi-parameter operators in previous literature. Notice that the paraproduct free assumption covers, for instance, bi-parameter convolution-type operators. 
\end{remark}

\subsection{Proof of Lemma \ref{l:qs}} \label{ss:pflqs}
   For the sake of definiteness, work in   the completely generic case $\iota=1, \af=\circ$.  Fix $z_2,\zeta_2 \in Z^{d_2}$, and consider without loss of generality the case $z_2\in Z^{d_2}_+(\zeta_2)$.    Invoking  the $\mathrm{BMO}(\R^{d_{1}})$ boundedness of the Riesz transform $R^{\gamma_1}$,  we  realize it must be shown that \[
\|b\|_{\mathrm{BMO}(\R^{d_{1}})} \lesssim  [z_2,\zeta_2]_{k_2+\eta} \qquad b \textrm{ defined by } \langle   b,  g_1 \rangle =  \Lambda(x_1^{\gamma_1} \otimes \chi_{z_2,\zeta_2}, |\nabla|^{k_1} g_1\otimes \varphi_{\zeta_2} ). \]
Let $M=2^8(1+k_1)$. By $H^1-\mathrm{BMO}$ duality and $H^1$-density  of the latter class of functions, it will be enough to show that whenever  $w_1=(y_1,t_1) \in Z^{d_1}$,    $\psi\in \mathcal S(\R^{d_1})$ is a Schwartz function such that $\Psi\coloneqq \mathsf{Sy}_{w_1}^{-1} \psi$ satisfies
\[
\|\Psi\|_{\star,4M,1} \leq 1, \qquad \mathrm{supp}\, \widehat { \Psi} \subset\{y\in \R^{d_j}: 1\leq |y| \leq 2\}
\] 
there holds
\begin{equation}
\label{e:bw1}
|\langle b, \psi \rangle| \lesssim [z_2,\zeta_2]_{k_2+\eta} .
\end{equation}
The frequency support of $\psi$ ensures that $\upsilon\coloneqq t_1^{k_1} |\nabla |^{k_1} \psi\in \Psi^{3M,1; 0}_{w_1} $. 
Now introduce the local notation, with reference to \eqref{e:cutoffs}
\[
p(\cdot) =\left( \frac{\cdot-y_1}{t_1}\right)^{\gamma_1},    \quad \Theta_2\coloneqq \chi_{z_2,\zeta_2}\alpha_{\zeta_2}, \quad  \Xi_2 \coloneqq   \chi_{z_2,\zeta_2}\beta_{\zeta_2}.
\]
There is an additional technical complication brought by the fact that $\upsilon$ is not of compact support. This is dealt with it by introduction of a sequence of smooth functions $q_n\in \mathcal C^\infty(\R^{d_1})$ with $\sum_{n=0}^\infty q_n=1$, $\mathrm{supp}\, q_0 \subset \mathsf{B}_{(0,t_1)}$, $\mathrm{supp}\,  q_n \subset A_n\coloneqq \mathsf{B}_{(0,2^{n+1}t_1)}\setminus  \mathsf{B}_{(0,2^{n-1}t_1)}$ for $n \geq 0$. Define $
p_n \coloneqq p q_n,$  and $ \upsilon_n \coloneqq \upsilon q_n.$
With these notations, the definition of $b$, and the fact that  $\mathbf{q}_{\gamma'_1} \equiv 0$ for all   $|\gamma'_1|<|\gamma_1|$,
\begin{equation}
\label{e:bw2}
\begin{split}& \quad
\langle b, \psi  \rangle =   \Lambda(p \otimes \chi_{z_2,\zeta_2},\upsilon\otimes \varphi_{\zeta_2} ) \\ & = 
\sum_{m\sim n} + \sum_{m\not \sim n}   \Lambda(p_n \otimes \chi_{z_2,\zeta_2},\upsilon_m \otimes \varphi_{\zeta_2} )  + \Lambda(p_n\otimes  \chi_{z_2,\zeta_2},\upsilon_m  \otimes \varphi_{\zeta_2} ) 
\end{split}
\end{equation}
where    $m\sim n $ if $|m-n|<2^4$   and $m\not \sim n$ otherwise. 
The next task consists of bound each summand in the last right hand side of  \eqref{e:bw2}.
\subsubsection*{The $m \sim n $ summand} Notice that in this range  $\|p_n\|_\infty \lesssim  2^{k_1n}$, $\|\upsilon_m\|_\infty \lesssim t_1^{-d_1} 2^{-3Mm}\lesssim t_1^{-d_1} 2^{-2Mn} $ and $p_n,\upsilon_m$ are  supported on $\mathsf{B}_{(y_1,2^{n+2^5} t_1)}$. Also note that $\Theta_2$ is supported on $\mathsf{B}_{(\xi_2,4\sigma_2)}$ and   $\|\Theta_2\|_\infty\lesssim  [z_2,\zeta_2]_{k_2+\eta} $; this is obvious if $\chi_{z_2,\zeta_2}=\varphi_{z_2}$ and may be read from \eqref{e:Lpoly} otherwise.  Applying the weak boundedness property of $\Lambda$ with balls $\mathsf{B}_{(y_1,2^{n+2^5} t_1)}$ and  $\mathsf{B}_{(\xi_2,4\sigma_2)}$, 
\[
\left|\Lambda(p_n\otimes \Theta_2,\upsilon_m\otimes \varphi_{\zeta_2} ) \right| \lesssim \|p_n\|_\infty  \|\upsilon_m\|_\infty\lesssim   2^{-Mn}  [z_2,\zeta_2]_{k_2+\eta} .
\]
 Furthermore, $\Xi_2$ and $\varphi_{z_2}$ have separated supports. Therefore, using the partial kernel assumptions for the form $(f,g)\mapsto \Lambda(\Theta_1 \otimes f,  \upsilon_{w_1} \otimes g)$ and repeating the computations of \eqref{e:repet} for such form
\[
\left|\Lambda(p_n \otimes \Xi_2,\upsilon_m \otimes \varphi_{\zeta_2} ) \right| \lesssim    2^{-Mn}  [z_2,\zeta_2]_{k_2+\eta} .
\]
The last two estimates are summable on the diagonal $m\sim n$ and this  completes the bound for the $\sim$ summand in \eqref{e:bw2}. 
\subsubsection*{The $m\not\sim n$ summand} We now have that $p_n$ and $\upsilon_m$ have separated supports by $\sim t_1 2^{\max\{m,n\}}$.  
Applying the partial kernel assumptions for the form $(f,g)\mapsto \Lambda(f\otimes\Theta_2,g\otimes \varphi_{\zeta_2} )$ and arguing as in   \eqref{e:repet} yields 
\[
|\Lambda(p_n\otimes\Theta_2 ,\upsilon_m \otimes \varphi_{\zeta_2} )|  \lesssim \|p_n\|_\infty  2^{-\max\{m,n\}(k_1+\delta)}    [z_2,\zeta_2]_{k_2+\eta}\lesssim 2^{-\max\{m,n\}\delta}.    
\]
Finally, in the term below, the full kernel assumptions may be used due to functions in both parameters having disjoint supports. Standard computations relying on the kernel estimates as in \eqref{e:repet} then lead to   \[
|\Lambda(p_n\otimes \Xi_2,\upsilon_m \otimes \varphi_{\zeta_2} ) |\lesssim 2^{-\max\{m,n\}\delta}  [z_2,\zeta_2]_{k_2+\eta} .
\]
The above estimates are summable over $m\not\sim n$, 
which  completes both this case and the proof of the Lemma.
\begin{remark}
\label{r6:add1} We explain the details of the definitions \eqref{e:qs}. Using symmetry with respect to adjoints, it suffices to study  $\mathbf{q}^{1,\circ}_{ \gamma_1}(z_2,\zeta_2)$. The most complicated case is when  either one of $\chi_{z_2,\zeta_2},\chi_{\zeta_2,z_2}$ contains the polynomial summand. To fix ideas, work with $\chi_{z_2,\zeta_2}= \varphi_{z_2}-P_{z_2,\zeta_2}$. Let $\phi_2=\phi$ as in \eqref{e:phiR} with $d=d_2$ and  $p_{R}^{\gamma_1}$ as in \eqref{e:pR} with $d=d_1$ and $\alpha=\gamma_1$, see Subsection \ref{ss:czkdelta}. If $0 \notin \mathrm{supp}\,\widehat{f_1}$ then $g_1=|\nabla|^{|\gamma_1|}f_1\in \mathcal S_D(\R^{d_1})$. The partial kernel estimates of $\Lambda$ readily show that
\[
\langle \mathbf{q}^{1,\circ}_{ \gamma_1}(z_2,\zeta_2),f_1\rangle=\lim_{R\to \infty} \Lambda(p_{R}^{\gamma_1} \otimes [\Dil_{R}^\infty\phi_2] \chi_{z_2,\zeta_2}, g_1 \otimes \varphi_{\zeta_2} ), \qquad 
\]
exists and defines a linear continuous functional on the subspace of $\mathcal S(\R^{d_1})$ of functions supported off the frequency origin. 
\end{remark} 

 \section{Proof of Theorem \ref{t:T12p}} \label{s7}
  Below  $0<\delta\leq 1$, $0<\eps<\delta$ are fixed. Set $\eta=\frac{\delta+\eps}{2}$, so that $\eps<\eta<\delta$.  \subsection{Preliminaries} Before entering the main argument, a series of preparatory lemmas is required.
First, the two parameter version of Lemma \ref{l:average1} is provided.
\begin{lemma} 
\label{l:average2} 
 Let $\varphi_z$ be  as in \eqref{e:varphi2p} and $u:Z^{\mathbf{d}} \to \mathbb C$ be a Borel measurable function with $|u(z)|\leq 1$. Then, there exists $C\lesssim_{k_1,k_2,\eps} 1$ such that for all $z=(z_1,z_2)\in Z^\mathbf{d}$  with $z_j=(x_j,s_j)$, $j=1,2$
\[
\begin{split}
&\upsilon_{z }  \coloneqq \int\displaylimits_{\substack{( \alpha_1,\beta_1) \in Z^{d_1}  \\ ( \alpha_2,\beta_2) \in Z^{d_2} }}\left(\prod_{j=1}^2 [(\alpha_j,\beta_j)]_{k_j+\eta}\right)
u((\alpha_1,\beta_1),(\alpha_2,\beta_2))  \varphi_{((x_1+\alpha_1 s_1,\beta_1 s_1),(x_2+\alpha_2 s_2,\beta_2 s_2))}  \frac{ \d\beta_2 \d \alpha_2  \d \beta_1\d\alpha_1}{\beta_1\beta_2}
\end{split}
\] 
belongs to $C\Psi^{(k_1,k_2),\eps;0,0}_z$.  
\end{lemma}
\begin{proof} There is a direct argument along the lines of the one parameter proof. However, an argument that uses Lemma \ref{l:average1} as a black box will be given. To save space in the notation, this argument is carried out for $x_j=0,s_j=1$, $j=1,2$.  
Notice that for each fixed $w_1\in \R^{d_1}$ 
\[
\upsilon_{z }^{[1,w_1]}  =  
\int\displaylimits_{ (\alpha_2,\beta_2)\in Z^{d_2}} 
 [(\alpha_2,\beta_2)]_{k_2+\eta}\upsilon_{\alpha_2,\beta_2}(w_1) (\varphi_{2})_{(x_2+\alpha_2 s_2,\beta_2s_2) }  \frac{ \d\beta_2 \d \alpha_2   }{\beta_2},
\]
where
\[
\upsilon_{\alpha_2,\beta_2} =  \int\displaylimits_{ (\alpha_1,\beta_1) \in Z^{d_1}}
[(\alpha_1,\beta_1)]_{k_1+\eta} u((\alpha_1,\beta_1),(\alpha_2,\beta_2))   (\varphi_{1})_{(x_1+\alpha_1 s_1,\beta_1 s_1)}  \frac{ \d\beta_1 \d \alpha_1  }{\beta_1 }\in C\Psi^{k_1,\eps;0}_{z_1}
\]
with uniform constant $C$ by an application of \eqref{e:averagproc2} of Lemma \ref{l:average1} with $\eta$ in place of $\delta$. In particular the function $u_{w_1}(\alpha_2,\beta_2)\coloneqq \langle w_1\rangle^{(d_1+k_1+\eps)}\upsilon_{\alpha_2,\beta_2}(w_1)$ is uniformly bounded. Therefore, another application of \eqref{e:averagproc2}, with $u=u_{w_1}(\alpha_2,\beta_2)$ entails
\[\langle w_1\rangle^{(d_1+k_1+\eps)}
\upsilon_{z }^{[1,w_1]}=\int\displaylimits_{ ( \alpha_2 ,\beta_2)\in Z^{d_2} } 
[(\alpha_2,\beta_2)]_{k_2+\eta}u_{w_1}(\alpha_2,\beta_2)  (\varphi_{2})_{(x_2+\alpha_2 s_2,\beta_2s_2) }  \frac{ \d\beta_2 \d \alpha_2   }{\beta_2} \in C\Psi^{k_2,\eps;0}_{z_2}.
\]
Repeating for the second parameter and comparing with equation \eqref{e:bip1}, proves that $\upsilon_{z }\in C\Psi^{(k_1,k_2),\eps;0,0}_z$ and thus completes the proof of the lemma.
  \end{proof}
 The notation $\chi_{z_\iota,\zeta_\iota}$  appearing below refers to \eqref{e:U}.
\begin{lemma}  \label{l:U2}
For  $(z,\zeta)\in Z^{\mathbf{d}}\times  Z^{\mathbf{d}}$ with $z=(z_1,z_2)$, $\zeta=(\zeta_1,\zeta_2)$,  
\[ 
\left| \Lambda\left(\chi_{z_1,\zeta_1} \otimes \chi_{z_2,\zeta_2}, \chi_{\zeta_1,z_1}\otimes \chi_{\zeta_2,z_2}\right) 
\right|\lesssim  \|\Lambda\|_{\mathrm{CZ}(\R^{\mathbf{d}}),k ,\delta} \prod_{j=1,2} [z_j,\zeta_j]_{k_j+\eta}.\]
\end{lemma}
\begin{proof}
By symmetry with respect to the adjoint, it suffices to consider    the case where $z_j \in Z^{d_j}_+(\zeta_j)   $ for both $j=1,2$.  Nine different cases according to which of the sets in \eqref{e:partz} with $\zeta=\zeta_j$ each $z_j$ belongs to need to be considered. Only the case $z_j\in A(\zeta_j)$ for $j=1,2$ will be dealt with explicitly;  all remaining  cases may be dealt with by the same strategy that will be used for the summands appearing in \eqref{e:lu2}.
In this case, $\chi_{z_j,\zeta_j}= \varphi_{z_j}- P_{z_j,\zeta_j}$ and $\chi_{\zeta_j,z_j}=\varphi_{\zeta_j} $ for $j=1,2$, and thus
\begin{equation}
\label{e:lu2} 
\begin{split}
&  \Lambda\left(\chi_{z_1,\zeta_1} \otimes \chi_{z_2,\zeta_2}, \varphi_{\zeta_1} \otimes \varphi_{\zeta_2}\right) =\sum_{(\iota,\jmath)\in \{\mathsf{in},\mathsf{out}^2\}}     \Lambda\left(\Theta_{1,\iota} \otimes \Theta_{2,\jmath}, \varphi_{\zeta_1} \otimes \varphi_{\zeta_2}\right), 
\\ & 
\Theta_{j,\mathsf{in}} \coloneqq \chi_{z_j,\zeta_j} \alpha_{\zeta_j}, \qquad \Theta_{j,\mathsf{out}} \coloneqq \chi_{z_j,\zeta_j} \beta_{\zeta_j}, \qquad j=1,2.
\end{split}
\end{equation}
Each term in \eqref{e:lu2} will be estimated. The key to the first three summands is that for $j=1,2$ the function $\Theta_{j,\mathsf{in}}$ is supported on $4\mathsf{B}_{\zeta_j}$ and, from \eqref{e:Lpoly}, $ \|\Theta_{j,\mathsf{in}}\|_\infty \lesssim [z_j,\zeta_j]_{k_j+\eta} $. For the $\mathsf{in}, \mathsf{in}$ summand, employ the weak boundedness   of $\Lambda$ with points $\tilde \zeta_ j=(\xi_j,4\sigma_j)$ thus obtaining 
\[
\left|\Lambda\left(\Theta_{1,\mathsf{in}} \otimes \Theta_{2,\mathsf{in}}, \varphi_{\zeta_1} \otimes \varphi_{\zeta_2}\right)\right| \lesssim \prod_{j=1,2}\|\Theta_{j,\mathsf{in}}\|_\infty \lesssim \prod_{j=1,2} [z_j,\zeta_j]_{k_j+\eta}.
\]
The $\mathsf{in}, \mathsf{out}$ summand is bounded as follows. Observe that $\Theta_{2,\mathsf{out}}$ and $\varphi_{\zeta_2}$
have separated support. Now, apply the partial kernel/weak boundedness assumption to the form $(f,g) \mapsto  \sigma_1^{d_1}\Lambda(\Theta_{1,\mathsf{in}} \otimes f, \varphi_{\zeta_1},g)$ at point $\tilde \zeta_1=(\xi_1,4\sigma_1)$, to which estimate \eqref{e:repet} with $z=z_2$ and $\zeta=\zeta_2$ actually applies. Such estimate returns a factor of $[z_2,\zeta_2]_{k_2+\eta}$ while the factor $[z_1,\zeta_1]_{k_1+\eta}$ is obtained from $\|\Theta_{1,\mathsf{in}}\|_\infty$. 
The $\mathsf{out}, \mathsf{in}$ summand is handled in exactly the same way. 

In the   $\mathsf{out}, \mathsf{out}$ summand, note that $\Theta_{j,\mathsf{out}}$ and $\varphi_{\zeta_j}$ both 
have separated support. The full kernel estimates of $\Lambda$ are now employed. The calculation leading to the estimate
\[
\left|\Lambda\left(\Theta_{1,\mathsf{out}} \otimes \Theta_{2,\mathsf{out}}, \varphi_{\zeta_1} \otimes \varphi_{\zeta_2}\right)\right|   \lesssim \prod_{j=1,2} [z_j,\zeta_j]_{k_j+\eta}
\]
is the tensor product of the steps \eqref{e:taylor5a}-\eqref{e:repet} performed in each variable, the only difference being how the corresponding term in \eqref{e:taylor5a} involving the finite difference of the derivatives of the kernel is controlled. In that case, for a fixed $v=(v_1,v_2) \in (2\mathsf{B}_{\zeta_1}\times 2\mathsf{B}_{\zeta_2} )^c$ one uses the cancellation and $L^1$-estimate of $\sigma_j^{-k_j}\partial^{-\gamma_j} \varphi_{\zeta_j} $ and bounds 
\[
\sup_{u\in \mathsf{B}_{\zeta_1}\times \mathsf{B}_{\zeta_2}}\left|
\Delta^1_{u_1-\xi_1|\cdot}\Delta^2_{u_2-\xi_2|\cdot}\partial^{\gamma_1}_{u_1}\partial^{\gamma_1}_{u_1} K(\xi,v)\right| \lesssim \prod_{j=1}^2  \frac{\sigma_j^{k_j+\delta}}{|v_j-\xi_j|^{d_j+k_j+\delta}}
\]
 using the  kernel estimate in the fourth line \eqref{e:full}. 
This completes the proof of the Lemma.\end{proof} 
\subsection{Main line of proof of Theorem \ref{t:T12p}}
It is now possible to turn to the main line of proof of Theorem \ref{t:T12p}. Notice that \[
\|\Lambda\|_{\mathrm{CZ}(\R^d,(\kappa_1,\kappa_2),\delta)} \leq \|\Lambda\|_{\mathrm{CZ}(\R^d,k,\delta)}=1, \qquad 0\leq \kappa_j \leq k_j, \,j=1,2.
\]
The proof will be done via two consecutive inductions. The first runs for $0\leq \kappa\leq \min\{k_1,k_2\}$. The second, if necessary, runs for $\min\{k_1,k_2\}< \kappa \leq \max\{k_1,k_2\}$. 
The argument is symmetric with respect to interchanging parameters, therefore there is no loss in generality with assuming $k_1\geq k_2$.
 \subsubsection{Base case and main part of inductive step} The base case and main part of the inductive step of the proof works under  the additional assumption referring to Lemma \ref{l:qs}
 \[
 a(\kappa):
 \begin{cases} \displaystyle b_{\gamma}^{\af}=0 &\displaystyle \forall \af\in \vaf, \min_{\iota\in \{1,2\}} |\gamma_\iota| <\kappa,\\   \mathbf{q}^{\iota,\af}_{\gamma_\iota}=0& \forall \af \in \{\circ, \star\}, |\gamma_\iota|<\min\{\kappa,k_\iota\}, \iota\in\{1,2\},\end{cases}    
 \]
 namely, all  paraproducts   $T^\af(x^{\gamma_1} \otimes x^{\gamma_2}) $ vanish except possibly those with $|\gamma_1|=|\gamma_2|=\kappa$, and all half-paraproducts vanish except possibly those of highest order.
  Clearly, $a(0)$ is not an extra assumption. Moreover, if $\kappa\geq 1$, assumption $a(\kappa)$ implies that $\Lambda$ coincides with its $(\kappa-1)$-th order cancellative part $\Lambda_{\kappa-1}$, which means we are allowed to conflate the two forms and just write $\Lambda$ below. 
 
 Let now $f,g \in \mathcal S(\R^{\mathbf{d}})$. Using the bi-parameter analogue of \eqref{e:CRF},  bilinearity and $\mathcal S(\R^{\mathbf{d}})$-continuity of $\Lambda$, and later the definition of $U(z,\zeta)  $ leads to the decomposition
\begin{equation} 
\label{e:rep1,2}
\begin{split}
&\quad \Lambda(f,g)  =  \int\displaylimits_{Z^{\mathbf d}\times Z^{\mathbf d}}  \langle f,\varphi_{z} \rangle \langle \varphi_{\zeta},g \rangle \Lambda(\varphi_{z},  \varphi_{\zeta}) \, \d \mu(z) \d \mu(\zeta)\\ 
&= \int\displaylimits_{Z^{\mathbf d}\times Z^{\mathbf d}}  \langle f,\varphi_{z} \rangle \langle \varphi_{\zeta},g \rangle  \Lambda\left(\chi_{z_1,\zeta_1} \otimes \chi_{z_2,\zeta_2}, \chi_{\zeta_1,z_1}\otimes \chi_{\zeta_2,z_2}\right) \,  {\d \mu(z) \d \mu(\zeta)} 
\\ &+ 
\int\displaylimits_{Z^{d_1}_{\zeta_1}\times Z^{d_2}_{\zeta_2} }   \int\displaylimits_{Z^{d_2}_{z_2}}  \int\displaylimits_{z_1\in A(\zeta_1)}  \langle f,\varphi_{z} \rangle \langle \varphi_{\zeta},g \rangle \Lambda(P_{z_1,\zeta_1}\otimes \chi_{z_2,\zeta_2}  , \varphi_{\zeta_1} \otimes \chi_{\zeta_2,z_2} ) \,  {\d \mu(z) \d \mu(\zeta) } +\cdots \\
 &+  
\int\displaylimits_{Z^{d_1}_{\zeta_1}\times Z^{d_2}_{\zeta_2} }  \int\displaylimits_{z_1\in A(\zeta_1)} \int\displaylimits_{z_2\in A(\zeta_2)}   \langle f,\varphi_{z} \rangle \langle \varphi_{\zeta},g \rangle \Lambda( P_{z_1,\zeta_1}\otimes P_{z_2,\zeta_2} ,\varphi_{\zeta_1} \otimes \varphi_{\zeta_2} ) \, {\d \mu(z) \d \mu(\zeta) }+ \cdots.
\end{split}
\end{equation}
Here, the dots in the third line are hiding three more terms where the integration domain is respectively restricted to $z_2\in A(\zeta_2)$, $\zeta_1\in A(z_1)$, $\zeta_2\in A(z_2)$, and the integrands involve respectively the coefficients
\[
\Lambda(\chi_{z_1,\zeta_1}\otimes P_{z_2,\zeta_2}  ,  \chi_{\zeta_1,z_1}\otimes \varphi_{\zeta_2} ), \quad \Lambda(\varphi_{z_1}\otimes \chi_{z_2,\zeta_2}  , P_{z_1,\zeta_1} \otimes \chi_{\zeta_2,z_2} ), \quad 
\Lambda(\chi_{z_1,\zeta_1}\otimes  , \varphi_{z_2}, \chi_{\zeta_1,z_1}\otimes  P_{z_2,\zeta_2}  ),
\]
while the dots in the fourth line also hide three more terms where the integration domain is  restricted to $\{z_1\in A(\zeta_1),\zeta_2\in A(z_2)\}$, $\{\zeta_1\in A(z_1),z_2\in A(\zeta_2)\}$, $\{\zeta_1\in A(z_1),\zeta_2\in A(z_2)\}$  and the integrands involve respectively the coefficients
\[
\Lambda(\varphi_{z_1}\otimes P_{z_2,\zeta_2}  ,  P_{\zeta_1,z_1} \otimes \varphi_{\zeta_2} ), \quad 
\Lambda(P_{z_1,\zeta_1}\otimes \varphi_{z_2}  ,\varphi_{\zeta_1} \otimes P_{\zeta_2,z_2} ), \quad 
\Lambda(\varphi_{z_1}\otimes \varphi_{z_1} , P_{\zeta_1,z_1}\otimes P_{\zeta_2,z_2}  ).
\]
It is possible to turn the first summand in \eqref{e:rep1,2} into the first summand of \eqref{e:representation2}. First, make the change of variable 
\[
\zeta=\zeta(z,(\alpha_1,\beta_1), (\alpha_2,\beta_2) )= \left((x_1+\alpha_1s_1,\beta_1 s_1),( x_2+\alpha_2 s_2,\beta_2 s_2)\right)
\] and then use Fubini's theorem in the inner variable of $g$. The  first summand in \eqref{e:rep1,2} then equals  \[ 
\begin{split}& 
\int\displaylimits_{Z^{\mathbf{d}}}  \langle f,\varphi_{z} \rangle   \langle \upsilon_{z},g \rangle\,  \d
\mu(z),  \qquad
 \upsilon_{z}  
\coloneqq \int\displaylimits_{\substack{( \alpha_1,\beta_1) \in Z^{d_1}  \\ ( \alpha_2,\beta_2) \in Z^{d_2} }}\Lambda\left(\chi_{z_1,\zeta_1} \otimes \chi_{z_2,\zeta_2}, \chi_{\zeta_1,z_1}\otimes \chi_{\zeta_2,z_2}\right)\varphi_{\zeta}  \frac{ \d\beta_2 \d \alpha_2  \d \beta_1\d\alpha_1}{\beta_1\beta_2},
\end{split}
\] where under the integral sign $\zeta=  \zeta(z,(\alpha_1,\beta_1), (\alpha_2,\beta_2))$.
With the same convention,  
\[
u_z\big((\alpha_1,\beta_1),(\alpha_1,\beta_1)\big) = \Lambda\left(\chi_{z_1,\zeta_1} \otimes \chi_{z_2,\zeta_2}, \chi_{\zeta_1,z_1}\otimes \chi_{\zeta_2,z_2}\right)\left(\prod_{j=1}^2 [(\alpha_j,\beta_j)]_{\min\{\kappa, k_j\}+\eta}\right)^{-1}
\]
is uniformly bounded via   Lemma \ref{l:U2}, and applying Lemma \ref{l:average2} yields   $\upsilon_{z}\in C\Psi_{z}^{(\kappa, \min\{\kappa,k_2\}),\eps;0}$.

It remains to identify the remaining terms in \eqref{e:rep1,2} as a sum of paraproduct terms. Here it is crucial to use assumption $a(\kappa)$, which tells us that
\[
|\gamma_\iota|<\min\{\kappa, k_\iota\} \implies\mathbf{q}^{\iota,\af}_{\gamma_\iota}(z_{\hat \iota},\zeta_{\hat \iota})=0 \qquad \forall z_{\hat \iota},\zeta_{\hat \iota} \in Z^{d_{\hat \iota}}.
\] 
Focus on the term in the third line of \eqref{e:rep1,2} first.  The above observation, the definition of $P_{z_1,\zeta_1}$ from Lemma \ref{l:alpert}, the definition of $\mathbf{q}_{ \gamma_1}^{1,\circ}$, the fact that $\partial^{-\alpha_1} =  R^{\alpha_1}|\nabla|^{-|\alpha_1|} $ with the definition of $\varphi_{\alpha_1,\zeta_1} $, see \eqref{e:ck1paral} and Remark \ref{r:riesz}, gives
\[
\begin{split}
&\quad \Lambda(P_{z_1,\zeta_1}\otimes \chi_{z_2,\zeta_2}  , \varphi_{\zeta_1} \otimes \chi_{\zeta_2,z_2}  )=   \sum_{|\gamma_1|=k_1} \langle \varphi_{z_1},  \mathsf{Sy}_{\zeta_1} \phi_{\gamma_1 }   \rangle \Lambda(x_1^{\gamma_1} \otimes \chi_{z_2, \zeta_2},  {\sigma_1}^{-k_1}\varphi_{\zeta_1}\otimes  \chi_{\zeta_2,z_2} )
 \\ &=  \sum_{|\gamma_1|=k_1} \langle \varphi_{z_1},\mathsf{Sy}_{\zeta_1} \phi_{\gamma_1 }   \rangle \langle \mathbf{q}^{1,\circ}_{ \gamma_1}(z_2,\zeta_2), \mathsf{Sy}_z[|\nabla|^{-k}\varphi_{1}] \rangle    =\sum_{|\gamma_1|=|\alpha_1|=k_1} \langle \varphi_{z_1},\mathsf{Sy}_{\zeta_1} \phi_{\gamma_1 }  \rangle \langle R^{\alpha_1} \mathbf{q}^{1,\circ}_{ \gamma_1}(z_2,\zeta_2),  \varphi_{\alpha_1,z_1}\rangle.
\end{split} 
 \]
Finally  using Lemma \ref{l:averages} with $h=\langle f,\varphi_{z_2} \rangle_2$, the summand in the third line of \eqref{e:rep1,2} equals the sum over $|\gamma_1|=|\alpha_1|=\kappa$ of  
\[
\begin{split}
 &  \;   \int\displaylimits_{(Z^{d_2})^2 }    \int\displaylimits_{Z^{d_1}_{\zeta_1}}\langle \mathbf{a}^{1,\circ}_{ \gamma_1,\alpha_1}(z_2,\zeta_2),  \varphi_{\alpha_1,z_1}\rangle   \left\langle  \langle f,\varphi_{z_2} \rangle_2,\vartheta_{\gamma_1,\zeta_1} \right\rangle_1 
 \left\langle\langle  \varphi_{\zeta_2},g \rangle_2, \varphi_{\zeta_1} \right\rangle_1  \d \mu(\zeta_1)  \, [z_2,\zeta_2]_{ \kappa_2+\eta}   {\d \mu(z_2) \d \mu(\zeta_2) }\\ & = 
 \int\displaylimits_{Z^{d_2} \times Z^{d_2} }    \Pi_{ \mathbf{a}^{1,\circ}_{ \gamma_1,\alpha_1}(z_2,\zeta_2),\gamma_1, \alpha_1}\left(\langle f,\varphi_{z_2}\rangle_2,\langle g,\varphi_{\zeta_2}\rangle_2\right) [z_2,\zeta_2]_{\kappa_2+\eta} \d\mu(z_2) \d\mu(\zeta_2)
 =
 \Pi_{\mathbf{a}^{1,\circ}_{\gamma_1,\alpha_1},\gamma_1,\alpha_1 } (f,g)
 \end{split}
\]
where $\kappa_2=\min\{\kappa,k_2\}$,
which is one of the summands in the second line of \eqref{e:representation2}. The three other types of summands in the second and third line of \eqref{e:representation2}, constructed in exactly the same way, arise from the $\cdots$ terms in the third line of \eqref{e:rep1,2}.

It remains to identify the terms of the type appearing in the third line \eqref{e:rep1,2}. Using again $a(\kappa)$, these terms will appear only if $\kappa \leq k_2$. Lemma \ref{l:alpert} and the definition of the paraproducts of $\Lambda$ then yield
\[
\begin{split}
& \quad\Lambda( P_{z_1,\zeta_1}\otimes P_{z_2,\zeta_2} ,\varphi_{\zeta_1} \otimes \varphi_{\zeta_2} ) \\ &
=   \sum_{|\gamma_1|= |\gamma_2|=\kappa } \langle \varphi_{z_1},\mathsf{Sy}_{\zeta_1} \phi_{\gamma_1 }    \rangle \langle \varphi_{z_2},\mathsf{Sy}_{\zeta_2} \phi_{\gamma_2 }   \rangle \Lambda(x_1^{\gamma_1} \otimes x_2^{\gamma_2},  {\sigma_1}^{-k_1}\varphi_{\zeta_1}\otimes  {\sigma_1}^{-k_1}\varphi_{\zeta_1} )  \\ &=  \sum_{|\gamma_1|= |\gamma_2|=\kappa } \langle \varphi_{z_1},\mathsf{Sy}_{\zeta_1} \phi_{\gamma_1 }   \rangle \langle \varphi_{z_2},\mathsf{Sy}_{\zeta_2} \phi_{\gamma_2 }   \rangle \langle b_{\gamma}^{\circ},\varphi_{\gamma_1,\zeta_1}\otimes \varphi_{\gamma_2,\zeta_2} \rangle.
\end{split}
\]
An application of Lemma \ref{l:averages} with $h=\langle f,\varphi_{z_1} \rangle_1$ yields
\[ 
 F(z_1) \coloneqq  \int\displaylimits_{z_2\in A(\zeta_2)}\langle \langle f,\varphi_{z_1} \rangle_1,\varphi_{z_2} \rangle_2 \langle \varphi_{z_2},\mathsf{Sy}_{\zeta_2} \phi_{\gamma_2 }   \rangle  \d \mu(z_2) = \langle  \langle f,\varphi_{z_1} \rangle_1,\vartheta_{\gamma_2,\zeta_2}\rangle_2  = \langle f,\varphi_{z_1}\otimes\vartheta_{\gamma_2,\zeta_2}\rangle
\]  so that
the summand in the third line of \eqref{e:rep1,2} equals the sum over $|\gamma_1|=\kappa_1,|\gamma_2|=\kappa_2$ of
\[
\begin{split}
& \quad \int\displaylimits_{Z^{d_2} \times Z^{d_2} }\langle b_{\gamma}^{\circ},\varphi_{\gamma_1,\zeta_1}\otimes \varphi_{\gamma_2,\zeta_2} \rangle \langle \varphi_{\zeta},g \rangle  \int\displaylimits_{z_1\in A(\zeta_1)} \langle \varphi_{z_1},\mathsf{Sy}_{\zeta_1} \phi_{\gamma_1 }    \rangle F(z_1) \d \mu(z_1)  \d \mu(\zeta)
\\&=
\int\displaylimits_{Z^{d_2} \times Z^{d_2} }\langle b_{\gamma}^{\circ},\varphi_{\gamma_1,\zeta_1}\otimes \varphi_{\gamma_2,\zeta_2} \rangle \langle \varphi_{\zeta},g \rangle  \int\displaylimits_{z_1\in A(\zeta_1)} \langle\langle f,\vartheta_{\gamma_2,\zeta_2} \rangle_2,\varphi_{z_1} \rangle_1 \langle \varphi_{z_1},\mathsf{Sy}_{\zeta_1} \phi_{\gamma_1 }    \rangle   \d \mu(z_1)  \d \mu(\zeta)
\\ & = \int\displaylimits_{Z^{d_2} \times Z^{d_2} }\langle b_{\gamma}^{\circ}\mathsf{Sy}_{\zeta_1} \phi_{\gamma_1 } \otimes \varphi_{\gamma_2,\zeta_2} \rangle \langle \varphi_{\zeta},g \rangle   \langle f, \vartheta_{\gamma_1,\zeta_1}\otimes \vartheta_{\gamma_2,\zeta_2} \rangle\, \d \mu(\zeta) = \Pi^\circ_{b_{\gamma}^{\circ},\gamma} (f,g)
\end{split}
\] 
where another application of Lemma \ref{l:averages} with $h=\langle f,\vartheta_{\gamma_2,\zeta_2} \rangle_2$ has been carried out and the definition of full paraproduct is finally taken advantage of: this is one of the terms appearing in the fourth line of \eqref{e:representation2}. This procedure may be repeated for the additional terms in the third line of \eqref{e:rep1,2}, thus completing the roster of terms in \eqref{e:representation2} under the additional assumption $a(\kappa)$. Namely, under this assumption, we have proved that
\begin{equation}
\label{e:indu0}\begin{split}
 \Lambda(f,g) &=  \int_{Z^{\mathbf{d}}}     \langle f,  \varphi_{z}\rangle  \langle \upsilon_{z},g \rangle  \, \d \mu(z) \\ & + 
\begin{cases}\displaystyle
\sum_{ \substack{ |\gamma_1|=|\gamma_2|=\kappa  \\ \af \in \vaf }} \left[\Pi_{ b_{\gamma}^{\af},\gamma}\right]^{\af} (f,g)  +
\sum_{\substack{ \iota \in \{1,2\} \\ \af \in \{\circ, \star\}\\ |\alpha_\iota|=|\gamma_\iota|=\kappa}}
 \left[ \Pi^\iota_{\mathbf{a}^{\iota,\af}_{\gamma_\iota,\alpha_\iota}}\right]^{\af}(f,g) & \kappa \leq k_2
 \\
\displaystyle
\sum_{\substack{    |\alpha_1|=|\gamma_1|=\kappa \\ \af \in \{\circ, \star\}  }}
 \left[ \Pi^1_{\mathbf{a}^{1,\af}_{\gamma_1,\alpha_1}}\right]^{\af}(f,g) & k_2<\kappa\leq k_1
 \end{cases}
 \end{split}
\end{equation}
with  families $ \{\upsilon_{z}, \varphi_\zeta\in C\Psi_{z}^{(\kappa,\min\{\kappa,k_2\}),\eps;0,0}:z\in Z^{\mathbf{d}} \}$ if $\kappa \leq k_2$.

The assumption $a(\kappa)$ is then removed by an inductive argument.  Recall that $k_1\geq k_2$. Let $0\leq \kappa <k_1$ and 
assume that the representation \eqref{e:representation2} holds true for $k=(\kappa, \min\{\kappa,k_2\})$. Let $\widetilde{\Lambda}(f,g)$ be the form obtained by subtracting from $\Lambda$ the second line of \eqref{e:representation2}.
Then
\[\widetilde
\Lambda(f,g) =   \int_{Z^{\mathbf{d}}}     \langle f,  \varphi_{z}\rangle  \langle \upsilon_{z},g \rangle  \, \d \mu(z), \qquad \left\{\upsilon_{z}, \varphi_z\in C\Psi_{z}^{(\kappa,\min\{\kappa,k_2\}),\eps;0,0}:z\in Z^{\mathbf{d}}\right\}
\]
coincides with the $\kappa$-th order cancellative part of $\Lambda$, is a bi-parameter wavelet form of type \eqref{e:cancformbip} and satisfies assumption $a(\kappa+1)$ of having all the relevant paraproducts up to order $\kappa$ vanishing, see Remark \ref{r:indbp}.
We may thus apply the main step to $\widetilde
\Lambda(f,g)$ with $\kappa+1$ in place of $\kappa$, resulting in \eqref{e:indu0}, and obtain that
\[
\begin{split}
 \Lambda(f,g)   &=   \widetilde
\Lambda(f,g)+
\sum_{  \kappa'\leq \min\{\kappa,k_2\}}  \sum_{\substack {\gamma=(\gamma_1,\gamma_2) \\ |\gamma_1|=|\gamma_2|=\kappa' \\ \af \in \vaf }} \left[\Pi_{ b_{\gamma}^{\af},\gamma}\right]^{\af} (f,g) 
+ \sum_{\substack {  |\gamma_1| \leq \kappa \\   |\gamma_2| \leq \min\{\kappa,k_2\}  }}
\sum_{\substack{ \iota \in \{1,2\}\\ |\alpha_\iota|=|\gamma_\iota| \\ \af \in \{\circ, \star\} }}
 \left[ \Pi^\iota_{\mathbf{a}^{\iota,\af}_{\gamma_\iota,\alpha_\iota}}\right]^{\af}(f,g)
\\ & = \int_{Z^{\mathbf{d}}}     \langle f,  \varphi_{z}\rangle  \langle \widetilde \upsilon_{z},g \rangle  \, \d \mu(z)
\\ &  + \sum_{  \kappa'\leq \min\{\kappa+1,k_2\}}  \sum_{\substack {\gamma=(\gamma_1,\gamma_2) \\ |\gamma_1|=|\gamma_2|=\kappa' \\ \af \in \vaf} } \left[\Pi_{ b_{\gamma}^{\af},\gamma}\right]^{\af} (f,g) 
+ \sum_{\substack {  |\gamma_1| \leq \kappa+1 \\   |\gamma_2| \leq \min\{\kappa+1,k_2\}  }}
\sum_{\substack{ \iota \in \{1,2\}\\ |\alpha_\iota|=|\gamma_\iota|\\ \af \in \{\circ, \star\} }}
 \left[ \Pi^\iota_{\mathbf{a}^{\iota,\af}_{\gamma_\iota,\alpha_\iota}}\right]^{\af}(f,g)
\end{split}
\]
with  $\{\widetilde{\upsilon}_{z}, \varphi_z\in C\Psi_{z}^{(\kappa+1,\min\{\kappa+1,k_2\}),\eps;0,0}:z\in Z^{\mathbf{d}}\}$. This
  achieves \eqref{e:representation2} for $\Lambda$ with $k=(\kappa+1,\min\{\kappa+1,k_2\})$, thus completing the inductive step and the proof of Theorem \ref{t:T12p}. 

 \section{Weighted Estimates for Intrinsic Operators} 
 \label{s8}
 This section contains the proofs of quantitative, and in some cases sharp, weighted estimates for the four types of summands occurring in the representation \eqref{e:representation2}: see Propositions \ref{p:wbTcanc} and \ref{p:full}. Throughout,  $[w]_{A_p}$  denotes the standard product weight characteristic on $\R^{\mathbf d}=\R^{d_1} \times \R^{d_2}$, see for example \cite{FP,HLP}.
 
\subsection{Quantitative Bounds for Bi-Parameter Calder\'on-Zygmund Model Operators} 
To begin with, the operator $T$ appearing in the following proposition is the adjoint of the first summand in   \eqref{e:representation2}, in the basic case $k=0$.
\begin{proposition} \label{p:wbTcanc} For $\delta>0$ and   $\left\{\upsilon_{z}\in C\Psi_{z}^{\delta;0,0}:z\in Z^{\mathbf{d}}\right\}$ consider the operator \[
Tf  =   \int\displaylimits_{Z^{\mathbf{d}}}     \langle f,  \mathsf{Sy}_{z_1}\varphi_1 \otimes \mathsf{Sy}_{z_2}\varphi_2 \rangle   \upsilon_{z }  {\d\mu(z)}.
\] Then   
$\displaystyle
\|T\|_{L^p(\R^{\mathbf d}; w)} \lesssim [w]_{A_p}^{\max\left\{2, 1+\frac{2}{p-1}\right\}}
$ for all $1<p<\infty$, and
this estimate   is sharp when $p\geq 3$.
\end{proposition}
Next,  adjoints to the full and partial paraproduct terms in \eqref{e:representation2} are treated: compare with the definitions in \eqref{e:paragammas}.
\begin{proposition} \label{p:full} Let $D\geq 8(d_1+d_2)$. Fix $b\in \mathrm{BMO}(\R^{\mathbf d})$,  
$ \{\vartheta_{z_j}\in C\Psi_{z_j}^{D,1;1}:z_j\in Z^{d_j} \}$, 
$ \{\upsilon_{z_j},\psi_{z_j}\in C\Psi_{z_j}^{D,1;0}:z_j\in Z^{d_j}\}$ for $j=1,2.$ Then,   the operators \[\begin{split}
&\Pi_{(0,0),b} f    \coloneqq \int\displaylimits_{Z^{\mathbf d}} \langle b,\upsilon_{z_1}\otimes \upsilon_{z_2}\rangle  \langle f,\vartheta_{ z_1} \otimes \vartheta_{z_2}\rangle  \psi_{ z_1}  \otimes \psi_{z_2}  \,\d \mu(z), 
\\
&\Pi_{(0,1),b} f  \coloneqq  \int\displaylimits_{Z^{\mathbf d}} \langle b,\upsilon_{z_1}\otimes \upsilon_{z_2}\rangle  \langle f,\vartheta_{ z_1} \otimes \psi_{z_2}\rangle  \psi_{ z_1}  \otimes \vartheta_{z_2}  \,\d \mu(z),\end{split} 
\]   satisfy  the estimates 
\[\begin{split}
&\|\Pi_{(0,0),b}\|_{L^p(\R^{\mathbf d}; w)} \lesssim [w]_{A_p}^{\frac{\scriptstyle\max\{3,2p\}}{ \scriptstyle p-1}}\|b\|_{\mathrm{BMO}(\R^{\mathbf d})}, \qquad 1<p<\infty,\\ &  \|\Pi_{(0,1),b}\|_{L^p(\R^{\mathbf d}; w)} \lesssim [w]_{A_p}^{\frac{\scriptstyle\max\{2p+3,4p,5p-3\}}{ \scriptstyle 2(p-1)}} \|b\|_{\mathrm{BMO}(\R^{\mathbf d})}, \qquad 1<p<\infty .\end{split}
\]
\end{proposition}
The last proposition  concerns  adjoints to the half paraproduct terms in \eqref{e:representation2}, see  \eqref{e:hpp}. 
\begin{proposition} \label{p:half} Let $0<\delta<1,$  $\mathbf{a}\in \mathcal C(Z^{d_2}\times Z^{d_2}; \mathrm{BMO}(\R^{\mathbf d}))$,  and fix families
\[ \{\vartheta_{z_1}\in C\Psi_{z_1}^{1,1;1}:z_1\in Z^{d_1} \}, \qquad\{\upsilon_{z_1}, \psi_{z_1}\in C\Psi_{z_1}^{1,1;0}:z_1\in Z^{d_1} \},\qquad   \{ \psi_{z_2}\in C\Psi_{z_2}^{1,1;0}:z_2\in Z^{d_2}\}.\]
  Then,  the operator \[\begin{split}
 \Pi_{\mathbf{a}} f  & \coloneqq  \int\displaylimits_{  Z^{d_2}_{z_2}}\int\displaylimits_{  Z^{d_2}_{\zeta_2}} [z_2,\zeta_2)]_\delta\int\displaylimits_{Z^{d_1}_{z_1}} \left\langle \mathbf{a}(z_2,\zeta_2)\big),\upsilon_{z_1} \right\rangle     \langle f,\vartheta_{ z_1} \otimes \psi_{z_2}\rangle  \psi_{ z_1} \otimes \psi_{\zeta_2} \,\d\mu(z_1)\d\mu(\zeta_2) \d\mu(z_2),
\end{split} 
\]   satisfies   the estimate 
\[
\|\Pi_{\mathbf{a}}\|_{L^p(\R^{\mathbf{d}};w)} \lesssim [w]_{A_p}^{\frac{\scriptstyle 1}{\scriptstyle 2} + \max\left\{\scriptstyle\frac{\scriptstyle \max\{2,p\}}{\scriptstyle p-1},\frac{\scriptstyle 3}{\scriptstyle 2}\right\}}\|\mathbf{a}\|_{\mathcal C(Z^{d_2}\times Z^{d_2}; \mathrm{BMO}(\R^{\mathbf d}))}.
\]
\end{proposition}

The proofs of Propositions    \ref{p:wbTcanc}, \ref{p:full} and \ref{p:half} are collected in Subsection \ref{ss:wproof}. Along the way, we will make use of  sharp weighted bounds for the intrinsic square function \eqref{e:Sdelta2}, as well as the mixed square-maximal operators
\[
\begin{split}
& \mathrm{SM}  (x) = \left(\, \int\displaylimits_{(0,\infty)} \sup_{t_2>0} \sup_{\psi \in \Psi^{\delta; 0,1}_{((x_1,t_1),(x_2,t_2))}}   |\langle f, \psi \rangle|^2 \frac{\d t_1  }{t_1}\right)^{\frac12},
\\ 
& \mathrm{MS}  (x) =\sup_{t_1>0} \left(\, \int\displaylimits_{(0,\infty)} \sup_{\psi \in \Psi^{\delta; 1,0}_{((x_1,t_1),(x_2,t_2))}}   |\langle f, \psi \rangle|^2 \frac{\d t_2  }{t_2}\right)^{\frac12}, \qquad x=(x_1,x_2) \in \R^{\mathbf d}
\end{split}
\]
which enter the $L^p$ and weighted theory of the full and partial paraproduct terms. The square-maximal and maximal-square operators appearing below generalize those introduced in \cite{MPTT1,MPTT06}. There seem to be no pre-existing weighted  estimate in past literature, thus  our  results are stated as a theorem.
\begin{theorem} \label{t:w2}The operator norm bound \begin{equation}
\label{e:wSS}
\displaystyle\|\mathrm{SS} \|_{L^p(\R^{\mathbf d}; w)} \lesssim [w]_{A_p}^{\max\left\{1,\frac{2}{p-1}\right\}}, \qquad 
\displaystyle\|\mathrm{SM} \|_{L^p(\R^{\mathbf d}; w)}, \|\mathrm{MS} \|_{L^p(\R^{\mathbf d}; w)} \lesssim [w]_{A_p}^{\frac{1}{p-1}\max\left\{2,\frac{p+1}{2}\right\}}, 
\end{equation}
holds for all $0<\delta\leq 1$  and $1<p<\infty$. Furthermore, the exponent of \eqref{e:wSS} may not be improved for a generic weight $w$. 
\end{theorem}
The $\mathrm{SS}$ bound  in Theorem \ref{t:w2} is the bi-parameter analogue of \cite[Theorem 1.2]{LLP}. Its proof is given in the concluding  Subsection \ref{ss:pftw2} below. 
   Another inequality that will be used a few times in Subsection \ref{ss:wproof} is a lower bound for the smaller tensor product square function
 \[
 \mathrm{SS}_{\otimes} f(x_1,x_2) = \left(\, \int\displaylimits_{(0,\infty)^2} 
 \left|\left\langle f, \mathsf{Sy}_{(x_1,t_1)}\varphi_1  \otimes \mathsf{Sy}_{(x_2,t_2)}\varphi_2 \right\rangle\right|^2 \frac{\d t_1 \d t_2}{t_1 t_2}\right)^{\frac12}
 \]  
  associated  with  the wavelets $\varphi_1,\varphi_2$  from  \eqref{e:varphi2p}. The proof is a simple iteration argument, and is given immediately.
\begin{proposition} \label{t:w2rev} 
$ 
\displaystyle\| f\|_{L^p(\R^{\mathbf d}; w)} \lesssim [w]_{A_p}
\displaystyle\|\mathrm{SS}_{\otimes} f\|_{L^p(\R^{\mathbf d}; w)} $ for all $1<p<\infty.$
\end{proposition}
\begin{proof} Apply the main result of \cite{WTam}, see also \cite[Theorem 2.7]{LLP2}, first on each $x_1$-fiber in the second parameter, and subsequently in vector-valued form in the second parameter to see that
\[
\begin{split}
&\quad \displaystyle\| f\|_{L^p(\R^{\mathbf d}; w)} \lesssim [w]_{A_p}^{\frac 12}  \left\| \langle f(x_1,\cdot),  \mathsf{Sy}_{(x_2,t_2)}\varphi_2\rangle\right\|_{L^p(w(x_1,x_2); L^2(\mathrm{d} t_2/t_2))}
\\ & \lesssim [w]_{A_p}  \left\| \left\langle  \langle f,  \mathsf{Sy}_{(x_2,t_2)}\varphi_2\rangle_2,  \mathsf{Sy}_{(x_1,t_1)}\varphi_1\right\rangle_1\right\|_{L^p(w(x_1,x_2); L^2(\mathrm{d} t_1/t_1;L^2(\mathrm{d} t_2/t_2)) )} =  [w]_{A_p} \|\mathrm{SS}_{\otimes} f\|_{L^p(\R^{\mathbf d}; w)}
\end{split}
\]
as claimed.
\end{proof}

\subsection{Proofs of Propositions \ref{p:wbTcanc}, \ref{p:full} and \ref{p:half}}
\label{ss:wproof}
\begin{proof}[Proof of Proposition \ref{p:wbTcanc}] Sharpness of $[w]_{A_p}^2$  for $p\geq 3$ follows by taking the tensor product of two   counterexamples to sharpness of $[w]_{A_p}$ in one parameter. For the rest of the proof, we claim the pointwise bound 
\begin{equation}
\label{e:pwSS} \mathrm{SS}_{\otimes} (Tf) \lesssim \mathrm{SS} f.
\end{equation}
Assuming \eqref{e:pwSS} holds, 
\[
\|Tf\|_{L^p(\R^{\mathbf d}; w)} \lesssim [w]_{A_p} \|\mathrm{SS}_{\otimes} (Tf) \|_{L^p(\R^{\mathbf d}; w)} \lesssim   [w]_{A_p} \|\mathrm{SS}f \|_{L^p(\R^{\mathbf d}; w)} \lesssim   [w]_{A_p}^{1+\max\left\{1,  \frac{2}{p-1}\right\}} \| f \|_{L^p(\R^{\mathbf d}; w)}
\]
thanks to an application of Proposition \ref{t:w2rev} in the first step and  of Theorem \ref{t:w2} in the last. This proves Proposition \ref{p:wbTcanc} up to the verification of claim \eqref{e:pwSS}, which follows. Fix \[\zeta=((\xi_1,\sigma_1),(\xi_2,\sigma_2))\in Z^{ \mathbf{d}}.\]
 Using the notation \eqref{e:varphi2p} for $\varphi_\zeta$,  writing $z=((x_1,t_1),(x_2,t_2))\in Z^{ \mathbf{d}}$, and making the usual change of variable 
\[
\begin{split}
&\langle Tf, \varphi_\zeta\rangle = \int_{z\in Z^{ \mathbf{d}}} \langle f,\varphi_z\rangle \langle \upsilon_z, \varphi_\zeta\rangle  \, \d\mu(z) = \langle f, \psi_\zeta\rangle, \\
&\psi_\zeta\coloneqq \int\displaylimits_{\substack{( \alpha_1,\beta_1) \in Z^{d_1}  \\ ( \alpha_2,\beta_2) \in Z^{d_2} }}\langle \upsilon_{((\xi_1+\alpha_1 \sigma_1,\beta_1 \sigma_1),(\xi_2+\alpha_2 \sigma_2,\beta_2 \sigma _2))}, \varphi_\zeta\rangle \varphi_{((\xi_1+\alpha_1 \sigma_1,\beta_1 \sigma_1),(\xi_2+\alpha_2 \sigma_2,\beta_2 \sigma_2))}  \frac{ \d\beta_2 \d \alpha_2  \d \beta_1\d\alpha_1}{\beta_1\beta_2}.
\end{split}
\]
Applying Lemma \ref{l:ttstar2}, 
\[
\left|\langle \upsilon_{((\xi_1+\alpha_1 \sigma_1,\beta_1 \sigma_1),(\xi_2+\alpha_2 \sigma_2,\beta_2 \sigma _2))}, \varphi_\zeta\rangle \right| \lesssim \left(\prod_{j=1}^2 [(\alpha_j,\beta_j)]_{\frac\delta 2}\right) 
\]
whence by Lemma \ref{l:average2}, $\psi_\zeta\in C\Psi^{\frac{\delta}{2}; 0,0}_\zeta$, and \eqref{e:pwSS} follows immediately from the definition of the intrinsic square function $\mathrm{SS}$.
\end{proof}
\begin{proof}[Proof of Proposition \ref{p:full}] Let  $\sigma\coloneqq w^{-\frac{1}{p-1}} $ be the dual weight to $w\in A_p$, so that $[\sigma]_{A_{p'}}=[w]_{A_p}^{\frac{1}{p-1}}$. Recall that  $ \mathrm{M}_{d_1,d_2}  $ is the bi-parameter maximal function on $\R^{\mathbf{d}}$.
The proof for $ \Pi_{(0,0),b}$ begins with an appeal to $H^1-\mathrm{BMO}$ duality, leading to   
\begin{align*}
\left\lvert \langle \Pi_{(0,0),b} f,g \rangle \right\vert & \leq \left\Vert b\right\Vert_{\mathrm{BMO}(\R^{\mathbf{d}})}\left\Vert \mathrm{SS}_{\otimes}\left(\int_{Z^{\mathbf d}}  \langle f,\vartheta_{ z_1} \otimes \vartheta_{z_2}\rangle \langle \psi_{ z_1}  \otimes \psi_{z_2}, g\rangle \upsilon_{z_1} \otimes \upsilon_{z_2}  \,\d \mu(z) \right)\right\Vert_{L^1(\R^{\mathbf{d}})}\\
& \lesssim  \left\Vert b\right\Vert_{\mathrm{BMO}(\R^{\mathbf{d}})}\left\Vert \mathrm{M}_{d_1,d_2}   (f) \mathrm{SS}g\right\Vert_{L^1(\R^{\mathbf{d}})} \\ &\leq\left\Vert b \right\Vert_{\mathrm{BMO}(\R^{\mathbf{d}})}\left\Vert \mathrm{M}_{d_1,d_2} (f)\right\Vert_{L^p(\R^{\mathbf{d}};w)} \left\Vert \mathrm{SS}(g)\right\Vert_{L^{p'}(\sigma,\R^{\mathbf{d}})}\\
& \lesssim [w]_{A_p}^{\frac{\scriptstyle \max\{3,2p\}}{\scriptstyle p-1}}\left\Vert b \right\Vert_{\mathrm{BMO}(\R^{\mathbf{d}})}\left\Vert f\right\Vert_{L^p(\R^{\mathbf{d}};w)} \left\Vert g\right\Vert_{L^{p'}(\sigma,\R^{\mathbf{d}})}.
\end{align*}
 The passage to the second line is justified by a pointwise bound, whose proof is similar to \eqref{e:pwbmix} below, and is omitted.  In the  last line, Theorem \ref{t:w2} has been appealed to, and to the quantitative weighted estimate for  the strong maximal function and square functions.  The claimed estimate for $\Pi_{(0,0),b}$ then follows by duality.

The proof for  $ \Pi_{(0,1),b}$  is similar. Preliminarily notice that
\[ 
\begin{split}
&[(\alpha_1,\beta_1)]_{d_1} [(\alpha_2,\beta_2)]_{d_2} |\langle f, \vartheta_{(x_1+\alpha_1 t_1, \beta_1 t_1)} \otimes \psi_{(x_2+\alpha_2 t_2, \beta_2 t_2)}\rangle | \lesssim   \sup_{\psi \in \Psi^{\delta; 0,1}_{((x_1,t_1),(x_2,t_2))}}   |\langle f, \psi \rangle|, 
\\ &
[(\alpha_1,\beta_1)]_{d_1} [(\alpha_2,\beta_2)]_{d_2} |\langle g, \psi_{(x_1+\alpha_1 t_1, \beta_1 t_1)} \otimes \vartheta_{(x_2+\alpha_2 t_2, \beta_2 t_2)}\rangle | \lesssim   \sup_{\psi \in \Psi^{\delta; 1,0}_{((x_1,t_1),(x_2,t_2))}}   |\langle g, \psi \rangle|.
\end{split}
\]
As $D\geq 8d_1, 8d_2$, Lemma \ref{l:ttstar1} applied componentwise to bound $\langle  \upsilon_{z_1} \otimes \upsilon_{z_2}, \mathsf{Sy}_{(x_1,t_1)}\varphi_1  \otimes \mathsf{Sy}_{(x_2,t_2)}\varphi_2\rangle $, with $z_j=(x_j+\alpha_j t_j, \beta_j t_j)$, $j=1,2$ then yields
\begin{equation}
\label{e:pwbmix}
\begin{split} & \quad 
\left|\left\langle \int_{Z^{\mathbf d}}  \langle f,\vartheta_{ z_1} \otimes \psi_{z_2}\rangle \langle \psi_{ z_1}  \otimes \theta_{z_2}, g\rangle \upsilon_{z_1} \otimes \upsilon_{z_2}  \,\d \mu(z), \mathsf{Sy}_{(x_1,t_1)}\varphi_1  \otimes \mathsf{Sy}_{(x_2,t_2)}\varphi_2 \right \rangle \right| \\ &\lesssim \left(\sup_{\psi \in \Psi^{\delta; 0,1}_{((x_1,t_1),(x_2,t_2))}}   |\langle f, \psi \rangle| \right) \left( \sup_{\psi \in \Psi^{\delta; 1,0}_{((x_1,t_1),(x_2,t_2))}}   |\langle g, \psi \rangle|\right).
\end{split}
\end{equation}
The proof proper begins now. Using $H^1-\mathrm{BMO}$ duality again, followed by \eqref{e:pwbmix} and one application of $L^2-L^\infty$ H\"older inequality in each parameter,
\begin{align*}
\left\lvert \langle \Pi_{(0,1),b} f,g \rangle \right\vert & \leq \left\Vert b\right\Vert_{\mathrm{BMO}(\R^{\mathbf{d}})}
\left\Vert \mathrm{SS}_{\otimes}\left(\int_{Z^{\mathbf d}}  \langle f,\vartheta_{ z_1} \otimes \psi_{z_2}\rangle \langle \psi_{ z_1}  \otimes \theta_{z_2}, g\rangle \upsilon_{z_1} \otimes \upsilon_{z_2}  \,\d \mu(z) \right)\right\Vert_{L^1(\R^{\mathbf{d}})}
\\
& \leq  \left\Vert b\right\Vert_{\mathrm{BMO}(\R^{\mathbf{d}})}\left\Vert \mathrm{SM}(f) \mathrm{MS}(g)\right\Vert_{L^1(\R^{\mathbf{d}})}\\
& \leq \left\Vert b\right\Vert_{\mathrm{BMO}(\R^{\mathbf{d}})}\left\Vert \mathrm{SM}(f)\right\Vert_{L^p(\R^{\mathbf{d}};w)} \left\Vert \mathrm{MS}(g)\right\Vert_{L^{p'}(\sigma,\R^{\mathbf{d}})}\\
& \lesssim  [w]_{A_p}^{\frac{\scriptstyle\max\{2p+3,4p,5p-3\}}{ \scriptstyle 2(p-1)}}\left\Vert b \right\Vert_{\mathrm{BMO}(\R^{\mathbf{d}})}\left\Vert f\right\Vert_{L^p(\R^{\mathbf{d}};w)} \left\Vert g\right\Vert_{L^{p'}(\sigma,\R^{\mathbf{d}})}.
\end{align*}
In the last line, the quantitative weighted estimates of the operators $\mathrm{SM}$ and $\mathrm{MS}$ from Theorem \ref{t:w2} have been called upon.  By duality, this estimate proves the claimed bound of $\Pi_{(0,1),b}$ on $L^p(w)$ and completes the proof of the proposition.
\end{proof}
\begin{proof}[Proof of Proposition \ref{p:half}] This proof relies on the auxiliary operators
\[
\mathrm{P}_{b} h(y_1) =\int_{Z^{d_1}} \langle b, \upsilon_{z_1}\rangle \langle h, \vartheta_{z_1} \rangle \psi_{z_1} (y_1) \, 
\d \mu(z_1), \qquad y_1 \in \R^{d_1}
\]
which is a paraproduct with symbol $b\in \mathrm{BMO}(\R^{d_1}) $ in the first parameter, and
\[
\mathrm{S}_{(2), (\alpha_2,\beta_2)} h(y_2)= \left(\int_{0}^\infty \left|\left\langle g, \mathsf{Sy}_{(y_2+\alpha_2t_2,\beta_2 t_2)} \varphi_2\right\rangle \right|^2 \frac{\d t_2}{t_2} \right)^{\frac12}, \qquad y_2 \in \R^{d_2}, \, (\alpha_2,\beta_2) \in Z^{d_2}
\]
which is a shifted square function in the second parameter with  smooth, compactly supported mother wavelet $\varphi_2$ as in \eqref{e:varphi2p}; the simplified notation $\mathrm{S}_{(2)}$ is used in place of $\mathrm{S}_{(2),(0,1)}$.   The main results of \cite{Brocchi20,LLP} yield the operator norm bounds
\begin{equation}
\label{e:weight}
\|\mathrm{S}_{(2), (\alpha_2,\beta_2)}\|_{L^{p}(\R^{d_2};W)} \lesssim_\eps (\min\{1,\beta_2\})^{-\eps} [W]_{A_p}^{\max\{\frac12,\frac{1}{p-1}\}}
\end{equation}
for all $\eps>0$,
where $W$ is a weight on $\R^{d_2}$ and $[W]_{A_p}$ denotes the corresponding weight characteristic.
Then
\[
\Pi_{\mathbf{a}}f(u) = \int\displaylimits_{Z^{d_2}_{z_2}}\int\displaylimits_{Z^{d_2}_{\zeta_2}} [z_2,\zeta_2]_\delta \mathrm{P}_{a(z_2,\zeta_2)} (\langle f,\psi_{z_2} \rangle_2)(u_1) \otimes \psi_{\zeta_2}(u_2) \d \mu(\zeta_2) \d \mu (\zeta_1), \qquad u=(u_1,u_2) \in \mathbb R^{\mathbf d}. 
\]
 A calculation involving Lemma \ref{l:ttstar1} applied to the inner product $\langle\psi_{\zeta_2}, \mathsf{Sy}_{(y_2,t_2)}\varphi_2\rangle$ followed by the  change of variables 
 $
z_2=(y_2+a_2 t_2,b_2t_2)
$, $\zeta_2=(y_2+\alpha_2 t_2,\beta_2t_2) $
 then yields
\[
\begin{split} 
\mathrm{S}_{(2)} [\Pi_{\mathbf{a}}f](y_1,y_2)&\lesssim \int\displaylimits_{ \substack{\omega_2\coloneqq(\alpha_2,\beta_2)\in Z^{d_2} \\ w_2\coloneqq (a_2,b_2)\in Z^{d_2}}}\left(   \int\displaylimits_0^\infty   \left| \mathrm{P}_{\mathbf{a}((y_2+a_2 t_2,b_2t_2),(y_2+\alpha_2 t_2,\beta_2 t_2))}\left(\langle f,\psi_{ (y_2+a_2 t_2,b_2t_2)}\rangle_2\right) (y_1)  \right|^2 \frac{\d t_2}{t_2}\right)^{\frac12}\\ &\times  [\omega_2]_1  [w_2,\omega_2]_\delta\,\d\mu(w_2) \d\mu(\omega_2) .
\end{split}
\]
Applying the  reverse square function bound of \cite{WTam} in the second parameter,  
followed by  the sharp weighted estimate for the vector-valued paraproduct $\mathrm{P}_{\mathbf{a}((y_2+a_2 t_2,b_2t_2),(y_2+\alpha_2 t_2,\beta_2 t_2))}$ to pass to the second line, and finally appealing to \eqref{e:weight} with  choice $\eps=\frac\delta 2$, we obtain
\[
\begin{split}
& \quad 
\|\Pi_{\mathbf{a}}f\|_{L^p(\R^{\mathbf d}; w)} \lesssim [w]_{A_p}^{\frac12} \|\mathrm{S}_{(2)} \Pi_{\mathbf{a}}f\|_{L^p(\R^{\mathbf d}; w)}\\ &
\lesssim [w]_{A_p}^{\frac12 +\max\{1,\frac{1}{p-1}\}}
 \|\mathbf{a}\| 
  \hspace{-.2in}\int\displaylimits_{\substack{(\alpha_2,\beta_2)\in Z^{d_2} \\ (a_2,b_2)\in Z^{d_2} }} [(\alpha_2,\beta_2)]_{1}[(a_2,b_2),(\alpha_2,\beta_2)]_{\delta} \left\|  \mathrm{S}_{(2),(a_2,b_2)} f\right\|_{L^p(\R^{\mathbf d}; w)} \frac{\d a_2 \d b_2}{b_2} \frac{\d \alpha_2 \d \beta_2}{\beta_2}
\\ &  \lesssim [w]_{A_p}^{\frac12 +\max\{1,\frac{1}{p-1}\} +  \max\{\frac{1}{2},\frac{1}{p-1}\}}
 \|\mathbf{a}\| 
\|f\|_{L^p(\R^{\mathbf d}; w)}.
\end{split}
\]
For display reasons, above $\|\mathbf{a}\| $ stands for $\|\mathbf{a}\|_{\mathcal C(Z^{d_2}\times Z^{d_2}; \mathrm{BMO}(\R^{\mathbf d}))}$.
The proof of Proposition \ref{p:half} is thus complete.
\end{proof}

\subsection{Proof of Theorem \ref{t:w2}}\label{ss:pftw2} Sharpness of the exponent follows by tensor product of the usual one parameter examples. The one parameter square function example  is   discussed in \cite{LLP} and references therein, while  the example for the one parameter maximal operator is entirely classical.

The proof  of the upper bound  is 
analogous for all three operators, as it proceeds by reduction to iteration of one parameter, vector-valued weighted bounds. To fix ideas, the argument is given for $\mathrm{SS}$, which is the most difficult case.

Fix  $f\in L^\infty_0(\R^{\mathbf d})$ and let $\left\{\psi_{(x_1,t_1), (x_2,t_2)}\in  \Psi^{\delta;0,0}_{(x_1,t_1), (x_2,t_2)}; x_j \in \R^{d_j}, 0<t_j<\infty, j=1,2\right\}$ be a family linearizing the supremum in \eqref{e:Sdelta2}.
Throughout this proof, $\eta\coloneqq \frac{\delta}{16}$.
The first step consists of a  decomposition of the linearizing family  into wavelets with compact frequency support in one of the parameters. Let $\alpha\in \mathcal{S}(\R^{d_1})$ be a radial function with 
\[
\supp \widehat \alpha \subset \mathsf{B}_{(0,2)} \setminus \mathsf{B}_{(0,\frac12)} ,\qquad 
\int_{-\infty}^\infty \widehat \alpha(2^s \xi) \, {\d s}  =\cic{1}_{\R^{d_1}\setminus \{0\}}(\xi),
\]
and also let $\beta  \in \mathcal{S}(\R^{d_1})$ satisfy 
\[
\supp \widehat \beta \subset \mathsf{B}_{(0,3)} \setminus \mathsf{B}_{(0,\frac13)}, \qquad  
\widehat \beta (\xi)= 1\quad \forall \xi \in \supp \widehat{\alpha}.
\]
Set
\[
\alpha_s= \mathsf{Dil}^1_{2^s} \alpha, \qquad \beta_s= \mathsf{Dil}^1_{2^s} \beta,\qquad \psi_{(x_1,t_1), (x_2,t_2)}^s\coloneqq 2^{\eta|s|} \psi_{(x_1,t_1), (x_2,t_2)} *_1 \alpha_{s+\log t_1} 
\]
so that it is understood that $*_j$ denotes convolution in the $j$-th parameter only, and note   that the scale of the parameter
in $\alpha_s$, $\beta_s$ is logarithmic. For instance $\alpha_{s+\log t_1}$ below has Fourier support in the annulus $\sim {t_1}^{-1} 2^{-s}$.
\begin{lemma} \label{l:sdec} For all $s\in \R$, $ x_j \in \R^{d_j},$ $0<t_j<\infty,$ $ j=1,2$ we have $\psi_{(x_1,t_1), (x_2,t_2)}^s\in C\Psi^{\eta;0,0}_{(x_1,t_1), (x_2,t_2)}$.
\end{lemma}
\begin{proof} By  bi-parametric invariance of the assumption and assertion, it suffices to prove the case $x_j=0_{\R^{d_j}}, t_j=1$ for $j=1,2$.  For simplicity write $\psi$ in place of $\psi_{(x_1,t_1), (x_2,t_2)}$.
Applying Lemma \ref{l:ttstar1} for each fixed $y_2$ gives
\begin{equation}
\label{e:ldecs1}
\left|\psi*_1 \alpha_{s}(y_1,y_2)\right| =\left| \langle \psi(\cdot, y_2), \mathsf{Tr}_{y_1}\alpha_{s} \rangle   \right| \lesssim \frac{[(0,1),(y_1,s)]_{8\eta}}{ \langle y_2 \rangle^{(d_2+8\eta)}} \lesssim 
\frac{2^{-8\eta|s|} }{\langle y_1 \rangle^{(d_1+8\eta)}\langle y_2 \rangle^{(d_2+8\eta)}}
\end{equation}
as $\langle y_2 \rangle^{d_2+8\eta}\psi^{[2,y_2]}(\cdot) \in \Psi^{\delta;0}_{(0,1)}$ and $\mathsf{Tr}_{y_1}\alpha_{s}\in \Psi^{\delta;0}_{(y_1,s)}$. The last inequality is best seen by verifying the cases $s\geq 0$, $s<0$ separately.
Using the Fourier support and normalization of $\alpha_s$, similarly  
\[
\left|\nabla_1(\psi*_1 \alpha_{s})(y_1,y_2)\right|
= \left| \psi*_1 \nabla \alpha_{s}(y_1,y_2)\right|
 \lesssim 
\frac{2^{-s-8\eta|s|} }{\langle y_1 \rangle^{(d_1+8\eta)}\langle y_2 \rangle^{(d_2+8\eta)}}.
\]
If $0<|h|<1$ then, by the mean value theorem
\begin{equation}
\label{e:ldecs2}
\begin{split}
&\quad \left|\psi*_1 \alpha_{s}(y_1+h,y_2)- \psi*_1 \alpha_{s}(y_1,y_2) \right|\\ & \lesssim |h|^\eta \left( \sup_{|u_1|\sim |y_1| }\nabla_1(\psi*_1 \alpha_{s})(u_1,y_2) \right)^{\eta} \left(\left|\psi*_1 \alpha_{s}(y_1+h,y_2)\right|+\left| \psi*_1 \alpha_{s}(y_1,y_2) \right|\right)^{1-\eta}    
\\ &\lesssim |h|^\eta \frac{2^{|s| \left[\eta- 8\eta(1-\eta)\right] }}{\langle y_1 \rangle^{(d_1+8\eta)}\langle y_2 \rangle^{(d_2+8\eta)}}\leq 2^{-|s|\eta} \frac{|h|^\eta}{\langle y_1 \rangle^{(d_1+8\eta)}\langle y_2 \rangle^{(d_2+8\eta)}}
\end{split}\end{equation}
using the elementary inequality $6\eta>8\eta^2$. The inequality  
\[
\left|\psi*_1 \alpha_{s}(y_1,y_2+h)- \psi*_1 \alpha_{s}(y_1,y_2) \right| \lesssim  \frac{|h|^\delta}{\langle y_1 \rangle^{(d_1+\delta)}\langle y_2 \rangle^{(d_2+\delta)}}
\]
is immediate from \eqref{e:bip1} and averaging, so that another interpolation with \eqref{e:ldecs1} yields
\begin{equation}
\label{e:ldecs3} 
\left|\psi*_1 \alpha_{s}(y_1,y_2+h)- \psi*_1 \alpha_{s}(y_1,y_2) \right| \lesssim \frac{|h|^{\frac\delta 2}2^{-4\eta|s|}}{ \langle y_1 \rangle^{(d_1+8\eta)}\langle y_2 \rangle^{(d_2+8\eta)}} \leq 2^{-\eta|s|} \frac{|h|^{\eta}}{ \langle y_1 \rangle^{(d_1+\eta)}\langle y_2 \rangle^{(d_2+\eta)}}.
\end{equation}
Collecting \eqref{e:ldecs1},  \eqref{e:ldecs2}, and  \eqref{e:ldecs3}, and comparing with \eqref{e:bip1}, completes the proof. 
\end{proof} \noindent
The definitions of $\alpha_s$ and $\beta_s$ lead to the equalities
\[
\psi_{(x_1,t_1), (x_2,t_2)}= \int_{-\infty}^\infty  2^{-\eta|s|}\psi^s_{(x_1,t_1), (x_2,t_2)} \, \d s, \qquad  \langle f,  \psi^s_{(x_1,t_1), (x_2,t_2)}\rangle = \langle f*_1 \beta_{s+\log t_1},  \psi^s_{(x_1,t_1), (x_2,t_2)}\rangle. \] 
Therefore, in view of Lemma \ref{l:sdec}, and using the convergence of the geometric integral, it will be enough to prove the same estimate for the operator
\[
O_sf(x) = \left(\, \int\displaylimits_{(0,\infty)^2}   |\langle f*_1 \beta_{s+\log t_1},  \psi^s_{(x_1,t_1), (x_2,t_2)} \rangle|^2 \frac{\d t_1 \d t_2}{t_1 t_2}\right)^{\frac12}
\]
uniformly in the parameter $s\in \mathbb \R$, which will be kept fixed until the end of the proof. 
The operator $O_s$ is estimated  relying on the auxiliary family of square functions with parameter $t_1>0$\[
\mathrm{S}_{t_1}h(x_1,x_2)\coloneqq  \left(\int_{0}^\infty   |\langle h,  \psi^s_{(x_1,t_1), (x_2,t_2)} \rangle|^2 \frac{  \d t_2}{ t_2}\right)^{\frac12},
\]
which satisfies \[
\|\mathrm{S}_{t_1} \|_{L^p(\R^{\mathbf d}; w)}\lesssim [w]_{A_p}^{\max\{\frac{1}{p-1},\frac12\}}.\] This can be seen by repeating the sparse domination  bound for the Christ-Journ\'e type square function e.g. of \cite{Brocchi20,LLP}, where the averages in the sparse operators are associated to rectangles with side of fixed length $t_1$ in the first parameter. The fact that the weight is a product weight ensures that the bound is uniform over all $t_1$. The weighted bound above upgrades immediately to vector-valued, and may be used in the second step below to yield
\[\begin{split}
\|O_sf\|_{L^p(w)} &= \|\mathrm{S}_{t_1} (   f*_1 \beta_{s+\log t_1}) \|_{L^p(w; L^2( \frac{\d t_1}{t_1}))}
  \lesssim [w]_{A_p}^{\max\{\frac{1}{p-1},\frac12\}} \|    f*_1 \beta_{s+\log t_1} \|_{L^p(w; L^2( \frac{\d t_1}{t_1}))}
\\&= [w]_{A_p}^{\max\{\frac{1}{p-1},\frac12\}} \|    f*_1 \beta_{\log t_1} \|_{L^p(w; L^2( \frac{\d t_1}{t_1}))} \lesssim [w]_{A_p}^{\max\{\frac{2}{p-1},1\}} \|   f \|_{L^p(w )}. 
\end{split}
\]
The very last inequality is obtained by using the straightforward weighted Littlewood-Paley square function bound of \cite{Brocchi20,LLP} in the first parameter  and Fubini's theorem. The proof of Theorem \ref{t:w2} is complete.

\bibliography{wrt}
\bibliographystyle{amsplain}
\end{document}